\documentclass[reqno,a4paper,11pt,oneside,final]{amsart}   
\usepackage{etex}
\usepackage[foot]{amsaddr} 

\usepackage{float}
\usepackage{mathrsfs}
\usepackage{amssymb,amsmath,amsthm,amsfonts}
\usepackage{mathtools}
\usepackage{fixmath}
\usepackage{bbm}
\usepackage{braket}
\usepackage{caption}
\usepackage{subcaption}
\usepackage{enumitem}
\usepackage{units}
\usepackage[final]{graphicx}
\usepackage{xcolor}
\usepackage{booktabs}
\usepackage[english]{babel}
\usepackage[T1]{fontenc}
\usepackage{lmodern}

\usepackage{tikz}
\usepackage{pgfplots}
\pgfplotsset{compat=1.17}

\usepackage{algorithm}
\usepackage{algorithmic}
\usepackage[final,stretch=10]{microtype}

\usepackage[colorlinks, citecolor=hanblue,linkcolor=red]{hyperref}
\definecolor{hanblue}{rgb}{0.27, 0.42, 0.81}
\usepackage{todonotes}

\DeclarePairedDelimiter{\norm}{\lVert}{\rVert}

\newcommand{\Om}{\Omega}

\DeclareMathOperator{\dom}{\operatorname{dom}}

\DeclareMathOperator{\supp}{supp}

\DeclareMathOperator*{\argmin}{arg\,min}
\DeclareMathOperator*{\argmax}{arg\,max}

\DeclareMathOperator*{\esssup}{ess\,sup}

 \newcommand{\nor}[1]{\left\| #1 \right\|}

\newcommand{\G}{\mathcal{G}}

\newcommand{\N}{\mathbb{N}}

\newcommand{\R}{\mathbb{R}}
\newcommand{\I}{\mathcal{I}}
\newcommand{\de}{\mathrm{d}}
\newcommand{\Ext}[1]{\operatorname{Ext}(#1)}

\newcommand{\M}{\mathcal{M}}
\newcommand{\K}{\mathcal{K}}
\newcommand{\Cc}{\mathcal{C}}
\newcommand{\Bb}{\mathcal{B}}

\newcommand{\optu}{\bar{u}}

\newcommand{\mnorm}[1]{\|#1\|_{\mathcal{M}}}

\newcommand{\cnorm}[1]{\|#1\|_{\mathcal{C}}}
\newcommand{\ynorm}[1]{\|#1\|_{Y}}
\newcommand{\tr}[1]{tr(C(1)^{-1})}

\newcommand{\eps}{\varepsilon}

\newcommand{\vertiii}[1]{{\left\vert\kern-0.25ex\left\vert\kern-0.25ex\left\vert #1
    \right\vert\kern-0.25ex\right\vert\kern-0.25ex\right\vert}}

\newcommand{\sm}[2]{\llangle #1,#2 \rrangle}

\newcommand{\e}{\varepsilon}

\newcommand{\f}{\varphi}
     
\newcommand{\zak}{%
  \mathbin{\vrule height 1.6ex depth 0pt width
0.13ex\vrule height 0.13ex depth 0pt width 1.3ex}
}    %

\newcommand{\weak}{\rightharpoonup}   %
\newcommand{\weakstar}{\stackrel{*}{\rightharpoonup}}  %

\DeclareMathOperator{\ext}{Ext} %

\DeclareRobustCommand{\rchi}{{\mathpalette\irchi\relax}}
\newcommand{\irchi}[2]{\raisebox{\depth}{$#1\chi$}}

  \let\div\relax
  \DeclareMathOperator*{\div}{div}

\makeatletter
\DeclareFontFamily{OMX}{MnSymbolE}{}
\DeclareSymbolFont{MnLargeSymbols}{OMX}{MnSymbolE}{m}{n}
\SetSymbolFont{MnLargeSymbols}{bold}{OMX}{MnSymbolE}{b}{n}
\DeclareFontShape{OMX}{MnSymbolE}{m}{n}{
    <-6>  MnSymbolE5
   <6-7>  MnSymbolE6
   <7-8>  MnSymbolE7
   <8-9>  MnSymbolE8
   <9-10> MnSymbolE9
  <10-12> MnSymbolE10
  <12->   MnSymbolE12
}{}
\DeclareFontShape{OMX}{MnSymbolE}{b}{n}{
    <-6>  MnSymbolE-Bold5
   <6-7>  MnSymbolE-Bold6
   <7-8>  MnSymbolE-Bold7
   <8-9>  MnSymbolE-Bold8
   <9-10> MnSymbolE-Bold9
  <10-12> MnSymbolE-Bold10
  <12->   MnSymbolE-Bold12
}{}
 
\let\llangle\@undefined
\let\rrangle\@undefined
\DeclareMathDelimiter{\llangle}{\mathopen}%
                     {MnLargeSymbols}{'164}{MnLargeSymbols}{'164}
\DeclareMathDelimiter{\rrangle}{\mathclose}%
                     {MnLargeSymbols}{'171}{MnLargeSymbols}{'171}
\makeatother

\makeatletter
\newcommand*\rel@kern[1]{\kern#1\dimexpr\macc@kerna}
\newcommand*\widebar[1]{%
  \begingroup
  \def\mathaccent##1##2{%
    \rel@kern{0.8}%
    \overline{\rel@kern{-0.8}\macc@nucleus\rel@kern{0.2}}%
    \rel@kern{-0.2}%
  }%
  \macc@depth\@ne
  \let\math@bgroup\@empty \let\math@egroup\macc@set@skewchar
  \mathsurround\z@ \frozen@everymath{\mathgroup\macc@group\relax}%
  \macc@set@skewchar\relax
  \let\mathaccentV\macc@nested@a
  \macc@nested@a\relax111{#1}%
  \endgroup
}
\makeatother
 
\numberwithin{equation}{section}
 
\definecolor{darkred}{rgb}{.7,0,0}

\definecolor{green}{rgb}{0,0.7,0}

\usepackage[notref,notcite]{showkeys}

\theoremstyle{plain}

\newtheorem{theorem}{Theorem}[section]
\newtheorem{lemma}[theorem]{Lemma}

\newtheorem{proposition}[theorem]{Proposition}

\theoremstyle{definition} 

\newtheorem{assumption}[theorem]{Assumption}
\newtheorem{definition}[theorem]{Definition}

\newtheorem{remark}[theorem]{Remark}

\begin{document}
\title[Fully-corrective GCG methods]{Asymptotic linear convergence of fully--corrective generalized conditional gradient methods}

\pagestyle{myheadings}

 \author[K. Bredies]{Kristian Bredies}
\author[M. Carioni]{Marcello Carioni} 
\author[S. Fanzon]{Silvio Fanzon}
\author[D. Walter]{Daniel Walter}

 \address[Kristian Bredies, Silvio Fanzon]{University of Graz, Institute of Mathematics and Scientific Computing, Heinrichstra\ss e 36, 8010 Graz, Austria}

 \address[Marcello Carioni]{Department of Applied Mathematics, University of Twente, 7500AE Enschede, The Netherlands}

\address[Daniel Walter]{Johann Radon Institute for Compuational and Applied Mathematics, Altenberger Stra\ss e 69, 4040 Linz}

\email[Kristian Bredies]{Kristian.Bredies@uni-graz.at}
\email[Marcello Carioni]{m.c.carioni@utwente.nl}
\email[Silvio Fanzon]{Silvio.Fanzon@uni-graz.at} 
 \email[Daniel Walter]{daniel.walter@oeaw.ac.at}

\begin{abstract}
\small{ %
We propose \emph{a fully-corrective generalized conditional gradient method} (FC-GCG) for the minimization of the sum of a smooth, convex loss function and a convex one-homogeneous regularizer over a Banach space. The algorithm relies on the mutual update of a finite set~$\mathcal{A}_k$ of extremal points of the unit ball of the regularizer and of an iterate~$u_k \in \operatorname{cone}(\mathcal{A}_k)$. Each iteration requires the solution of one linear problem to update~$\mathcal{A}_k$ and of one finite dimensional convex minimization problem to update the iterate. Under standard hypotheses on the minimization problem we show that the algorithm converges sublinearly to a solution. Subsequently, imposing additional assumptions on the associated dual variables, this is improved to a linear rate of convergence. The proof of both results relies on two key observations: First, we prove the equivalence of the considered problem to the minimization of a lifted functional over a particular space of Radon measures using Choquet's theorem. Second, the FC-GCG algorithm is connected to a \emph{Primal-Dual-Active-point Method} (PDAP) on the lifted problem for which we finally derive the desired convergence rates.

 

}
\end{abstract}
 
\maketitle
\section{Introduction} \label{sec:intro}
This paper is concerned with the analysis of an efficient solution algorithm for minimization problems in composite form
\begin{align} \label{def:minprob}
\inf_{u\in \M} J(u),\quad J(u) \coloneqq  F(Ku)+ \G(u)  \tag{$\mathcal{P}_{\M}$}
\end{align}
over a Banach space~$\M$. Here,~the forward operator $K$ maps continuously from~$\M$ into a Hilbert space~$Y$ of observations, not necessarily finite dimensional, and~$F$ denotes a smooth convex loss function. The second part of the objective functional is constituted by a convex but possibly non-smooth functional~$\mathcal{G}$ which promotes desired structural properties. We refer, e.g., to the sparsifying property of total variation regularization or the staircasing effect of bounded variation penalties. The observation that certain structural features of minimizers can be brought forth by a suitable choice of the functional~$\G$ and of the space~$\M$, has made the analysis~of problems of the form \eqref{def:minprob} a flourishing topic in the context of optimal control, inverse problems, compressed sensing and machine learning. As a consequence, the interest for such models has also sparked the demand for efficient solution algorithms. Since many of the structural features of~\eqref{def:minprob} are tightly linked to properties of the underlying, possibly infinite dimensional Banach space, a particular focus in this context lies on \textit{function space methods}, i.e., algorithms solving~\eqref{def:minprob} without discretizing~$\M$. This is a challenging task for a variety of reasons: On the one hand, it requires algorithms that can handle the non-smoothness of the objective functional~$J$. On the other hand, it forces to work directly on the Banach space~$\M$ which usually lacks ``nice'' properties such as reflexivity or uniform convexity. Efficient algorithms have been developed, for instance, for inverse problems regularized with $\ell^p$ penalties \cite{iteratedshrinkage, daubechies}, inverse problems in the space of measures regularized with the total variation \cite{boyd, pikka, Denoyelle_2019} and dynamic inverse problems with optimal transport regularizers \cite{DGCG, bredies2020optimal, superposition, extremal}. 

\subsection{Contribution \& related work}
In many interesting applications it is meaningful to assume that~$\M$ is given as~the topological dual of a separable Banach space~$\Cc$, and~$K$ is the adjoint of a ``predual'' operator~$K_* \colon Y \to \Cc$,~$(K_*)^*=K$. We refer to Section~\ref{sec:examples} for a few examples. In this case a simple approach to computing a minimizer to~\eqref{def:minprob} is constituted by~\textit{generalized conditional gradient} (GCG) algorithms. Assuming that the solution set of~\eqref{def:minprob} is bounded by some~$M>0$, this method updates the iterate~$u_k$ by computing the dual variable~$p_k=-K^*\nabla F(Ku_k)$ and setting
\begin{align*}
v_k \in \argmin_{\mnorm{v}\leq M} \left \lbrack -\langle  p_k, v \rangle+ \G(v)\right \rbrack,~
u_{k+1}=u_k+ s_k (v_k-u_k) \,,
\end{align*}
where~$s_k \in (0,1)$ is chosen according to some stepsize rule. If~$\mathcal{G}(u)=I_{A}(u)$ is the indicator function of a compact convex set~$A \subset \M$, then the iteration scheme reduces to the \textit{Frank-Wolfe} (FW) algorithm for constrained minimization~\cite{frank,demyanov,dunnopen}.

In the present paper we focus on convex positively homogeneous functionals~$\G$ with compact sublevel sets, where compactness is intended with respect to the weak* topology induced by $\Cc$. The compactness assumption is equivalent to
$\G$ having weak* closed and norm bounded sublevels. In particular, these assumptions on $\G$ encompass the important case of norm regularization~$\mathcal{G}(u)=\mnorm{u}$, and also allow for seminorm penalties if the problem is posed in a suitable quotient space~\cite{bc}. More in general, our setting covers the case of  all~\textit{gauge functions} of the form
 \[
 \kappa_\mathcal{A}(u):=\inf \left\{ \rho \geq 0 \; | \; u \in \rho A     \right\}  \quad \text{where } \,\, \rho A = \{\rho v \; | \; v \in A\}\,,  
 \]
 and~$\mathcal{A} \subset \mathcal{M}$ is a weak* compact and convex set. This type of functional is of particular relevance in machine learning applications \cite{chandrasekaran,yu}.

 In this general setting, GCG methods benefit from several desirable properties. For example, they only rely on the repeated solution of (partially) linearized problems that, in some interesting cases, can be solved analytically. Additionally, the descent direction can be chosen to satisfy~$v_k= M_k \hat{v}_k$ where~$M_k \geq 0$ is a scaling factor and~$\widehat{v}_k$ is an~\textit{extremal point} of the ``unit ball''~$B=\{\,u\; |\;\mathcal{G}(u) \leq 1\,\}$ of the regularizer. We denote by~$\operatorname{Ext}(B)$ the set of such points. Thus, the above mentioned partially linearized problems can be solved, equivalently, in the set~$\operatorname{Ext}(B)$, reducing considerably the complexity of each iteration of the algorithm. As a further consequence, initializing the algorithm by $u_0=0$ and selecting descent directions~$v_k$ that are extremal points of $B$, the iterate~$u_k$ exhibits \textit{k-sparsity}, i.e., it is contained in the conic hull of at most~$k$ points in $\operatorname{Ext}(B)$. The connection between the structure enhancing properties of a functional~$\mathcal{G}$ and the set of extremal points of~$B$ has been recently studied in~\cite{bc} and~\cite{boyer}. In this context, a central result is given by~\textit{convex representer theorems}. Loosely speaking, these state that problems of the form~\eqref{def:minprob} admit solutions~$\bar{u}$ which are contained in the conic hull of at most~$\operatorname{dim} Y$ extremal points.

While the FW method~\cite{dunnnonsingular,dunnimplicit,freund, jaggiwolfe}, and its many variants such as \textit{away-step} FW~\cite{guelat,lacosteaway}, or \textit{fully-corrective} FW~\cite{holloway, hearn, hohenbalken}, have received a lot of attention, GCG algorithms for general non-smooth functionals~$\mathcal{G}$ are less frequently studied. We refer, e.g., to~\cite{DGCG, brediesgeneral, sunsafe, yu} as well as~\cite[Chapter 6]{daniel_thesis} which all prove a global sublinear~$\mathcal{O}(1/k)$ rate of convergence of~$J(u_k)$ towards the minimum value. Note that this rate is known to be optimal~\cite{canon}. Moreover, the absence of steps that remove extremal points from~$u_k$  often leads to clustering phenomena in practice. Thus, despite its various advantages, these shortcomings of GCG methods limit their practical utility.

The main contribution of the present work is the analysis of a~\textit{fully-corrective generalized conditional gradient method} (FC-GCG) for~\eqref{def:minprob}. This relies on the mutual update of a sequence of sparse iterates~$u_k$ and of a sequence of finite, ordered~\textit{active sets}
\[
\mathcal{A}^u_k= \left\{\,u^k_i \in \operatorname{Ext}(B) \;|\;i=1,\dots,N_k\,\right\}\,,  \,\, N_k \in \N\,,
\]
of extremal points, which constitute the \textit{atoms} representing the sparse iterate $u_k$. Given the current iterate~$u_k$ and active set~$\mathcal{A}_k$, the proposed method first enlarges~$\mathcal{A}_k$ by setting~$N^+_k=N_k +1$ and
\begin{align} \label{eq:updateAkintro}
\mathcal{A}^{u,+}_k= \mathcal{A}^u_k \cup \{\widehat{v}_k\}\,,\quad~\widehat{v}_k \in \argmax_{v \in \operatorname{Ext}(B)}\, \langle p_k, v \rangle\,.
\end{align}
Subsequently, the new iterate iterate~$u_{k+1}$ is found by solving the subproblem
\begin{align} \label{eq:subprobintro}
\min_{u \in \operatorname{cone}(\mathcal{A}^{u,+}_k)} \left \lbrack F(Ku) +\kappa_{\mathcal{A}^{u,+}_k}(u) \right \rbrack,
\end{align}
where the minimization occurs over the cone spanned by~$\mathcal{A}^{u,+}_k$, and~$\G$ is replaced by the gauge function associated with~$\mathcal{A}^{u,+}_k$, which in this case \cite{bonsall, chandrasekaran} simplifies to
\begin{align*}
\kappa_{\mathcal{A}^{u,+}_k}(u)= \min_{\lambda \in \R_+^{N^+_k}}\left.\left\{\,\sum^{N^+_k}_{i=1} \lambda_i\;\right\vert\;u= \lambda_{N^+_k } \widehat{v}_k+\sum^{N_k}_{i=1} \lambda_i u^k_i\,\right\}.
\end{align*}
Finally, $\mathcal{A}^{u,+}_k$ is pruned by removing the extremal points for which the weight is set to zero by \eqref{eq:subprobintro}, obtaining the next active set~$\mathcal{A}^u_{k+1}$. 

As we will see, the proposed method combines the advantages of GCG methods, e.g., its global convergence, with an improved convergence behavior and sparser iterates. 
More in detail, we first prove a global sublinear convergence rate~$\mathcal{O}(1/k)$ for~$J(u_k)$ towards the minimum of~\eqref{def:minprob}, see Theorem~\ref{thm:convofcondu}. This is achieved
under mild assumptions, see~\ref{ass:func1}-\ref{ass:func3} discussed in Section \ref{sec:equivalence}. These are quite standard and are, in particular, sufficient for the well-posedness of~\eqref{def:minprob}, as shown in Proposition \ref{prop:existu}.
Subsequently, in Theorem~\ref{thm:fastconvpreview},
we show that~$J(u_k)$ converges asymptotically at a linear rate of~$\mathcal{O}(\zeta^k)$,~$\zeta \in(0,1)$, provided that the optimal dual variable and the loss in~\eqref{def:minprob} meet certain structural requirements, see Assumptions~\ref{ass:fast:F1}-\ref{ass:fast:F5} in Section \ref{subsec:nondegen}. Specifically, these assumptions imply that the minimizer to~\eqref{def:minprob} is unique and sparse, i.e., of the form~$\bar{u}=\sum^N_{i=1} \bar{\lambda}_i \bar{u}_i$ for a finite number of extremal points~$\bar{u}_i \in \operatorname{Ext}(B)$ and coefficients~$\bar{\lambda}_i >0$. More crucially, we also require the existence of a ``distance function''~$g$ defined on $\operatorname{Ext}(B)$ such that the forward operator~$K$ and the optimal dual variable~$\bar{p}=-K_* \nabla F(K\bar{u})$ associated to $\bar{u}$ satisfy, respectively, the following Lipschitz and quadratic growth conditions
\begin{equation} \label{eq:intro_B5}
\ynorm{K(u -\bar{u}_i)} \lesssim  g(u,\bar{u}_i), \qquad 1-\langle \bar{p}, u  \rangle \gtrsim  g(u,\bar{u}_i)^2 \,,
\end{equation}
for all~$i=1,\dots,N$, and extremal points $u$ in the weak* vicinity of $\bar{u}_i$, respectively.

The proofs of the convergence results of Theorems~\ref{thm:convofcondu},~\ref{thm:fastconvpreview} rely on a novel lifting strategy, discussed in Section~\ref{sec:lifting}, which allows us to
recast \eqref{def:minprob} into an equivalent minimization
problem on $M^+(\mathcal{B})$, the space of positive Radon measures on the set $\mathcal{B} := \overline{\operatorname{Ext}(B)}^*$. 
Specifically, this is achieved by making three key 
observations:
First, using Choquet's Theorem~\cite[Page 14]{phelps}, we argue the existence of a surjective mapping~$\mathcal{I}$ from $M^+(\mathcal{B})$ to the domain of $\mathcal{G}$ with
\begin{align*}
\langle p, \mathcal{I}(\mu) \rangle= \int_{\mathcal{B}} \langle p,v \rangle~\mathrm{d} \mu(v)\,, \,\, \text{ for all } \, p \in \Cc,~\mu \in M^+(\mathcal{B})\,.
\end{align*}
Second, we consider the auxiliary problem
\begin{align}
\label{def:sparseprob} \tag{$\mathcal{P}_{M^+}$}
\inf_{\mu \in M^+(\mathcal{B})} \left \lbrack F(\mathcal{K}\mu) +\|\mu\|_{M(\mathcal{B})} \right \rbrack\,,
\end{align}
where the forward operator~$K$ is replaced by a lifted model~$\mathcal{K}\colon M^+(\mathcal{B}) \to Y$ satisfying~$\mathcal{K}\mu=K \mathcal{I}(\mu)$ for all~$\mu \in M^+(\mathcal{B})$.
As it turns out, see Theorem~\ref{prop:equivalence}, the above minimization problems are equivalent, in the sense that~$\bar{u}= \mathcal{I}(\bar{\mu})$ solves~\eqref{def:minprob} if and only if~$\bar{\mu}$ minimizes in~\eqref{def:sparseprob}. 
As a third key ingredient, such equivalence allows to interpret the iteration scheme introduced in~\eqref{eq:updateAkintro}--\eqref{eq:subprobintro} as one step of an~\textit{exchange algorithm} or a \textit{Primal-Dual-Active-Point method} (PDAP), cf. Algorithm~\ref{alg:pdap}, applied to~\eqref{def:sparseprob}. 
These methods are fully-corrective variants of the generalized conditional gradient method for minimization problems over spaces of Borel measures introduced in~\cite{pikka}, and their linear convergence has been recently proven in the finite-dimensional case, i.e., for measures supported on a compact subset of the euclidean space, see~\cite{flinthlinear,pieper}. 
Combining these three key observations, and carefully extending the techniques and results of~\cite{pieper} to the present setting, 
we are able to conclude the asymptotic linear convergence of our algorithm. 

If we interpret~\eqref{def:minprob} as a \textit{sparse dictionary learning problem}, in which the dictionary is given by~$\operatorname{Ext}(B)$, our FC-GCG method can also be linked to a~\textit{fully-corrective greedy selection} method,~\cite{shalev}, or an~\textit{accelerated gradient boosting} algorithm,~\cite{wang}. In this context, a similar lifting approach was implicitly used in~\cite{matrixnormzhang} to derive a fully-corrective GCG method for matrix-regularization problems which eventually converges at a global linear rate, albeit under very restrictive conditions. More in detail, expressed in the notation of the present paper, the authors require the strong convexity of~$F \circ K$, as well as that~$\Ext{B}$ is a finite set. For a related result in settings with finitely many atoms we also point out~\cite{shalev}. In contrast, the present manuscript is geared towards potentially compact forward operators~$K$, i.e.~$F \circ K$ is not strongly convex, as well as regularizers whose sublevel sets admit infinitely many extremal points. These additional difficulties will be circumvented by a careful localization of~$\mathcal{A}^{u,+}_k$ around the optimal extremal points, as well as by exploiting the assumed quadratic growth behavior in~\eqref{eq:intro_B5}. While the main contribution of the present work is clearly given by the FC-GCG method, the lifting strategy of Section~\ref{sec:lifting} may be of independent interest, and could also be applicable beyond the analysis of efficient solution algorithms.

In this paper, we mainly focus on the theoretical aspects of the FC-GCG method and, in particular, its convergence. Naturally, from the practical standpoint, the applicability of our method relies on the knowledge 
of the set $\operatorname{Ext}(B)$, as well as on the efficient solution of the constrained linear problem in~\eqref{eq:updateAkintro}. Especially the computational cost of the latter greatly varies between different instances of Problem~\eqref{def:minprob}.
To illustrate the working of our method and the role of assumptions \ref{ass:fast:F1}-\ref{ass:fast:F5} we present in Section \ref{sec:examples} several examples of applications where our algorithm can be implemented. For the first three examples we provide a natural and easy to verify set of assumptions that imply \ref{ass:fast:F1}--\ref{ass:fast:F5}, and for all of them we discuss the computational burden of computing a solution for the constrained linear problem in~\eqref{eq:updateAkintro}. We first consider the problem of identifying the initial source of a heat equation from given temperature measurements, see Section~\ref{subsec:guideex}.
For this example we also demonstrate numerically the expected linear convergence of Algorithm \ref{alg:gcg}, discussing the stopping criterion and the advantages of our algorithm compared to classical GCG methods. Then, in Section~\ref{sec:rank-one}, we 
consider the trace regularization of linear operators from a Hilbert space into itself, which favours the reconstruction of rank-one operators. In Section~\ref{subsec:mineffort}, we discuss minimum effort problems, where the regularization enforced by the supremum norm favours binary solutions.
Finally, in Section \ref{sec:dynamic_inverse}, we briefly deal with the optimal transport regularization of dynamic inverse problems \cite{bredies2020optimal}, showing that the algorithm introduced in \cite{DGCG} is, in some instances, a particular case of Algorithm~\ref{alg:gcg}. %
For this example the verification of the hypotheses \ref{ass:fast:F1}-\ref{ass:fast:F5}, necessary to ensure fast convergence, is non-trivial and is left to future work.

\subsection{Outline}
The paper is structured as follows. In Section \ref{sec:equivalence} we formulate the minimization problem \eqref{def:minprob} we are interested to solve. We further 
make a basic set of Assumptions~\ref{ass:func1}-\ref{ass:func3}, under which we 
prove its well-posedness, see Proposition~\ref{prop:existu}. In Section~\ref{sec:condgrad} we introduce the FC-GCG algorithm, discussing well-posedness of its steps and providing a stopping criterion. In addition we introduce sufficient conditions for fast 
convergence \ref{ass:fast:F1}-\ref{ass:fast:F5}, and state the sublinear global and linear local convergence results in Theorems \ref{thm:convofcondu}, \ref{thm:fastconvpreview} respectively. These results, together with the assumptions, are discussed for specific examples in Section~\ref{sec:examples}. We then pass to the convergence proofs for the FC-GCG algorithm. These are broken into three parts:
First, in Section \ref{sec:lifting}, we introduce a ``lifting'' of the minimization problem \eqref{def:minprob} to the space of Radon measures $M^+(\mathcal{B})$ by virtue of Choquet's theorem. We then prove the equivalence of \eqref{def:minprob} to the auxiliary problem \eqref{def:sparseprob}, see Theorem \ref{prop:equivalence}. Second, we propose an extension of the PDAP algorithm to compute solutions to \eqref{def:sparseprob} in Section~\ref{subsec:PDAP}. It turns out that PDAP and FC-GCG are equivalent given the correct interpretation, see Theorem \ref{thm:eqofpdapandgcg}. This equivalence will be extensively used in Section \ref{sec:convergence} to carry out the proofs of the main convergence statements. Finally, the Appendix contains some auxiliary results, as well as of proofs omitted in the main body of the paper.

\section{The minimization problem} \label{sec:equivalence}
In this section we introduce the minimization problem we are concerned with solving, and prove its well-posedness under suitable assumptions. We start by establishing some notations. Throughout the paper $\Cc$ denotes a separable Banach space with norm~$\cnorm{\cdot}$ and topological dual space~$\M \simeq \Cc^*$. We denote the duality pairing between $p \in \Cc$ and $u \in \M$ by~$\langle p , u \rangle$. The space~$\M$ is equipped with the canonical dual norm
\begin{align*}
\mnorm{u} := \sup_{\cnorm{p}\leq 1} \langle p,u  \rangle \quad \text{ for all } \,\, u \in \M \,.
\end{align*}
Let~$\G \colon \M \to [0,\infty]$ be a convex, weak* lower semi-continuous and positively one-homogeneous functional, that is $\G(\lambda u)=\lambda \G(u)$ for all $\lambda \geq 0$. Let $K\colon \M \to Y$ be a linear weak*-to-weak continuous operator, mapping to a given Hilbert space~$Y$, and~$F \colon Y \to \R$ be a convex mapping. %
The inner product and induced norm on~$Y$ will be denoted by~$(\cdot,\cdot)_Y$ and~$\|\cdot\|_Y$, respectively. Our interest lies in efficient solution algorithms for problems of the form~\eqref{def:minprob}, which we remind the reader is of the form
\[
\inf_{u\in \M} J(u), \quad J(u):=  F(Ku)+ \G(u) \,.
\]

\begin{remark}
Notice that the weak*-to-weak continuity of the linear operator $K : \mathcal{M} \rightarrow Y$ implies existence of a linear continuous operator $K_* : Y \rightarrow \mathcal{C}$ that is the pre-adjoint of $K$, i.e.,
\begin{align*}
\langle K_* y,u \rangle =(Ku, y)_Y  \quad \text{ for all } \,\,y \in Y,~u \in \M\,.
\end{align*}
See, for example, \cite[Remark 3.2]{pikka}. Moreover, the existence of a continuous pre-adjoint $K_*$ implies the strong-to-strong continuity of the operator $K$.
\end{remark}

Throughout the paper we moreover require the following basic assumptions.
 
\begin{assumption} \label{ass:functions}
Assume that:
\begin{enumerate}[label=\textnormal{(A\arabic*)}]
\item \label{ass:func1} $F$ is bounded from below, strictly convex and Fr\'echet differentiable on~$Y$. Its gradient $\nabla F \colon Y \to Y$ is Lipschitz continuous on compact sets.
\item \label{ass:func2} The sublevel set
\begin{align*}
S^-(\G,\alpha):= \left\{\,u \in \M\;|\; \G(u)\leq \alpha\, \right\}
\end{align*}
is weak* compact for every~$\alpha \geq 0$.
\item \label{ass:func3} The operator $K : \mathcal{M} \rightarrow Y$ is sequentially weak*-to-strong continuous in 
\[
{\rm dom}(\G) := \{ u \in \M \;|\;\G(u) <\infty\,\} \,,
\]
namely, for every sequence $\{u_k\}_k $ in ${\rm dom}(\G)$ such that $u_k  \weakstar u$ for some~$u \in \M$, it holds that $Ku_k \to Ku$ in $Y$.
\end{enumerate}
\end{assumption}
Note that Assumption~\ref{ass:func2} is equivalent to ask that the sublevel set $S^-(\G,\alpha)$ is closed and norm bounded for every $\alpha \geq 0$. 
The above conditions guarantee the existence of minimizers to \eqref{def:minprob}.

\begin{proposition} \label{prop:existu}
Assume~\ref{ass:func1}-\ref{ass:func3}. Then there exists at least one minimizer to~\eqref{def:minprob}. Moreover~$\optu \in \M$ is a solution to~\eqref{def:minprob} if and only if~$\bar{p}=-K_* \nabla F(K\optu)\in \Cc$ satisfies
\begin{align} \label{eq:subdiff}
\langle \bar{p}, \bar{u} \rangle= \G(\optu)\,, \quad ~\max_{v \in S^-(\mathcal{G}, 1)} \langle \bar{p},v \rangle \leq 1\,.
\end{align}
Finally, if~$\optu_1, \optu_2 \in \M$ are two solutions to~\eqref{def:minprob}, then~$K\optu_1=K\optu_2$.
\end{proposition}

\begin{proof}
The existence of a minimizer follows by the direct method of calculus of variations. Indeed, the sublevel sets of $J$ are weak* compact thanks to Assumption~\ref{ass:func1} and Assumption~\ref{ass:func2}. Moreover $J$ is weak* lower semicontinuous, since $\mathcal{G}$ is weak* lower semicontinuous, $K$ is weak*-to-weak continuous and $F$ is convex and continuous (Assumption~\ref{ass:func1}).
In order to show \eqref{eq:subdiff}, notice
that the function~$f(u)=F(Ku)$ is G\^{a}teaux-differentiable with
$f'(u)(v)= \langle K_* \nabla F(Ku), v \rangle$ for all $v \in \mathcal{M}$.
Hence, see e.g.~\cite[Proposition~6.3]{daniel_thesis},~$\optu \in \mathcal{M}$ is a solution to~\eqref{def:minprob} if and only if 
\begin{align*}
\langle \bar{p},u- \bar{u} \rangle+ \mathcal{G}(\bar{u}) \leq \mathcal{G}(u) \,\, \text{ for all } \, u \in \M.
\end{align*}
This variational inequality holds if and only if
\begin{align*} 
\langle \bar{p}, \bar{u} \rangle= \G(\optu), \quad \langle \bar{p},v \rangle \leq \G(v) \,\, \text{ for all } \, v \in \M \,,
\end{align*}
which is equivalent to~\eqref{eq:subdiff}, thanks to the one-homogeneity of $\G$. Last, the identity $K\optu_1=K\optu_2$ for two solutions~$\optu_1,\optu_2$ of~\eqref{def:minprob} follows from the strict convexity of~$F$.
\end{proof}
For the remainder of the paper we refer to
\begin{equation} \label{eq:unique_dual}
\bar{y} \coloneqq K\bar{u} \in Y\,, \quad \bar{p}:=-K_* \nabla F(K\optu)\in \Cc \,,
\end{equation}
as the (unique) optimal observation and dual variable for~\eqref{def:minprob}, respectively, where $\bar{u}$ is any minimizer to \eqref{def:minprob}. Set $B = S^-(\mathcal{G}, 1)$. In the following we further require the notion of an~\textit{extremal point} of~$B$.
\begin{definition}
An element~$u \in B$ is called an extremal point of~$B$ if there are~\textit{no}~$u_1,u_2 \in B$ and~$s\in(0,1)$ with~$u=(1-s)u_1+su_2$. The set of all extremal points of $B$ is denoted by~$\Ext B$.
\end{definition}
Since~$B$ is weak* compact, non empty and convex thanks to Assumptions~\ref{ass:func1}-\ref{ass:func3}, by the Krein-Milman Theorem we infer that $ \operatorname{Ext}(B) \neq \emptyset$ and for every~$u \in \dom(\G)$ there is~$\{u_k\}_k$ in $\operatorname{cone}(\operatorname{Ext}(B))$  with~$u_k \weakstar u$. Set~$\mathcal{B}:=\overline{\Ext{B}}^*$. Since the predual space~$\Cc$ is separable and~$\mathcal{B}$ is weak* compact, there exists a metric~$d_{\mathcal{B}}$ metrizing the weak* convergence on $\Bb$, that is, for all sequences $\{u_k\}_k$ in $\Bb$ and $u \in \Bb$ we have
\begin{equation} \label{eq:metric_db}
u_k \weakstar u \,\, \text{ if and only if } \,\, \lim_{k \rightarrow \infty} d_{\mathcal{B}}(u_k,u)=0 \,.
\end{equation}
In particular, we have that $(\mathcal{B},d_{\mathcal{B}})$ is a compact separable metric space. 
\begin{remark}
Let us mention that all the results in this paper still hold  under slightly weaker conditions on the functional~$F$. In particular, the strict convexity of~$F$ can be replaced by assuming~$Ku_1=Ku_2$ for all solutions~$u_1,u_2$ of~$\eqref{def:minprob}$. Moreover, all results also apply to convex, weakly lower semicontinuous functionals~$F\colon Y \to \R \cup \{+\infty\}$ for which~$\dom(F)$ is open. In this case,~$F$ is required to be smooth on~$\dom(F)$ and~$\nabla F \colon \dom (F) \to Y$ is assumed to be Lipschitz continuous on compact subsets of~$\dom(F)$. For more details, we refer the interested reader to \cite{pieper,daniel_thesis}.
\end{remark}

\section{A numerical minimization algorithm} \label{sec:condgrad}

This section concerns the development of an implementable and efficient solution algorithm for the minimization problem \eqref{def:minprob}, namely, the~\textit{fully-corrective generalized conditional gradient method} (FC-GCG). As anticipated in the introduction, FC-GCG
comprises the two basic steps at~\eqref{eq:updateAkintro} and~\eqref{eq:subprobintro}, which we 
now describe in more detail.
For this purpose, recall that every~$u \in \dom(\mathcal{G})$ can be approximated, in the weak* topology of $\M$,  by a sequence in~$\operatorname{cone}(\operatorname{Ext}(B))$ thanks to Krein-Milman's Theorem. The considered method exploits this observation by alternating between the update of a finite set
\begin{align*}
\mathcal{A}^u_k= \{u^k_i\}^{N_k}_{i=1} \subset  \operatorname{Ext}(B),
\end{align*}
the so-called~\textit{active set}, and of an iterate
\begin{align*}
u_k \in \operatorname{cone}(\mathcal{A}^u_k)= \left\{\,u=\sum^{N_k}_{i=1} \lambda^k_i u^k_i\;|\;\lambda^k \in \R^{N_k}_+\,\right\}\,,
\end{align*}
where~$\R_+=[0,+\infty)$. Given the current active set~$\mathcal{A}^u_k=\{u^k_i\}^{N_k}_{i=1}$ with~$u^k_i \in \operatorname{Ext}(B)$ and the iterate~$u_k \in \operatorname{cone}(\mathcal{A}^u_k)$, we compute the dual variable~$p_k= -K_* \nabla F(Ku_k) \in \Cc$ and update
\begin{align}\label{eq:descr_insertion}
\mathcal{A}^{u,+}_k= \mathcal{A}^u_k \cup \{\widehat{v}^u_k\}, \quad ~\widehat{v}^u_k \in \argmax_{v \in \operatorname{Ext}(B)} \,\langle p_k, v \rangle\,.
\end{align}
The above step requires the minimization of a linear functional over the not necessarily compact set of extremal points. 
Remarkably, this problem is well-posed, as shown in Lemma \ref{lem:existoflinearized} in Appendix~\ref{append:A1}.
Setting~$N^+_k=N_k +1$ and~$u^{k}_{N^+_k}=\widehat{v}^u_k$, we introduce the gauge function associated with~$\mathcal{A}^{u,+}_k$ as 
\begin{align*}
\kappa_{\mathcal{A}^{u,+}_k}(u) := \min_{\lambda \in \R_+^{N^+_k}}\left\{ \left.\,\sum^{N^+_k}_{i=1} \lambda_i\;\right\vert\;u= \sum^{N^+_k}_{i=1} \lambda_i u_i^k \, \right\}.
\end{align*} 
Note that~$\kappa_{\mathcal{A}^{u,+}_k}(u)$ is well-defined on~$\operatorname{cone}(\mathcal{A}^{u,+}_k)$, as the minimum is 
achieved due to lower semicontinuity of the~$\ell^1$ norm and closedness of~$\R_+$.
Subsequently, the next iterate~$u_{k+1}$ is found by solving the subproblem
\begin{align} \label{eq:subprobalgo}
\min_{u \in \operatorname{cone}(\mathcal{A}^{u,+}_k)} \left \lbrack F(Ku) +\kappa_{\mathcal{A}^{u,+}_k}(u) \right \rbrack,
\end{align}
where the search for the minimizer is restricted to the cone spanned by~$\mathcal{A}^{u,+}_k$ and the possibly complicated regularizer~$\G$ is replaced by the easier-to-handle gauge function of the finite set $\mathcal{A}^{u,+}_k$.  We point out that the objective functional in~\eqref{eq:subprobalgo} constitutes an upper bound on~$J$, i.e., there holds
\begin{align*}
J(u) \leq F(Ku) +\kappa_{\mathcal{A}^{u,+}_k}(u) \quad \text{for all}~u \in \operatorname{cone}(\mathcal{A}^{u,+}_k)
\end{align*}
due to the convexity and one-homogeneity of~$\mathcal{G}$. Apart from that, Problems~\eqref{def:minprob} and~\eqref{eq:subprobalgo} share the same basic structure of minimizing the sum of a smooth fidelity term and a gauge-like regularizer. In particular, the existence result and the necessary first-order conditions in Proposition~\ref{prop:existu} also apply to~\eqref{eq:subprobalgo}. However, as shown in Proposition \ref{prop:eqgauge} in Appendix~\ref{append:A1}, the big advantage of~\eqref{eq:subprobalgo} is that a solution can be computed as 
\begin{equation} \label{eq:new_iterate}
u_{k+1}= \sum^{N^+_k}_{i=1}\lambda^{k+1}_i u^k_i \,,
\end{equation}
where $\lambda^{k+1} \in \R_+^{N^+_k}$ solves the following finite dimensional optimization problem:
\begin{align} \label{eq:subprobcoeffs}
\min_{ \lambda \in \R_+^{N^+_k}} \left \lbrack F\left(\sum^{N^+_k}_{i=1} \lambda_i K u^k_i\right)+ \sum^{N^+_k}_{i=1} \lambda_i \right \rbrack\,.
\end{align}
Thus, in practice, we first determine a minimizer~$\lambda^{k+1} \in \R^{N^+_k}$ to~\eqref{eq:subprobcoeffs}. We remark that~\eqref{eq:subprobcoeffs} constitutes a finite dimensional non-smooth convex minimization problem, which can be efficiently solved by proximal methods or generalized Newton algorithms provided that~$F$ is sufficiently smooth.
Once this is accomplished, the new iterate~$u_{k+1}$ is defined according to~\eqref{eq:new_iterate}. As a final step, we truncate the active set~$\mathcal{A}^{u,+}_k$ by removing all extremal points that were assigned a zero weight by the optimization 
procedure~\eqref{eq:subprobcoeffs}, i.e., we set
\begin{align*}
\mathcal{A}^u_{k+1}= \mathcal{A}^{u,+}_k \setminus \{\,u^k_i\;|\;\lambda^{k+1}_i=0,~i=1,\dots, N^+_k\,\}
\end{align*}
and increment~$k$ by one. The method is summarized in Algorithm~\ref{alg:gcg} below.

\begin{algorithm}[tbh]\caption{FC-GCG for~\eqref{def:minprob}}
\begin{algorithmic} \label{alg:gcg}
\STATE 1. Let~$u_0= \sum^{N_0}_{i=1} \lambda^0_i u^0_i$,~$\lambda^0_i >0$,~$\mathcal{A}^u_0= \{u^0_i\}^{N_0}_{i=1}\subset \Ext B$.
\FOR{$k=0,1,2,\dots$}
\STATE 2. Given~$\mathcal{A}^u_k=\{u^k_i\}^{N_k}_{i=1} \subset \operatorname{Ext}(B) $ and~$u_k \in \operatorname{cone}(\mathcal{A}^u_k)$, calculate $p_k$ and $\widehat{v}^u_k$ with
\begin{align*}
p_k=-K_*\nabla F(Ku_k) \,, \quad \widehat{v}^u_k \in \argmax_{v\in \operatorname{Ext}(B)}\,\langle p_k,v\rangle \,.
\end{align*}
\IF{$\langle p_k,\widehat{v}^u_k \rangle\leq 1~\text{and}~k \geq 1$
}
\STATE 3. Terminate with~$\optu= u_k$ a minimizer to~\eqref{def:minprob}.
\ENDIF
\STATE 4. Update~$N^+_k = N_k +1$,~$u^k_{N^+_k} =\widehat{v}^u_k$ and~$\mathcal{A}^{u,+}_k= \mathcal{A}^u_k \cup \{\widehat{v}^u_k\}$.
\STATE 5. Determine~$\lambda^{k+1}$ with
\[
\lambda^{k+1} \in \argmin_{ \lambda \in \R_+^{N^+_k}} \left \lbrack F\left(\sum^{N^+_k}_{i=1} \lambda_i K u^k_i\right)+ \sum^{N^+_k}_{i=1} \lambda_i \right \rbrack\,,
\]
and set~$u_{k+1}= \sum^{N^+_k}_{i=1} \lambda^{k+1}_i u^k_i$.
\STATE 6. Update
\begin{align*}
\mathcal{A}^u_{k+1}= \mathcal{A}^{u,+}_k \setminus \{\,u^k_i\;|\;\lambda^{k+1}_i=0,~i=1,\dots, N^+_k\,\} 
\end{align*}
and set~$N_{k+1}=\# \mathcal{A}^u_{k+1}$.
\ENDFOR
\end{algorithmic}
\end{algorithm}

\subsection{Stopping condition}
Note that Algorithm~\ref{alg:gcg} terminates at iteration~$k\geq 1$ if 
\begin{equation} \label{eq:stopping_rev}
\max_{v \in \Ext B} \langle p_k, v \rangle \leq 1 \,.
\end{equation}
This is justified by the fact that, in this case, the current iterate $u^k$ is a minimizer to~\eqref{def:minprob}, as shown
in Proposition~\ref{prop:stopping} below. In particular, in this situation, the algorithm converges in a finite number
of iterations.

\begin{proposition}\label{prop:stopping}
Suppose~\ref{ass:func1}-\ref{ass:func3} in Assumption~\ref{ass:functions} hold. 
Let~$u_k \in \operatorname{cone}(\mathcal{A}^u_k)$ be generated by Algorithm~\ref{alg:gcg}, with~$k\geq 1$. Set~$p_k =-K_* \nabla F(Ku_k) $ and assume \eqref{eq:stopping_rev}. Then~$u_k$ solves~\eqref{def:minprob}.
\end{proposition}

\begin{proof}
Thanks to Proposition \ref{prop:existu} and Lemma \ref{lem:existoflinearized},~$u_k$ solves~\eqref{def:minprob} 
if and only if
\begin{align}\label{eq:opppcond}
\langle p_k, u_k \rangle= \mathcal{G}(u_k)\,,\quad 
\max_{v \in  B} \langle p_k, v \rangle = \max_{v \in \Ext B} \langle p_k, v \rangle \leq 1\,.
\end{align}
The second condition in \eqref{eq:opppcond} is satisfied by assumption, so we are only left to check the first.
If~$u_k=0$ this is trivial, since~$\G(0)=0$.
Therefore assume~$u_k \neq 0$. This implies 
$\G(u_k) \neq 0$. Indeed, if by contradiction 
$\G(u_k) = 0$, then by one-homogeneity of $\G$ we would obtain 
$\lambda u_k \in S^-(\G,0)$ for all $\lambda \geq 0$. But then  
$S^-(\G,0)$ would not be norm bounded, contradicting~\ref{ass:func2}. Thus $\G(u_k) \neq 0$. Since~$u_k/\mathcal{G}(u_k) \in B$, from the second condition in
\eqref{eq:opppcond} we get
\begin{align} \label{eq:optestimate}
\langle p_k, u_k \rangle \leq \mathcal{G}(u_k).
\end{align}
By construction $u_k=\sum^{N_k}_{i=1} \lambda^k_i u^k_i$, with $u_i^k \in \operatorname{Ext}(B)$ and~$\lambda^k \in \R_+^{N_k}$ a minimizer of
\begin{align*}
\min_{ \lambda \in \R_+^{N_k}} \left \lbrack F\left(\sum^{N_k}_{i=1} \lambda_i K u^k_i\right)+ \sum^{N_k}_{i=1} \lambda_i \right \rbrack.
\end{align*}
Note that since~$\lambda^k$ is a minimizer to the above problem, we have
\begin{align} \label{eq:optimlocal}
\langle p_k,u_k \rangle= \sum^{N_k}_{i=1} \lambda^k_i \langle p_k, u^k_i \rangle= \sum^{N_k}_{i=1} \lambda^k_i\,, \quad \langle p_k, u^k_i \rangle \leq 1 \,,
\end{align}
for~$i=1,\dots,N_k$.
As $u_k=\sum^{N_k}_{i=1} \lambda^k_i u^k_i$ with $u_i^k \in \operatorname{Ext}(B)$, by convexity and one-homogeneity of~$\G$ we estimate
$\G(u_k) \leq \sum^{N_k}_{i=1} \lambda^k_i$. 
Therefore $\langle p_k,u_k \rangle = \G(u_k)$ by  \eqref{eq:optestimate}-\eqref{eq:optimlocal}, ending the proof.
\end{proof}

\begin{remark}
Note that Algorithm~\ref{alg:gcg} terminates also in case~$\widehat{v}^u_k \in \mathcal{A}^u_k$ for some $k \geq 1$, i.e., if the algorithm cannot find a new point to insert. 
Indeed in this situation we have 
$\langle p_k,\widehat{v}^u_k \rangle \leq 1$
by the optimality conditions at~\eqref{eq:optimlocal}. Therefore
the stopping condition \eqref{eq:stopping_rev} is satisfied.
\end{remark}

\subsection{Worst-case convergence rates}
The main contribution of the present manuscript is the derivation of convergence results for the sequence of~\textit{residuals}
\begin{align} \label{def:residualu}
r_J(u_k)=J(u_k)- \min_{u \in \M} J(u)
\end{align}
associated to the iterates generated by Algorithm~\ref{alg:gcg}. This is a challenging task for a variety of reasons. For example the space~$\mathcal{M}$ is, in general, not reflexive. Moreover the functional $J$ lacks useful properties, such as smoothness or strict convexity. 
Classical approaches \cite{dunnimplicit, dunnnonsingular, auslender} provide sublinear convergence rates for GCG methods defined in Banach spaces, however very few linear convergence results are available in the literature and these results are usually limited to specific cases.
In this section, we state the convergence results anticipated in the introduction, and we detail the additional assumptions which are needed to prove them. Their proofs, which are rather technical, are then postponed to Sections~\ref{subsec:worstpdap} and~\ref{subsec:fastconvproof}.
We start with the following sublinear convergence result, which holds under the basic Assumptions~\ref{ass:func1}-\ref{ass:func3}.
\begin{theorem} \label{thm:convofcondu}
Let \ref{ass:func1}-\ref{ass:func3} in Assumption~\ref{ass:functions} hold. Then, Algorithm~\ref{alg:gcg} either terminates after a finite number of steps, with~$\bar{u}=u_k$ a minimizer to~\eqref{def:minprob}, or there is a constant $c>0$ such that
\begin{equation} \label{thm:convofcond:1}
r_J(u_k) \leq c \, \frac{1}{k+1} \quad \text{ for all } \,\, k \in \N\,,
\end{equation}
where $r_J$ is defined at \eqref{def:residualu}. Moreover, in this case, the sequence~$\{u_k\}_{k}$ admits at least one weak* accumulation point and each of such point is a solution to~\eqref{def:minprob}. If the solution~$\bar{u}$ to~\eqref{def:minprob} is unique, then we have~$u_k \weakstar \bar{u}$ in~$\M$ for the whole sequence. 
\end{theorem}
\subsection{Non-degeneracy and fast convergence} \label{subsec:nondegen}
While Theorem~\ref{thm:convofcondu} proves the convergence of Algorithm~\ref{alg:gcg}, the provided, slow, sublinear rate of convergence does not match the computed results of Section~\ref{subsec:guideex}. Motivated by this gap between theory and numerical observations, we argue the asymptotic linear convergence of~$r_J(u_k)$ provided that certain structural assumptions on problem~\eqref{def:minprob} are satisfied. First, in addition to Assumptions~\ref{ass:func1}-\ref{ass:func3}, we require that the solution set of the linear problem
\[
\max_{v \in \Bb}\, \langle \bar p,v \rangle
\]
consists of a finite number of extremal points $\bar{u}_1, \ldots, \bar{u}_N \in \ext(B)$, where $\bar{p}$ denotes the unique dual variable of \eqref{def:minprob}, see \eqref{eq:unique_dual}. Moreover, we ask that the restriction of the operator~$K$ into the span of~$\bar{\mathcal{A}}\coloneqq \{\bar{u}_i\}^N_{i=1}$ is injective. Such assumptions ensure the uniqueness of  the minimizer to \eqref{def:minprob}, denoted by $\bar u$, see Proposition \ref{prop:uniqueu} below. Additionally, we ask that $F$ is strongly convex around the unique optimal observation $\bar{y}$, see \eqref{eq:unique_dual}.  This set of assumptions is summarized below.
\begin{assumption}{\textit{(Uniqueness and strong convexity)}} \label{ass:fastconv1}
\newcounter{count}
\begin{enumerate}[label=\textnormal{(B\arabic*)}]
\item \label{ass:fast:F1} The map $F \colon Y \to \R$ is strictly convex and strongly convex around the unique optimal observation~$\bar{y}$, i.e., there  exists a neighborhood~$\mathcal{N}(\bar{y})$ of~$\bar{y}$ and~$\theta>0$ such that
\begin{align*}
(\nabla F(y_1)-\nabla F(y_2),y_1-y_2 )_Y \geq \theta \ynorm{y_1-y_2}^2 \,, \,\, \text{ for all } \,\,  y_1,y_2 \in \mathcal{N}(\bar{y})\,,
\end{align*}
\item \label{ass:fast:F2} There is~$N>0$ and a finite collection of extremal points~$ \bar{\mathcal{A}}:=\{\bar{u}_i\}^N_{i=1} \subset \operatorname{Ext}(B)$ such that the unique dual variable~$\bar{p}=-K_* \nabla F(\bar{y}) \in \mathcal{C}$ satisfies
\[
\argmax_{v \in \Bb}\, \langle \bar p,v \rangle = \left \{\,v \in \mathcal{B}\; | \;\langle \bar{p},v \rangle=1\,\right\}= \{\bar{u}_i\}^N_{i=1} \,,
\]
\item \label{ass:fast:F3} The set~$\{K\bar{u}_i\}^N_{i=1}\subset Y$ is linearly independent.
\setcounter{count}{\value{enumi}}
\end{enumerate}
\end{assumption}
We now check that the above assumptions imply uniqueness of solutions to \eqref{def:minprob}. 
\begin{proposition} \label{prop:uniqueu}
Let Assumptions~\ref{ass:func1}-\ref{ass:func3} and \ref{ass:fast:F2}-\ref{ass:fast:F3} hold. Then \eqref{def:minprob} admits a unique solution $\bar u$, which is of the form
\begin{equation} \label{eq:sparse_min}
\bar{u}=\sum^N_{i=1}\bar{\lambda}_i\bar{u}_i\,,
\end{equation}
for some $\bar{\lambda}_i\geq 0$, where $\{\bar{u}_i\}_{i=1}^N$ are the points from \ref{ass:fast:F2}.
\end{proposition}

\begin{proof}
Note that \eqref{def:minprob} admits at least one solution $\bar u$ thanks to Proposition~\ref{prop:existu}. We first claim that $\bar u \in \operatorname{cone}(\bar{\mathcal{A}})$, i.e., there exist~$\bar{\lambda}_i \geq 0$ such that~$\optu=\sum^N_{i=1} \bar{\lambda}_i \bar{u}_i$. If $\bar u =0$ the claim is trivial. Hence, suppose $\bar u \neq 0$. By the assumptions on $\G$ we conclude $\G(\bar u ) > 0$. Recalling the first order optimality conditions from Proposition~\ref{prop:existu}, we then infer $\langle \bar p, \bar v \rangle = 1$ with $\bar v:=\bar u / \G(\bar u) \in B$. Moreover, by Assumption~\ref{ass:fast:F2}, and 
arguing similarly to the proof of Lemma~\ref{lem:existoflinearized}, we get
\begin{align*}
\argmax_{v \in B} \ \langle \bar{p},v \rangle=\operatorname{conv}(\bar{\mathcal{A}}) \,, \qquad \max_{v \in B} \ \langle \bar{p},v \rangle = 1\,.
\end{align*}
 Therefore $\bar v \in \operatorname{conv}(\bar{\mathcal{A}})$, showing our claim that $\bar u \in \operatorname{cone}(\bar{\mathcal{A}})$.
Now assume that $\optu_1, \optu_2$ minimize in \eqref{def:minprob}. Then 
$\optu_1=\sum^N_{i=1}\bar{ \lambda}_i \bar{u}_i$ and  $\optu_2=\sum^N_{i=1}\bar{ \gamma}_i \bar{u}_i$ for some $\bar{\lambda}_i, \bar{\gamma}_i \geq 0$. By Proposition \ref{prop:existu} we have~$K \optu_1=K\optu_2$. 
From \ref{ass:fast:F3} and linearity of $K$ we then conclude $\bar{\lambda}_i = \bar{\gamma}_i$, so that $\optu_1=\optu_2$. Thus \eqref{def:minprob} has a unique solution, and such a solution is of the form \eqref{eq:sparse_min}.
\end{proof}

In the next set of assumptions we suppose~\textit{strict complementarity} for the minimizer $\bar u$, i.e.,
\begin{align*}
\bar{u} \not \in \operatorname{cone}\left(\bar{\mathcal{A}}\setminus \{\bar{u}_i\}\right)  \,, \quad \text{for every }\,\,i=1, \dots,N\,,
\end{align*}
or, equivalently,~$\bar{\lambda}_i >0$ for all~$i=1,\dots,N$. The final assumption concerns the existence of a ``distance function''~$g$ such that~$K$ is Lipschitz continuous and the linear functional $u \mapsto \langle \bar{p},u \rangle$ grows quadratically, both with respect to~$g$ and in the vicinity of~$\bar{u}_i \in \bar{\mathcal{A}}$. Of course, the particular form of~$g$ depends on the space~$\M$ and the functional~$\G$ and thus it has to be constructed on a case-by-case basis. We give an example in Section~\ref{subsec:guideex}. This set of assumptions is summarized below.
\begin{assumption}{\textit{(Non-degeneracy)}} \label{ass:fastconv2}
\begin{enumerate}[label=\textnormal{(B\arabic*)}]
 \setcounter{enumi}{\value{count}}
\item \label{ass:fast:F4} The unique minimizer~$\bar{u}$ of \eqref{def:minprob} in \eqref{eq:sparse_min} is such that $\bar{\lambda}_i>0$ for all $i = 1, \ldots,N$.
\item \label{ass:fast:F5} There exists a function~$g\colon \ext(B)\times \ext(B)\to [0,\infty)$, positive constants~$\tau,\kappa >0$ and pairwise disjoint $d_\Bb$-closed neighborhoods $\{\bar{U}_i\}_{i=1}^N$ of $\{\bar u_i\}_{i=1}^N$, such that
\begin{equation} \label{eq:quadgrowth}
\ynorm{K(u -\bar{u}_i)} \leq \tau g(u,\bar{u}_i) \, , \qquad 1-\langle \bar{p}, u  \rangle \geq \kappa g(u,\bar{u}_i)^2 \,,%
\end{equation}
for all~$ u \in U_i$ and $i=1,\dots,N$, where $U_i:=\bar{U}_i \cap \ext (B)$ and $d_\Bb$ is the metric at \eqref{eq:metric_db}. 
\end{enumerate}
\end{assumption}

\begin{remark} \label{rem:wlog}
Note that, without loss of generality, we can always choose $\{\bar{U}_i\}_{i=1}^N$ for which there exists a constant $\sigma > 0$ such that 
\begin{equation} \label{eq:isolmax}
\langle \bar{p}, u \rangle \leq 1-\sigma \quad \text{ for all } \,\, u \in \Bb \setminus \bar{U}_i,\  i = 1,\ldots, N\,,
\end{equation}
due to Assumption \ref{ass:fast:F2} and the $d_\Bb$-continuity of $u\mapsto \langle \bar{p}, u \rangle$.
\end{remark}

We can finally state the main convergence result of the paper. For its proof we refer to Section~\ref{subsec:fastconvproof}.
\begin{theorem} \label{thm:fastconvpreview}
Let Assumptions \ref{ass:func1}-\ref{ass:func3} and \ref{ass:fast:F1}-\ref{ass:fast:F5} hold. Then Algorithm~\ref{alg:gcg} either terminates after a finite number of steps, with~$\bar{u}=u_k$ the unique minimizer to~\eqref{def:minprob}, or there exists a constant $c>0$ and~$\zeta \in [3/4,1)$ such that
\begin{equation} \label{thm:convofcond:1rate}
r_J(u_k) \leq c \, \zeta^k 
\end{equation}
for all~$k\in\N$ sufficiently large. In this case there holds~$u_k \weakstar \optu$ in~$\M$.
\end{theorem}

\begin{remark} \label{rem:detailsonassumptions}
Before ending this section, we briefly comment on Assumptions~\ref{ass:fast:F1}-\ref{ass:fast:F5}. First, we note that Assumption~\ref{ass:fast:F3} is trivially fulfilled if~$K$ is injective, as seen for example in Section~\ref{subsec:guideex}. Moreover, Assumptions~\ref{ass:fast:F3},~\ref{ass:fast:F4}, as well as the uniqueness of the minimizer to~\eqref{def:minprob}, are not fully independent. Indeed, if~$\{K\bar{u}_i\}^N_{i=1}$ is linearly dependent and there is a solution~$\bar{u}=\sum^{N}_{i=1} \bar{\lambda}_i \bar{u}_i $ with~$\bar{\lambda}_i >0$, then there exists a second solution~$\tilde{u} \neq \bar{u}$ with~$\tilde{u}=\sum^{N}_{i=1} \tilde{\lambda}_i \bar{u}_i$,~$\tilde{\lambda}_i\geq 0$, and equality holds for at least one index, cf. also~\cite[Section~3.2]{pieper}. 
Second, it is worthwhile to further discuss the quadratic growth condition in Assumption~\ref{ass:fast:F5} and relate it to more well-known concepts in the literature. For this purpose, recall that the necessary and sufficient optimality conditions for \eqref{def:minprob} are given by the variational subgradient inequality
\begin{align*}
    \langle \bar{p}, u-\bar{u}\rangle+ \mathcal{G}(\bar{u}) \leq \mathcal{G}({u}) \quad \text{for all}~u \in \mathcal{M}\,.
\end{align*}
Due to the one-homogeneity of~$\mathcal{G}$, this can be equivalently reformulated as
\begin{align*}
\langle \bar{p}, \bar{u} \rangle = \mathcal{G}(\bar{u}), \quad 1 \geq \langle \bar{p}, u \rangle  \quad \text{for all}~u \in B\,,    
\end{align*}
see Proposition~\ref{prop:existu}. By applying Lemma~\ref{lem:existoflinearized}, this is equivalent to
\begin{align*}
    \langle \bar{p}, \bar{u} \rangle = \mathcal{G}(\bar{u}), \quad 1 \geq \langle \bar{p}, u \rangle  \quad \text{for all}~u \in \operatorname{Ext}(B).    
\end{align*}
Finally, due to Assumption~\ref{ass:fast:F2}, we arrive at
\begin{align*}
    \langle \bar{p}, \bar{u} \rangle = \mathcal{G}(\bar{u}), \quad 0 \geq \langle \bar{p}, u -\bar{u}_i \rangle  \quad \text{for all}~u \in \Ext{B}.
\end{align*}
In particular, this implies that~$\bar{p} \in \partial I_{\Ext{B}}(\bar{u}_i) $ where~$I_{\Ext{B}}(\bar{u}_i)$ denotes the subdifferential of the nonconvex indicator function of~$\Ext{B}$ at~$\bar{u}_i$. In this context, Assumption~\ref{ass:fast:F5} implies the locally strengthened condition 
\begin{align*}
    \langle \bar{p}, \bar{u} \rangle = \mathcal{G}(\bar{u}), \quad 0 \geq \langle \bar{p}, u- \bar{u}_i \rangle+ \kappa g(u, \bar{u}_i)^2  \quad \text{whenever}~u \in U_i\,.
\end{align*}
This is very reminiscent of the concept of~\textit{strongly metric subregular subdifferentials} in convex optimization which plays a vital role for the derivation of fast convergence rates for proximal point methods~\cite{artacho}, and ``vanilla'' generalized conditional gradient methods~\cite{kunischgcg}.   

We point out that, in the general case, we are not aware of possibly stronger, but more intuitive, structural assumptions on~$\bar{p}$ as well as~$\Ext{B}$ which eventually ensure Assumptions~\ref{ass:fast:F2} and~\ref{ass:fast:F5}. However, for particular instances such a characterization is indeed possible, see Section~\ref{sec:examples} and the examples and references therein. 
\end{remark}

\section{Examples}\label{sec:examples}
Summarizing the previous sections, we see that the practical application and the convergence analysis of Algorithm~\ref{alg:gcg} rest on three pillars: First, a characterization of the set of extremal points $\Ext{B}$, second, an efficient method for Step 2 in Algorithm~\ref{alg:gcg} and, third, the derivation of sufficient structural assumptions to ensure Assumptions~\ref{ass:fast:F1}-\ref{ass:fast:F5}. In this section, we outline this program for three examples, namely, sparse initial value identification, trace regularization, and minimum effort problems, see Sections~\ref{subsec:guideex},~\ref{sec:rank-one},~\ref{subsec:mineffort}, respectively. In all these cases, a particular focus is put on the computation of~$\widehat{v}^u_k$ and the verification of Assumption~\ref{ass:fast:F5}. For the sake of brevity, we decided to strike a balance between practically interesting settings and problems for which the characterization of~$\Ext{B}$ and the derivation of Assumption~\ref{ass:fast:F5} can be done in a concise manner. %
As a main take-away message, these examples suggest that there is no general ``recipe'' for the resolution of Step 2 in Algorithm~\ref{alg:gcg}. Quite the reverse, the method of choice for computing~$\widehat{v}^u_k$ as well as the associated computational burden strongly depend on the example at hand. 

We also stress that Algorithm~\ref{alg:gcg} is applicable to far more complex problems, in which characterizing extremal points and deriving quadratic growth conditions could also get much more convoluted. One such problem, namely the optimal transport regularization of dynamic inverse problems, is briefly teased in Section~\ref{sec:dynamic_inverse}, and will be the subject of a follow-up work. %
Other examples are given by works~\cite{krnorm,IglWal22}, in which the authors apply the program outlined above to certain regularizers given by certain infimal convolutions.

\subsection{Sparse source identification} \label{subsec:guideex}

Let us consider the inverse problem of identifying the initial source of a heat equation on a convex polygonal spatial domain~$\Omega \subset \R^2$ from distributed temperature measurements~$y_d$ at a given final time~$T>0$. Our particular interest lies in the recovery of sparse sources
\[
u^\dagger = \sum^N_{i=1} \lambda^\dagger_i \delta_{x^\dagger_i}\,,
\]
given as a linear combination of finitely many point measures, where the coefficients~$\lambda^\dagger_i \in \R$, the positions~$x^\dagger_i \in \Omega$, and the number~$N \in \N$ of points are all assumed to be unknown. Taking the ill-posedness of the described inverse problem into account we follow~\cite{pikka,leyk} and consider the convex Tikhonov-regularized problem 
\begin{align} \label{def:inverseheat}
\min_{u \in M(\Omega),y} \left \lbrack \frac{1}{2} \|y(T)-y_d\|^2_{L^2(\Omega)}+ \beta \|u\|_{M(\Omega)} \right \rbrack \,.
\end{align}
Here $M(\Omega)$ denotes the space of Borel measures on the open set~$\Omega$, $y$ is a scalar function defined on $[0,1] \times \Om$ with $y(t):=y(t)(\cdot)$,~$y_d \in L^2(\Omega)$ is a given desired state, and the pair $(y,u)$ satisfies, in the sense of distributions, the heat equation
\begin{align} \label{eq:heat}
\partial_t y-\bigtriangleup y=0~\text{in}~(0,T)\times \Omega,~y=0~ \text{in}~(0,T)\times \partial \Omega,~y(0)=u~\text{in}~\Omega\,.
\end{align}
The~\textit{a priori} assumption on the sparsity of the unknown source is encoded in the choice of the regularizer, defined as the total variation norm of~$u$ with~$\beta>0$. 
To fit~\eqref{def:inverseheat} into the setting of~\eqref{def:minprob} we set~$\mathcal{C}=C_0(\Omega)$, the space of continuous functions vanishing on~$\partial \Omega$, and we equip it with the canonical supremum norm
\begin{align*}
\|p\|_\mathcal{C}= \max_{x \in \Omega} |p(x)| \quad \text{for all}~ p \in C_0(\Omega) \,.
\end{align*}
This makes $\mathcal{C}$ a Banach space. According to the Riesz-Markov-Kakutani theorem we have~$\Cc^* \simeq \M$ for~$\M=M(\Omega)$. Moreover, define~$Y=L^2(\Omega)$,~$F=(1/2)\|\cdot-y_d\|^2_{L^2(\Omega)}$ and~$\mathcal{G}= \beta \|\cdot\|_{M(\Omega)}$. Of course, these functionals satisfy \ref{ass:func1} and \ref{ass:func2} in Assumption \ref{ass:functions}. Finally, we replace the PDE constraint by introducing a~\textit{source-to-observation} operator~$K \colon M(\Omega) \to L^2(\Omega)$ mapping a measure $u \in M(\Omega)$ to $y(T)$, where $y$ solves \eqref{eq:heat}. It is readily verified that~$K$ is injective, thanks to a priori estimates for weak solutions to \eqref{eq:heat} \cite[Lemma 2.2]{zuazua}, as well as weak*-to-strong continuous~\cite[Lemma 2.3]{zuazua}. Hence,~\eqref{def:inverseheat} admits a unique solution and \ref{ass:func3} in Assumption~\ref{ass:functions} is satisfied. Moreover, $K$ is the adjoint of the operator~$K_* \colon L^2(\Omega) \to C_0(\Omega)$ defined by~$K_* \varphi:=z(0)$, where the pair~$(z,\f)$ satisfies, in the sense of distributions, the backwards heat equation
\begin{align} \label{eq:heatadjoint}
\partial_t z + \bigtriangleup z= 0 ~\text{in}~(0,T)\times \Omega,~z=0~ \text{in}~(0,T)\times \partial \Omega,~z(T)=\varphi~\text{in}~\Omega \,.
\end{align}
For more details we refer to~\cite{leyk, zuazua}. Note that $K_*$ is well-defined as $z(0) \in C_0(\Omega)$, due to parabolic regularity estimates. For sake of completeness, this is justified in Lemma \ref{lem:regadjointheat} in the appendix.  \\
The next lemma characterizes the set of extremal points of the unit ball of the regularizer $\G$, that is, of the set~$B=S^{-}(\beta \|\cdot\|_{M(\Omega)},1 )$.
\begin{lemma}\label{lem:exttotalvariation}
We have
\begin{equation}
\ext(B) = \{\sigma \beta^{-1} \delta_x \, \colon \, x \in \Omega,\, \sigma \in \{-1,1\} \}\,.
\end{equation}
Moreover, $\overline{\ext(B)}^* = \ext(B) \cup \{0\}$.
\end{lemma}
\begin{proof}
The characterization of $\ext(B)$ is well-known \cite[Proposition 4.1]{bc}.
As for the second claim, note that $0 \in \overline{\ext(B)}^*$: indeed one can take $\{x_k\}_k$ in  $\Om$ such that $x_k \to x$ with $x \in \partial \Omega$, so that $u_k:=\beta^{-1} \delta_{x_k} \weakstar 0$. Thus $\ext(B) \cup \{0\} \subset \overline{\ext(B)}^*$. For the opposite  inclusion, assume given a sequence $u_k=\sigma_k \beta^{-1} \delta_{x_k}$ in $\ext(B)$, such that $u_k \weakstar u$. Then, up to subsequences, $\sigma_k \rightarrow \sigma \in \{-1,1\}$ and $x_k \rightarrow x \in \overline{\Om}$. Hence, if $x\in \partial \Omega$ then $u = 0$, while if $x \in \Omega$ then $u=\sigma \beta^{-1} \delta_{x} \in \ext(B)$.
\end{proof}
Moreover, applying Proposition \ref{prop:existu} and the characterization of $K_*$, we immediately deduce the optimality conditions for \eqref{def:inverseheat}. 
\begin{proposition}\label{prop:optheat}
Let~$\bar{u} \in M(\Omega)$ be given and denote by~$\bar{z}$ the solution to~\eqref{eq:heatadjoint} for~$\varphi=y_d-\bar{y}$,~$\bar{y}=K\bar{u}$. 
 Then~$\bar{u}$ is a minimizer of~\eqref{def:inverseheat} if and only if
\begin{align*}
\langle \bar{z}(0), \bar{u} \rangle= \beta \|\bar{u}\|_{M(\Omega)}, \quad \|\bar{z}(0)\|_\mathcal{C} \leq \beta\,.
\end{align*} 
This implies
\begin{align*}
\supp \bar{u} \subset \left\{\,x \in \Omega \;|\;|\bar{z}(0)(x)|=\beta\,\right\}.
\end{align*}
\end{proposition}
Thanks to the characterization of extremal points presented in Lemma~\ref{lem:exttotalvariation}, the FC-GCG method presented in Algorithm~\ref{alg:gcg} for solving \eqref{def:inverseheat} generates a sequence of iterates 
\[
u_k=\sum^{N_k}_{i=1} \lambda^k_i u_i^k \,, \quad u_i^k = \sigma^k_i \beta^{-1} \delta_{x^k_i}\,, \quad \sigma^k_i \in \{-1,1\}\,, \quad x^k_i \in \Omega \,, \quad \lambda_i^k \geq 0\,,
\]
as well as an associated sequence of active sets $\mathcal{A}^u_k=\{u^k_i\}^{N_k}_{i=1}$. Moreover, let~$z_k$ denote the solution of~\eqref{eq:heatadjoint} for~$\varphi=y_d-y_k(T)$,~$y_k(T)=Ku_k$. We now claim that the new candidate extremal point in the iteration~$k$ of Algorithm~\ref{alg:gcg} can be chosen as
\begin{align} \label{eq:candidatemeasures}
\widehat{v}^u_k = \operatorname{sign}(z_k(0) (\hat{x}_k)) \beta^{-1}\delta_{\hat{x}_k}~\text{ where }~\hat{x}_k \in \argmax_{x \in \Omega}|z_k(0)(x)|,
\end{align}
i.e.,~Step 2 in Algorithm~\ref{alg:gcg} is equivalent to computing a global extremum of a continuous function.
This is verified in the following proposition.
\begin{proposition}
Let~$\widehat{v}^u_{k}$ be defined as in~\eqref{eq:candidatemeasures}.
There holds~$z_k(0)=-K_*\nabla F(Ku_k)$ as well as
\begin{align*}
\langle z_k (0), \widehat{v}^u_k \rangle= \max_{v \in \mathcal{B}}\, \langle z_k(0),v \rangle\,.
\end{align*} 
\end{proposition}
\begin{proof}
We directly get~$z_k(0)=K_*(y_k(T)-y_d)=-K_*\nabla F(Ku_k)$ from the characterization of~$K_*$. The remaining statement follows directly from
\begin{align*}
\langle z_k(0),v \rangle \leq \cnorm{z_k(0)} \mnorm{v} \leq \cnorm{z_k(0)}/\beta \quad \text{for all}~ v \in \mathcal{B}
\end{align*}
as well as
\begin{equation*}
\langle z_k(0),\widehat{v}^u_k \rangle= \operatorname{sign}(z_k(0)(\hat{x}_k)) \beta^{-1} z_k(0)(\hat{x}_k)= \cnorm{z_k(0)}/\beta\,.\qedhere
\end{equation*}
\end{proof}

Thus, for sparsity examples, solving \eqref{eq:descr_insertion} amounts to computing an extremum of the continuous function $z_k(0) = -K_* \nabla F(Ku_k)$, where $u_k$ is the iterate generated by Algorithm \ref{alg:gcg}. Since $|z_k(0)|$ is, in general, non-concave, this optimization task could be non-trivial. However, since the spatial domain $\Omega$ is low dimensional in this example, it is possible to resort to heuristic strategies to approximate the extremum of $z_k(0)$. In particular, a standard widely accepted strategy \cite{pikka, boyd}, consists in discretizing the domain $\Omega$ using a uniform grid $\{x^h\}_{h = 1, \ldots N^2}$ and performing local searches around $x^h$ using gradient descent methods. Then the extremum of $z_k(0)$ can be estimated as 
\begin{align}
    \argmax_{y^h : \, h = 1, \ldots, N^2} z_k(0)(y^h)\,,
\end{align}
where $y^h$ is the outcome of the local search around the point of the grid $x^h$. Moreover, a practical implementation of Algorithm~\ref{alg:gcg} for this particular problem also entails a discretization of the heat equation, e.g. by piecewise polynomial and continuous finite elements. In this case, the computation of the new point~$\widebar{x}_k$ becomes trivial, see  Section~\ref{subsubsec:numex}.

We now discuss the non-degeneracy conditions. Denoting by $\bar u \in M(\Omega)$ the minimizer of \eqref{def:inverseheat}, we propose a natural and easy to verify set of assumptions for $\bar z(0)$ that implies our general non-degeneracy assumptions from Section~\ref{subsec:nondegen} for a suitable choice of $g$. More precisely, this new set of assumptions on $\bar z(0)$ will imply Assumption \ref{ass:fast:F2}, \ref{ass:fast:F3} and \ref{ass:fast:F5}. We remark that Assumption \ref{ass:fast:F4} still needs to be assumed to ensure the fast convergence; however we decide not to state it in the following set of assumptions as, for this specific example, it would be formulated exactly as in \ref{ass:fast:F4}. Moreover, its verification can be done straightforwardly looking at the structure of the unique minimizer of \eqref{def:inverseheat}.
We finally remind that $\bar{z}$ is the solution to~\eqref{eq:heatadjoint} for~$\varphi=y_d-\bar{y}$,~$\bar{y}=K\bar{u}$ and therefore by the characterization of $K_*$ it holds that ~$\bar z(0)=-K_*\nabla F(K \bar u)$.
The following structural assumptions are made, see also~\cite{leyk,pieper}. 

\setcounter{enumi}{\value{count}}
\begin{enumerate}[label=\textnormal{(C\arabic*)}]
\item \label{ass:guidingG1} There are~$\bar{x}_i \in \Omega$,~$i=1,\dots,N$,~$N>0$, such that
\begin{align} \label{finitesuppheat}
\left\{\,x \in \Omega \;|\;|\bar{z}(0)(x)|=\beta\,\right\}=\{\bar{x}_i\}^N_{i=1}.
\end{align}
\item  \label{ass:guidingG2}  There exists~$\gamma >0$ such that for all~$i=1,\dots, N$ we have
\begin{align*}
\operatorname{sign}(\bar{z}(0)(\bar{x}_i))(\delta x,\nabla^2 \bar{z}(0)(\bar{x}_i) \delta x)_{\R^2} \leq -\gamma |\delta x|^2 \,\, \text{ for all } \, \delta x \in \R^2.
\end{align*}
\end{enumerate}
Regarding \ref{ass:guidingG2} recall that~$\bar{z}(0)$ is at least two times continuously differentiable in the vicinity of~$\bar{x}_i,~i=1,\dots,N$, see Lemma~\ref{lem:regadjointheat}. 
Loosely speaking, the additional requirements  in Assumptions~\ref{ass:guidingG1}-\ref{ass:guidingG2} state that~$\bar{z}(0)$ only admits a finite number of global minima/maxima and its curvature around them does not degenerate. The latter corresponds to a second order sufficient optimality condition for the global extrema of~$\bar{z}(0)$.
We now prove that \ref{ass:guidingG1} and \ref{ass:guidingG2} imply \ref{ass:fast:F2}, \ref{ass:fast:F3} and \ref{ass:fast:F5}.
First we show that~$\ref{ass:guidingG1}$ guarantees Assumptions~\ref{ass:fast:F2}-\ref{ass:fast:F3}. 
For this purpose, set $\bar{u}_i=\operatorname{sign}(\bar{z}(0)(\bar{x}_i))\beta^{-1}\delta_{\bar{x}_i} \in \ext(B)$ for every $i=1,\dots,N$.
\begin{proposition}
Let Assumption $\ref{ass:guidingG1}$ hold.
Then, we have
\begin{align*}
\argmax_{v \in \overline{\ext(B)}^*} \langle \bar{z}(0), v \rangle= \left\{\bar{u}_i\right\}^N_{i=1}. 
\end{align*}
Moreover, the set~$\{K \bar{u}_i\}^N_{i=1}$ is linearly independent.
\end{proposition}
\begin{proof}
First recall that every~$v \in \Ext B $ is of the form~$v= \sigma \beta^{-1} \delta_x$ for some~$\sigma \in \{-1,1\},~x \in \Omega$. Moreover, we have
\begin{align*}
\langle \bar{z}(0), \sigma \beta^{-1} \delta_x \rangle= (\sigma/\beta) \bar{z}(0)(x) \leq \|\bar{z}(0)\|_\mathcal{C} /\beta=1,
\end{align*}
see Proposition~\ref{prop:optheat} and~\eqref{finitesuppheat},
with equality if and only if~$|\bar{z}(0)(x)|=\|\bar{z}(0)\|_\mathcal{C}$,~$\sigma= \operatorname{sign}(\bar{z}(0)(x))$. Hence, the claimed statement follows from~\eqref{finitesuppheat} and Lemma \ref{lem:exttotalvariation}. Finally, the linear independence of~$\{K \bar{u}_i\}^N_{i=1}$ follows from the injectivity of~$K$.
\end{proof}
Next we address Assumption~\ref{ass:fast:F5}. For every subdomain $\Omega_0$ with $\bar \Omega_0 \subset \Omega$ define the quantities
\begin{equation}
\|\psi\|_{\operatorname{Lip}(\Omega_0)} := \sup_{x,y \in \Omega_0, x \neq y} \frac{|\psi(x) - \psi(y)|}{|x-y|}
\end{equation}
for $\psi \in C_0(\Omega)$ and
\begin{equation}
\|K_*\|_{Y,\operatorname{Lip}(\Omega_0)} := \sup_{\|y\|_Y \leq 1} \|K_*y\|_{\operatorname{Lip}(\Omega_0)}\,.
\end{equation}
Note that $\|K_*\|_{Y,\operatorname{Lip}(\Omega_0)} < \infty$, due to Lemma~\ref{lem:regadjointheat}. The next lemma shows that Assumptions~\ref{ass:guidingG1}-\ref{ass:guidingG2} imply the quadratic growth of~$\beta-|\bar{z}(0)|$ around~$\{\bar{x}_i\}^N_{i=1}$. The proof can be found in Section \ref{app:lem:quadgrowthmeasure} in the appendix. 
\begin{lemma} \label{lem:quadgrowthmeasure}
Let Assumption~\ref{ass:guidingG1}-\ref{ass:guidingG2} hold and fix an index~$i=1,\dots,N$. Then, there exists~$R>0$ such that~$\bar{B}_R(\bar{x}_i)\subset \Omega$,~$\operatorname{sign}(\bar{z}(0)(x))=\operatorname{sign}(\bar{z}(0)(\bar{x}_i))$ for all~$x \in B_R(\bar{x}_i)$ and
\begin{align}
&\beta-|\bar{z}(0)(x)| \geq (\gamma/4)|x-\bar{x}_i|^2 \quad \text{for all}~ x \in B_R(\bar{x}_i)\,, \label{eq:heatest1}\\
\medskip
&\ynorm{K(\delta_{x}-\delta_{\bar{x}_i})} \leq \|K_*\|_{Y, \operatorname{Lip}(B_R(\bar{x}_i))} |x-\bar{x}_i|, \quad \text{for all}~ x \in B_R(\bar{x}_i)\,.\label{eq:heatest2}
\end{align}
\end{lemma}

Now, given arbitrary extremal points~$\sigma_1 \beta^{-1} \delta_{x_1}, \sigma_2 \beta^{-1} \delta_{x_2} \in \Ext B$, with $~\sigma_1,\sigma_2 \in \{-1,+1\}$ and $x_1,x_2 \in \Omega$, define the distance function $g\colon \ext(B)\times \ext(B)\to [0,\infty)$ by 
\begin{align}
g(u_1,u_2)= |\sigma_1-\sigma_2|+ |x_1-x_2|\,.
\end{align}
Such $g$ will be the one verifying Assumption~\ref{ass:fast:F5}.
In the next lemma we show that the weak* convergence  in~$M(\Omega)$ of a sequence of extremal points to $\bar u_i$ is equivalent to convergence with respect to~$g$. The proof can be found in Section \ref{app:lem:equivalenceofweaktog} in the appendix.
\begin{lemma} \label{lem:equivalenceofweaktog}
Consider a sequence~$\{u_k\}_{k}$ in~$ \Ext B$, i.e.,~$u_k= \sigma_k \beta^{-1} \delta_{x_k}$, for~$\sigma_k \in \{-1,+1\}$ and~$x_k \in \Omega$. Then, there holds
\begin{align*}
u_k \weakstar \bar{u}_i ~\text{if and only if}~ \lim_{k \rightarrow \infty} g(u_k,\bar{u}_i)=0\,.
\end{align*}
In particular, if~$u_k \weakstar \bar{u}_i$, then~$\sigma_k= \operatorname{sign}(\bar{z}(0)(\bar{x}_i))$ for all~$k \in \N$ large enough.
\end{lemma}
Finally we combine the previous observations to conclude Assumption~\ref{ass:fast:F5}. 
\begin{proposition}
Let~$R>0$ be chosen according to Lemma~\ref{lem:quadgrowthmeasure}. Then, there is a~$d_{\mathcal{B}}$-neighbourhood $\bar{U}_i$ of~$\bar{u}_i$ which satisfies
\begin{align*}
U_i \coloneqq \bar{U}_i \cap \Ext B \subset \left\{\,\sigma \beta^{-1}\delta_x \;|\;\sigma=\operatorname{sign}(\bar{z}(0)(\bar{x}_i)),~x \in B_R(\bar{x}_i)\,\right\}
\end{align*}
as well as
\[
\begin{gathered}
\|K(u-\bar{u}_i)\|_{L^2} \leq (\|K_*\|_{Y, \operatorname{Lip}(B_R(\bar{x}_i))}/\beta) \, g(u, \bar{u}_i) \\ \langle \bar{z}(0),\bar{u}_i-u \rangle \geq (\gamma/(4 \beta)) \, g(u, \bar{u}_i)^2	
\end{gathered}
\]
for every  $u \in U_i$ and for every $i=1, \ldots, N$.
\end{proposition}
\begin{proof}
The statement on the existence of a~$d_\mathcal{B}$-neighbourhood~$\bar{U}_i$ with the stated properties follows immediately from Lemma~\ref{lem:equivalenceofweaktog}. In fact, if such a neighbourhood does not exist, then there exists a sequence~$\{u_k\}_k$ in~$\Ext B$,~i.e., $u_k= \sigma_k \beta^{-1} \delta_{x_k}$, with~$u_k \weakstar \bar{u}_i$ as well as~$\sigma_k \neq \operatorname{sign}(\bar{z}(0)(\bar{x}_i))$  or~$x_k \not \in B_R(\bar{x}_i)$. This contradicts~$\lim_{k \rightarrow \infty} g(u_k, \bar{u}_i)=0$. The remaining statements are also readily verified using Lemma~\ref{lem:quadgrowthmeasure} noting that for every~$u= \sigma \beta^{-1} \delta_x \in U_i$, we have~$g(u,\bar{u}_i)=|x-\bar{x}_i|$,~$\langle \bar{z}(0), u \rangle= |\bar{z}(0)(x)|/\beta $ and~$\langle \bar{z}(0),\bar{u}_i \rangle=1$.
\end{proof}

\subsubsection{Numerical experiment} \label{subsubsec:numex}
We close this section with a numerical experiment showing the effectiveness of Algorithm \ref{alg:gcg} in a concrete setting. For this purpose, set~$\Omega=(0,1)^2$ and~$T=0.1$. Moreover, fix~$u^\dagger=25 \delta_{x_1}-10\delta_{x_2}$, where~$x_1=(0.75,0.75)$ and~$x_2=(0.25,0.25)$, as well as~$y_d=Ku^\dagger+ \zeta$ where~$\zeta$ is a noise term with~$\|\zeta\|_{L^2(\Omega)}/\|Ku^\dagger\|_{L^2(\Omega)}\approx 0.1$. The regularization parameter is chosen as~$\beta=0.001$. The heat equation is discretized using a dg(0)cg(1) scheme on a temporal grid with stepsize~$\delta=0.001$ and a uniform triangulation of~$\Omega$ with grid size~$h=1/128$ and nodes~$\{x^h_i\}^{N_h}_{i=1}$. For the adjoint equation, a conforming discretization scheme is considered. All computations were carried out in Matlab 2019 on a notebook with~$32$ GB RAM and an Intel\textregistered Core\texttrademark ~i7-10870H CPU@2.20 GHz. 
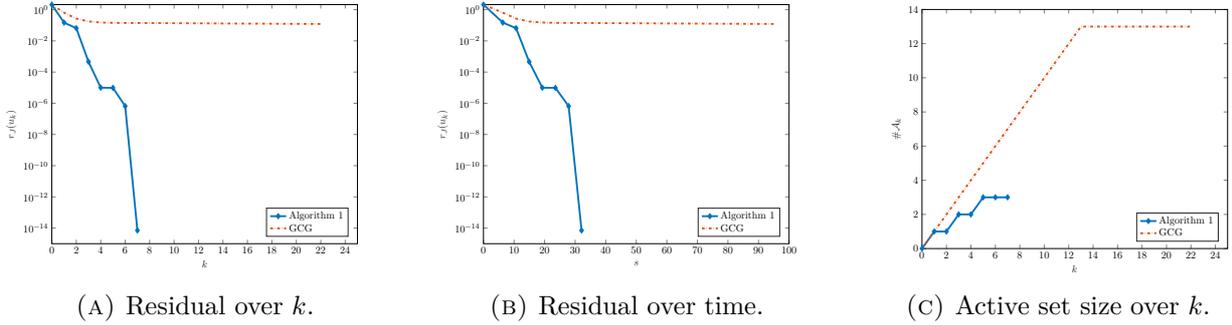
\begin{figure}[t!]
\begin{subfigure}[t]{.31\linewidth}
\centering
\scalebox{.35}{%
%
%
\definecolor{mycolor1}{rgb}{0.00000,0.44700,0.74100}%
\definecolor{mycolor2}{rgb}{0.85000,0.32500,0.09800}%
\begin{tikzpicture}

\begin{axis}[%
width=4.521in,
height=3.566in,
at={(0.758in,0.481in)},
scale only axis,
xmin=0,
xmax=25,
xlabel style={font=\color{white!15!black}},
xlabel={$k$},
ymode=log,
ymin=1e-15,
ymax=2.10832950982808,
yminorticks=true,
ylabel style={font=\color{white!15!black}},
ylabel={$r_J(u_k)$},
axis background/.style={fill=white},
legend style={at={(0.97,0.03)}, anchor=south east, legend cell align=left, align=left, draw=white!15!black}
]
\addplot [color=mycolor1, line width=2.0pt, mark=diamond, mark options={solid, mycolor1}]
  table[row sep=crcr]{%
0	2.10832950982809\\
1	0.149088700654287\\
2	0.0653573536294753\\
3	0.00045885891312108\\
4	9.98220022215197e-06\\
5	9.72615968179216e-06\\
6	6.57560524033069e-07\\
7	7.08658248948355e-15\\
};
\addlegendentry{Algorithm 1}

\addplot [color=mycolor2, dashdotted, line width=2.0pt]
  table[row sep=crcr]{%
0	2.10832950982809\\
1	0.638027412321589\\
2	0.268892675488535\\
3	0.175414199456924\\
4	0.150661458097042\\
5	0.143365771858187\\
7	0.139959320639744\\
14	0.130609635202745\\
22	0.12092646296188\\
};
\addlegendentry{GCG}

\end{axis}

\begin{axis}[%
width=5.833in,
height=4.375in,
at={(0in,0in)},
scale only axis,
xmin=0,
xmax=1,
ymin=0,
ymax=1,
axis line style={draw=none},
ticks=none,
axis x line*=bottom,
axis y line*=left
]
\end{axis}
\end{tikzpicture}%
}%
\caption{Residual over $k$.}
\label{fig:residualk}
\end{subfigure}
\quad
\begin{subfigure}[t]{.31\linewidth}
\centering
\scalebox{.35}{%
\centering
%
%
\definecolor{mycolor1}{rgb}{0.00000,0.44700,0.74100}%
\definecolor{mycolor2}{rgb}{0.85000,0.32500,0.09800}%
\begin{tikzpicture}

\begin{axis}[%
width=4.521in,
height=3.566in,
at={(0.758in,0.481in)},
scale only axis,
xmin=0,
xmax=100,
xlabel style={font=\color{white!15!black}},
xlabel={$s$},
ymode=log,
ymin=1e-15,
ymax=2.10832950982808,
yminorticks=true,
ylabel style={font=\color{white!15!black}},
ylabel={$r_J(u_k)$},
axis background/.style={fill=white},
legend style={at={(0.97,0.03)}, anchor=south east, legend cell align=left, align=left, draw=white!15!black}
]
\addplot [color=mycolor1, line width=2.0pt, mark=diamond, mark options={solid, mycolor1}]
  table[row sep=crcr]{%
0.000552399999996567	2.10832950982809\\
6.3138105	0.149088700654287\\
10.6612504	0.0653573536294758\\
14.9518766	0.000458858913121078\\
19.2530354	9.98220022215201e-06\\
23.5446021	9.72615968179208e-06\\
27.872096	6.57560524033064e-07\\
32.1023323	7.08658248948355e-15\\
};
\addlegendentry{Algorithm 1}

\addplot [color=mycolor2, dashdotted, line width=2.0pt]
  table[row sep=crcr]{%
0.00119599999999309	2.10832950982805\\
6.37494580000001	0.638027412321584\\
10.6335267	0.26889267548854\\
14.8600594	0.175414199456925\\
19.1056026	0.150661458097042\\
23.3568465	0.143365771858187\\
31.9192837	0.139959320639744\\
61.5736379	0.130609635202746\\
95.6039629	0.120926462961881\\
};
\addlegendentry{GCG}

\end{axis}

\begin{axis}[%
width=5.833in,
height=4.375in,
at={(0in,0in)},
scale only axis,
xmin=0,
xmax=1,
ymin=0,
ymax=1,
axis line style={draw=none},
ticks=none,
axis x line*=bottom,
axis y line*=left
]
\end{axis}
\end{tikzpicture}%
}%
\caption{Residual over time.}
\label{fig:residuals}
\end{subfigure}
\quad
\begin{subfigure}[t]{.31\linewidth}
\centering
\scalebox{.35}{%
\centering
%
%
\definecolor{mycolor1}{rgb}{0.00000,0.44700,0.74100}%
\definecolor{mycolor2}{rgb}{0.85000,0.32500,0.09800}%
\begin{tikzpicture}

\begin{axis}[%
width=4.521in,
height=3.566in,
at={(0.758in,0.481in)},
scale only axis,
xmin=0,
xmax=25,
xlabel style={font=\color{white!15!black}},
xlabel={$k$},
ymin=0,
ymax=14,
ylabel style={font=\color{white!15!black}},
ylabel={ $\# \mathcal{A}_k $},
axis background/.style={fill=white},
legend style={at={(0.97,0.03)}, anchor=south east, legend cell align=left, align=left, draw=white!15!black}
]
\addplot [color=mycolor1, line width=2.0pt, mark=diamond, mark options={solid, mycolor1}]
  table[row sep=crcr]{%
0	0\\
1	1\\
2	1\\
3	2\\
4	2\\
5	3\\
6	3\\
7	3\\
};
\addlegendentry{Algorithm 1}

\addplot [color=mycolor2, dashdotted, line width=2.0pt]
  table[row sep=crcr]{%
0	0\\
13	13\\
22	13\\
};
\addlegendentry{GCG}

\end{axis}
\end{tikzpicture}%
}%
\caption{Active set size over~$k$.}
\label{fig:support}
\end{subfigure}
\caption{Convergence behaviour of relevant quantities.}
\end{figure}
In Figure~\ref{fig:residualk}, we report on the convergence history of the residuals~$r_J(u_k)=J(u_k)-\min_{u \in M(\Omega)} J(u)$ associated with a sequence~$\{u_k\}_k$ generated by Algorithm~\ref{alg:gcg} starting from $u_0=0$ and~$\mathcal{A}_0=\emptyset$. Due to the dg(0)cg(1) scheme used in the discretization of the state equation, the dual variable~$z_k(0)$ is now a piecewise linear and continuous function on the spatial grid. As a consequence,~$|z_k(0)|$ achieves its global maximum in a gridpoint and the new candidate~$\widehat{x}_k$ can be cheaply computed as
\begin{align*}
    \widehat{x}_k \in \argmax_{x \in \{x^h_i\}^{N_h}_{i=1}} |z_k(0)|\,.
\end{align*}

In each iteration, the finite dimensional subproblem~\eqref{eq:subprobalgo} is solved using a semismooth Newton method. Moreover, we  plot the size of the support of~$u_k$ in dependence of~$k$ in Figure~\ref{fig:support}. By construction, this corresponds to the number of Dirac deltas in the active set~$\mathcal{A}^u_k$. In order to highlight the practical efficiency of Algorithm~\ref{alg:gcg}, we also include a comparison to the iterates generated by a generalized conditional gradient method (GCG) given by
\begin{align*}
u_{k+1}=(1-s_k) u_k+ s_k v_k \ \ ~\text{where} \ \ ~v_k= \begin{cases} M_0 \widehat{v}^u_k & \text{if } \, \|z_k(0)\|_{\Cc} \geq \beta \,,\\
 0 & \text{else.}
 \end{cases}
\end{align*}
Here~$\widehat{v}^u_k$ is chosen as before,~$M_0=J(0)/\beta$ and~$s_k \in (0,1)$ is an explicitly given stepsize as described in~\cite{pikka}. Both methods were run for a maximum of 100 iterations or until~$r_J(u_k)\leq 10^{-12}$. As expected, the GCG update exhibits the typical sublinear convergence behaviour of conditional gradient methods. In particular, after 200 iterations the residual is still of magnitude~$r_J(u_k) \approx 5 \times 10^{-2}$. In contrast, we observe a vastly improved rate of convergence for Algorithm~\ref{alg:gcg}. The stopping criterion is met after~$7$ iterations. Moreover, while the support size of~$u_k$ in GCG strictly increases in the first 13 iterations, Algorithm~\ref{alg:gcg} removes Dirac deltas which are assigned a zero coefficient. This leads to smaller active sets and thus sparser iterates. Both observations are testament to the practical efficiency of Algorithm~\ref{alg:gcg}. Finally, for a fair comparison, we also plot the residual as a function of the computational time (in seconds) for~$k=1,\dots,20$. This is done to acknowledge the vastly different computational cost of the update steps in both methods, i.e., forming a convex combination in GCG and the full resolution of a finite-dimensional minimization problem in~Algorithm~\ref{alg:gcg}. As we can see, the additional computational effort of fully resolving~\eqref{eq:subprobalgo} is outweighed by its practical utility. More in detail, Algorithm~\ref{alg:gcg} converges after around 30s while GCG fails to decrease~$r_J(u_k)$ below~$10^{-1}$ in the considered time frame.

\subsection{Rank-one matrix reconstruction by trace regularization} \label{sec:rank-one}
Let~$H$ be a separable Hilbert space with norm~$\|\cdot\|_H$ induced by an inner product~$(\cdot,\cdot)_H$. In the following, to simplify notation, we will focus on infinite dimensional Hilbert spaces, the finite dimensional case,~i.e.~$H \simeq \R^n$, follows by the same arguments. Denote by~$K(H)$ the space of bounded, linear, compact and self-adjoint operators from~$H$ into itself which we equip with the standard operator norm. An operator~$\mathcal U \in K(H)$ is called positive, denoted by~$\mathcal U\geq 0$, if
\begin{align*}
    (h,\mathcal Uh)_H \geq 0 \quad \text{for all}~h \in H\,.
\end{align*}
Moreover, let~$\{h_i\}_{i \in \N}$ be an orthonormal basis (ONB) of~$H$. For an operator~$ \mathcal U \in K(H)$ we formally define its trace as
\begin{align*}
    \operatorname{Tr}(\mathcal U)= \sum_{i \in \mathbb{N}} (h_i,\mathcal Uh_i)_H\,.
\end{align*}
An operator~$\mathcal U \in K(H) $ is called~\textit{trace-class} if~$\operatorname{Tr}(|\mathcal U|)< \infty$, where~$|\mathcal U|$ is the unique positive square root of~$\mathcal U^2$, and~\textit{Hilbert-Schmidt} (HS) if~$\operatorname{Tr}(\mathcal U^2) < \infty$. The set of all Hilbert-Schmidt operators together with the norm~$\|\mathcal U\|_{\text{HS}}=(\mathcal U,\mathcal U)^{1/2}_{\text{HS}}$ induced by the inner product 
\begin{align*}
   (\mathcal U_1,\mathcal U_2)_{\text{HS}}= \operatorname{Tr}(\mathcal U_1 \mathcal U_2) \quad \text{for all}~\mathcal U_1,\mathcal U_2 \in \operatorname{HS}(H)
\end{align*}
forms a Hilbert space~$\text{HS}(H)$. Moreover,
the space of all trace-class operators, denoted by~$T(H)$, forms a Banach space when equipped with the nuclear norm~$\|\mathcal U\|_{\text{T}}=\operatorname{Tr}(|\mathcal U|)$. In this case, we also have~$K(H)^* \simeq T(H)$ where the duality pairing is realized by
\begin{align*}
    \langle \mathcal U_1, \mathcal U_2  \rangle= \operatorname{Tr}(\mathcal U_1 \mathcal U_2) \quad \text{for all}~\mathcal U_1 \in K(H),~\mathcal U_2 \in T(H)\,.
\end{align*}
With these prerequisites, consider  
\begin{align} \label{def:rankproblem}
\min_{\mathcal U \in T(H),~\mathcal U \geq 0} \ \left \lbrack F(K\mathcal U)+ \beta \operatorname{Tr}(\mathcal U) \right \rbrack
\end{align}
where~$\beta>0$ and~$K \colon \text{HS}(H) \to Y $ is weak-to-strong continuous. Problems of this type appear, e.g., as convex relaxations of quadratic inverse problems, see~\cite{candesphase,brediesphase}, since they are known to favour solutions with rank one. As before, we start by computing the extremal points of
\begin{align*}
    B= \left\{\,\mathcal U\in T(H) \;|\; \beta \operatorname{Tr}(\mathcal U)\leq 1,~\mathcal U \geq 0\,\right\}\,.
\end{align*}
For this purpose, given~$h \in H$, we introduce the associated rank one operator~$\mathcal  U_h \coloneqq h \otimes h$ by~$\mathcal U_h h_1= (h,h_1)_H h$ for all~$h_1 \in H$. Note that~$\mathcal U_h \in T(H) $ with~$\operatorname{Tr}(\mathcal U_h)=\|h\|^2_H$. Moreover, we recall that for every~$\mathcal U \in K(H)$,~$\mathcal U \geq 0$, there exists an ONB~$\{h^{\mathcal {U}}_i\}_i \subset H$ as well as~$\{\sigma^{\mathcal{U}}_i\}_i \subset \R$ with~$\sigma^{\mathcal{U}}_i \geq \sigma^{\mathcal{U}}_{i+1} \geq 0 $,~$i\in\N$, and
\begin{align*}
    \mathcal{U} h^\mathcal{U}_i= \sigma^\mathcal{U}_i h^\mathcal{U}_i, \quad \mathcal{U}=\sum_{i \in \N} \sigma^\mathcal{U}_i h^\mathcal{U}_i \otimes h^\mathcal{U}_i\,.  
\end{align*}
We arrive at the following characterization.
\begin{lemma}\label{lem:exttrace}
We have
\begin{equation}
\ext(B) = \{\, \beta^{-1} h \otimes h\;|\; \|h\|_H=1\,\} \cup \{0 \},
\end{equation}
as well as
\begin{equation}
\mathcal{B} = \{\, \beta^{-1} h \otimes h\;|\;\|h\|_H \leq 1\,\}\,.
\end{equation}
\end{lemma}
A proof of Lemma \ref{lem:exttrace} can be found in the appendix, c.f. Section \ref{app:lem:exttrace}. A consequence of Lemma \ref{lem:exttrace} is that minimizers to~\eqref{def:rankproblem} always exhibit finite rank and share their eigenfunctions with the dual variable. Additionally, the next proposition provides optimality conditions for the problem \eqref{def:rankproblem}. The proof can be found in the appendix, c.f. Section \ref{app:prop:opttrace}.

\begin{proposition}\label{prop:opttrace}
Let~$\bar{\mathcal{U}} \in T(H)$,~$\bar{\mathcal{U}}  \geq 0$, be given and set
\begin{align*}
    \bar{P}=-K_*\nabla F(K \bar{\mathcal{U}}) \in K(H)\,.
\end{align*}
Then~$\bar{\mathcal{U}}$ is a minimizer of~\eqref{def:rankproblem} if and only if
\begin{align} \label{eq:optimalitytrace}
    \sigma^{\bar{P}}_1 \leq \beta, \quad \operatorname{Tr}( \bar{P} \bar{\mathcal{U}} )= \beta \operatorname{Tr}(\bar{\mathcal{U}})\,. 
\end{align}
In particular, if~$\sigma^{\bar{P}}_1 = \beta$ and~$\bar{N} \geq 1$ is the smallest index with~$\sigma^{\bar{P}}_{\bar{N}+1} < \sigma^{\bar{P}}_1$, this implies that a minimizer~$\bar{\mathcal{U}}$ of~\eqref{def:rankproblem} is of the form
\begin{align} \label{eq:representationtrace}
    \bar{\mathcal{U}}= \sum^{\bar{N}}_{i=1} \bar{\lambda}_i h^{\bar{P}}_i \otimes h^{\bar{P}}_i \quad \text{for some}~\bar{\lambda}_i \geq 0\,.
\end{align}
\end{proposition}

Now, denote by~$\mathcal{U}_k$ the~$k$-th iterate in Algorithm~\ref{alg:pdap} and let~$P_k$ be the corresponding dual variable. Then we immediately obtain 
\begin{align*}
    \argmax_{ \mathcal{V} \in \Ext B } \ \operatorname{Tr}(P_k \mathcal{V})= \left\{\,\beta^{-1} h \otimes h \;|\;P_k h=\sigma^{P_k}_1 h,~\|h\|_H=1 \,\right\}\,.
\end{align*}
In particular, selecting the new candidate point~$\widehat{\mathcal{V}}^{U}_k$ in Step 2 of Algorithm~\ref{alg:gcg} can be realized by computing one eigenfunction for the leading eigenvalue of~$P_k$ and~$\mathcal{U}_k$ has at most rank~$k$. Regarding the fast convergence of Algorithm~\ref{alg:pdap}, we recall that, in the best case, problem~\eqref{def:rankproblem} produces rank one solutions. In the following, we make a slightly stronger assumption. More in detail, denoting by~$\bar{P}$ the unique optimal dual variable for~\eqref{def:rankproblem}, i.e., there holds~$\bar{P}=-K_*\nabla F(K\bar{\mathcal{U}})$ for every minimizer~$\bar{\mathcal{U}}$ of~\eqref{def:rankproblem}, we assume that:
\begin{enumerate}[label=\textnormal{(D\arabic*)}]
    \item \label{ass:rankone} The first eigenvalue of~$\bar{P}$ is simple, i.e., $\beta=\sigma^{\bar{P}}_1 > \sigma^{\bar{P}}_2$\,.
\end{enumerate}
On the one hand, invoking~\eqref{eq:representationtrace}, this implies that the solution to~\eqref{def:rankproblem} has at most rank one. On the other hand, it also ensures the unique solvability of the linear problem induced by~$\bar{P}$ as well as a quadratic growth behavior with respect to the Hilbert-Schmidt norm as the next theorem shows. The proof can be found in the appendix, c.f. Section \ref{app:thm:rankprobquad}.
\begin{theorem}
\label{thm:rankprobquad}
Let Assumption~\ref{ass:rankone} hold. Then the unique optimal dual variable~$\bar{P}$ of~\eqref{def:rankproblem} satisfies
\begin{align} \label{eq:linearminoverrank}
    \argmax_{\mathcal{V} \in \Ext B} \ \langle \bar{P}, \mathcal{V} \rangle=\argmax_{\mathcal{V} \in B} \ \langle \bar{P}, \mathcal{V} \rangle= \{\beta^{-1} h^{\bar{P}}_1 \otimes h^{\bar{P}}_1\}\,.
\end{align}
In particular, this implies that the solution of~\eqref{def:rankproblem} is unique and of the form~$\bar{\mathcal{U}}=\bar{\lambda} h^{\bar{P}}_1 \otimes h^{\bar{P}}_1$,~$\bar{\lambda}\geq 0$. 
Moreover, setting~$\bar{\mathcal{U}}_1=\beta^{-1} h^{\bar{P}}_1 \otimes h^{\bar{P}}_1$, there are~$\tau, \kappa >0$ with
    \begin{gather} \label{eq:quadgrowthrank}
        \ynorm{K(\mathcal{U}-\bar{\mathcal{U}}_1)} \leq \tau \|\mathcal{U}-\bar{\mathcal{U}}_1\|_{\operatorname{HS}}, \\  1-\langle \bar{P}, \mathcal{U} \rangle  \geq \kappa \|\mathcal{U}-\bar{\mathcal{U}}_1\|^2_{\operatorname{HS}} \,,
    \end{gather}
    for all~~$\mathcal{U} \in B$ with $\operatorname{Tr}(\mathcal{U})=\beta^{-1}$.
\end{theorem}

Hence, Assumption~\ref{ass:rankone} implies Assumptions~\ref{ass:fast:F2} and~\ref{ass:fast:F5} for~$g(\cdot,\cdot)=\|\cdot-\cdot\|_{\text{HS}}$ and~$U_1=\Ext{B} \cap \mathcal{B}$. As a consequence, if~$F$ is strongly convex in the sense of Assumption~\ref{ass:fast:F1} as well as~$K\bar{\mathcal{U}}_1 \neq 0$ and~$\bar{\lambda}>0$, then Algorithm~\ref{alg:gcg} converges with an asymptotically linear rate.

\begin{remark} \label{rem:trace}
For Problem~\ref{def:rankproblem} and~$H \simeq \R^n$, Algorithm~\ref{alg:gcg} can be interpreted as a particular instance of a totally corrective conditional gradient method proposed for general matrix problems in~\cite{matrixnormzhang}. However, the presented linear convergence analysis in this earlier work does not cover trace-regularized problems since, even in the finite dimensional case,~$\operatorname{Ext}(B)$ is uncountable. Also in the finite dimensional setting, we want to mention~\cite{garber} where the author presents a conditional gradient method for the trace-constrained setting, i.e.,
\begin{align*}
    \min_{\mathcal{U} \in \R^{n \times n}}  \ F(K\mathcal{U}) \, \quad \text{s.t. } \,\,\, \mathcal{U} = \mathcal{U}^\top,~\operatorname{Tr}(\mathcal{U})\leq \alpha,~\mathcal{U} \geq 0\,,
\end{align*}
which eventually converges linearly under assumptions similar to~\ref{ass:rankone}. From this point of view, the results of this section can be seen as a first step to bridging the gap between the trace-regularized and the trace-constrained setting.
\end{remark}

\subsection{Minimum effort problems}\label{subsec:mineffort}
Let~$\Omega \subset \R^d$ be open and bounded, with $d \in \N$, $d \geq 1$. Consider the space~$\M:=L^\infty(\Omega)$ of essentially bounded Lebesgue measurable functions on~$\Omega$, which we equip with the usual essential supremum norm
\begin{align*}
\|u\|_{\infty}:= \esssup_{x \in \Om} |u(x)| \quad \text{ for all } \,\,  u\in L^\infty(\Omega)\,.
\end{align*}
We consider~\textit{minimum effort problems}, that is,
\begin{align} \label{def:propmineff}
\min_{u \in L^\infty(\Omega)} J(u)\,, \,\, J(u):=F(Ku)+ \alpha \|u\|_{\infty} \,, %
\end{align}
where $\alpha>0$ is a fixed parameter, see~\cite{clasonminimum}. This type of regularizer favors~\textit{binary} solutions, i.e., functions~$\bar{u}$ with~$\bar{u}(x) \in \{- \bar{\lambda}, \bar{\lambda}\}$ for a.e.~$x\in\Omega$ and some~$\bar{\lambda}\geq 0$. Problems of this form appear, e.g., in the optimal maneuvering of spacecrafts~\cite{burns}. Due to the nonsmoothness of the~$L^\infty$-norm, previous solution approaches~\cite{KKmineff,clasonminimum} were, e.g., based on a regularized semi-smooth Newton method for a bilinear reformulation of~\eqref{def:propmineff}. We now show that~FC-GCG yields a simple algorithm which eventually solves~\eqref{def:propmineff} without regularizing and/or reformulating the problem. For this purpose, we first note that~$L^1(\Omega)^* \simeq \mathcal{M}$ where~$L^1(\Omega)$ is equipped with the canonical norm~$\|\cdot\|_{1}$ and where the duality pairing is realized by
\begin{align*}
    \langle p,u \rangle= \int_{\Omega} p(x) u(x)~\mathrm{d}x \quad \text{for all}~p \in L^1(\Omega),~u \in L^\infty(\Omega)\,.
\end{align*}
If $K : L^\infty(\Omega) \rightarrow Y$ is a linear weak*-to-weak continuous operator and $F : Y \rightarrow \R$ is a convex fidelity satisfying \ref{ass:func1}, one can easily verify that the minimization problem \eqref{def:propmineff} satisfies Assumptions \ref{ass:func1}-\ref{ass:func3}.
As in the previous examples, we now characterize the extremal points of the unit ball of the regularizer, i.e., the set
\[
B:=\{\,u \in L^{\infty}(\Om) \;|\;\alpha \|u\|_\infty \leq 1\,\}\,.
\]
Subsequently, we use this characterization to describe minimizers of~\eqref{def:propmineff}. These are the subjects of the following Propositions~\ref{prop:extremalsoflinf}, ~\ref{prop:Optimalitymineffort}, the proofs of which can be found in Sections~\ref{app:prop:extremalsoflinf},~\ref{app:prop:Optimalitymineffort} of the Appendix, respectively.
\begin{proposition} \label{prop:extremalsoflinf}
For any $\alpha>0$ we have
\begin{equation} \label{lem:extremalsoflinf:1}
\operatorname{Ext}(B)= \left\{\,u\in L^\infty(\Omega)\, \colon \, \alpha|u|=1 \,\, \text{ a.e. in } \, \Om \,\right\}\,.
\end{equation}
Moreover there holds~$\Bb = B$.
\end{proposition}

\begin{proposition} \label{prop:Optimalitymineffort}
An element~$\bar{u}\in L^\infty(\Omega)$ is a minimizer of~\eqref{def:propmineff} if and only if the associated dual variable $\bar{p}:=-K_*\nabla F(K\bar{u})\in L^1(\Omega)$ satisfies
\begin{align} \label{eq:optimalitymineff}
\int_{\Omega} |\bar{p}(x)|~\de x \leq \alpha \,, \quad  \int_{\Omega} \bar{p}(x) \,\bar{u}(x)~\de x=\alpha \|\bar u\|_{\infty}\,.
\end{align} 
In particular, if~$\bar{u}\neq 0$ is a solution to~\eqref{def:propmineff}, then we have
\begin{align} \label{eq:representationlinf}
\int_{\Omega} |\bar{p}(x)|~\de x = \alpha, \quad    \bar{u}(x) \in \begin{cases}
    \{\bar{\lambda}\} & \text{if } \, \bar{p}(x)>0\,,  \\
    \{-\bar{\lambda}\} & \text{if } \,\bar{p}(x)<0\,, \\
    [-\bar\lambda, \bar\lambda] & \text{else}\,,
    \end{cases}
\end{align}
for some~$\bar{\lambda} > 0$ and~a.e.~$x \in \Omega$.
\end{proposition}

Now, let~$u_k$ denote the $k$-th iterate in Algorithm~\ref{alg:gcg} and denote by~$p_k=-K_* \nabla F(K u_k)$ the current dual variable. As in Proposition~\ref{prop:Optimalitymineffort}, we readily verify
\begin{align*}
 \max_{v \in \Ext B} \ \langle p_k, v  \rangle=\max_{v \in  B} \ \langle p_k, v  \rangle    =\alpha^{-1} \int_{\Omega} |p_k(x)|~\de x\,,
\end{align*}
where the maximum is realized for~$v^u_k= \operatorname{sign}(p_k)$. Hence, realizing Step 2 in Algorithm~\ref{alg:gcg} only requires the computation of the sign of~$p_k$. Next, assume that there is~$C>0$ with:
\begin{enumerate}[label=\textnormal{(E\arabic*)}]
\item \label{ass:linfinity1} For all~$y \in Y$ we have~$\|K_* y\|_\infty \leq C \|y\|_Y$.
\item\label{ass:linfinity2} Denoting by~$\mathcal{L}$ the $d$-dimensional Lebesgue measure, there holds
\begin{align*}
  \mathcal{L}\left( \left\{|\bar{p}|\leq \eps \right\} \right) := \mathcal{L} \left( \left\{\,x\in\Omega\, \colon \,-\eps \leq \bar{p}(x)\leq \eps\,\right\} \right) \leq C \eps \,, \quad \text{ for all } \,\, \e>0 \,.
\end{align*}
\end{enumerate}
We remark that Assumption~\ref{ass:linfinity2} has already been considered in the context of bang-bang optimal control, i.e. the $L^\infty$-norm constrained setting, see, e.g., \cite{cww}. Such assumption in particular implies that $\mathcal{L}(\{|\bar p|=0\})= 0$, and thus
\begin{equation} \label{eq:mineff_min}
\bar u := \frac{\bar{\lambda}}{\alpha} \left(  \rchi_{\{\bar p >0\}} - \rchi_{\{\bar p <0\}} \right)\,,
\end{equation}
for some~$\bar{\lambda} \geq 0$, thanks to Proposition~\ref{prop:Optimalitymineffort}.
Under these prerequisites, we show in the next theorem that Assumptions~$\ref{ass:fast:F2}$ and~$\ref{ass:fast:F5}$ are satisfied for the choice of~$g(u,v) \coloneqq \|u-v\|_1$ and~$\bar{U}_1 \coloneqq B$. 
\begin{theorem} \label{thm:ass3inLinf}
Let Assumptions~\ref{ass:linfinity1} and~\ref{ass:linfinity2} hold. Then we have
\begin{align}  \label{eq:characoflinmineff}
    \argmax_{v \in \Ext B} \ \langle \bar{p},v  \rangle=\argmax_{v \in  B} \ \langle \bar{p},v  \rangle= \{\alpha^{-1}\operatorname{sign}(\bar{p})\}\,.
\end{align}
Moreover, setting~$\bar{u}=\alpha^{-1} \operatorname{sign}(\bar{p})$, there holds
\begin{align*}
    \|K(u-\bar{u})\|_Y \leq \tau \|u-\bar{u}\|_1, \quad 1- \langle \bar{p}, u \rangle \geq \kappa \|u-\bar{u}\|^2_1 \,,
\end{align*}
for all~$u \in B$, and some~$\tau, \kappa >0$.
\end{theorem}
\begin{proof}
The statement in~\eqref{eq:characoflinmineff} follows immediately from~ $ \mathcal{L}(\{|\bar p|=0\})= 0$. The Lipschitz estimate on~$K$ is immediately deduced from Assumption~$\ref{ass:linfinity1}$, while the quadratic growth estimate can be established as in~\cite[Proposition~2.7]{cww}.
\end{proof}
Therefore, Assumptions \ref{ass:linfinity1}-\ref{ass:linfinity2} guarantee fast convergence for Algorithm \ref{alg:gcg}, provided the operators $F$ and $K$ are chosen to satisfy \ref{ass:fast:F1},\ref{ass:fast:F3}, and \ref{ass:fast:F4} holds for the unique minimizer of \eqref{def:propmineff}.

\subsection{Dynamic inverse problems regularized with the Benamou-Brenier energy} \label{sec:dynamic_inverse}
Consider the space $M(\Om)$ of Radon measures on the open bounded domain $\Om \subset \R^n$, $n \geq 1$, and denote by  
$X := [0,1] \times \Omega$ the time-space domain.
The aim of dynamic inverse problems is to reconstruct a time-dependent curve of measures $t \mapsto \rho_t$ with 
$\rho_t  \in M(\Om)$ and $t \in [0,1]$, starting 
from noisy ill-posed measurements $K \rho$, where $\rho=\de t \otimes \rho_t \in M(X)$ and $K \colon M(X) \to Y$ is a linear operator, with $Y$ a Hilbert space. In \cite{bredies2020optimal} it has been proposed to reconstruct the dynamic data $\rho$ by solving a variational inverse problem regularized with a coercive version of the Benamou-Brenier energy $J_{\alpha,\beta}$, allowing for a correlation between the measurements at different time instants (see also \cite{brune20104d}). Such variational models have been applied, for instance, to dynamic cell imaging in PET \cite{schmitzer2019dynamic}, to 4d image reconstruction in nanoscopy \cite{brune20104d} and to particle image velocimetry methods \cite{saumier2015optimal}. A general variant of the problem considered in \cite{bredies2020optimal} can be formulated in our setting by considering
\begin{align}\label{eq:bb}
\min_{(\rho,m) \in \M } \ F(K\rho) + J_{\alpha,\beta}(\rho,m)\,,
\end{align}
where $\M:= M(X) \times M(X; \R^d)$, $K : M(X) \rightarrow Y$ is weak*-to-weak continuous and such that \ref{ass:func3} holds, while $F : Y \rightarrow \R$ is a convex fidelity term satisfying \ref{ass:func1}. 
The Benamou-Brenier energy \cite{bb} can be regarded as a dynamic version of optimal transport, and is defined as follows. Setting 
$\Phi \colon \R \times \R^d \to [0,\infty]$ as
\begin{equation} \label{intro Psi}
\Phi(t,x):=\frac{|x|^2}{2t} \,\,\, \text{ if }\, t>0\,, \quad  
\Phi(t,x):= 0 \,\,\, \text{ if }\, t=|x|=0\,, 
\end{equation}
and $\Phi:=\infty$ otherwise, 
the Benamou-Brenier energy is defined for $(\rho,m)  \in \M $ by the formula
\begin{equation} \label{intro:BB_2}
J (\rho,m):= \int_X \Phi \left( \frac{\de\rho}{\de\sigma}, \frac{\de m}{\de\sigma} \right) \,\de \sigma \,,
\end{equation}
where $\sigma \in M^+(X)$ is an arbitrary measure such that $\rho,|m| \ll \sigma$. We refer the reader to \cite{ags,sant} for more details. Following \cite{bredies2020optimal}, the regularizer $J_{\alpha,\beta}$ is then defined by
\begin{equation} \label{intro:BB_coercive}
J_{\alpha,\beta} (\rho,m) := \beta J(\rho,m) + \alpha \nor{\rho}_{M(X)}+ I_{\mathcal{D}}(\rho,m)\,,
\end{equation}
where $\alpha,\beta >0$, and $\mathcal{D}$ is the set of pairs $(\rho,m) \in \M$ satisfying the continuity equation $\partial_t \rho + \div m = 0$ in $X$ in the weak sense. 
The functional $J_{\alpha,\beta}$ is convex, weak* lower semicontinuous, positively one-homogeneous and satisfies \ref{ass:func2}, see~\cite[Lemma 4]{extremal}. Denote by $B$ the unit ball of $J_{\alpha,\beta}$. In \cite[Theorem 6]{extremal} it has been shown that 
\[
\operatorname{Ext}(B)=\{(0,0)\} \cup \mathfrak{C}_{\alpha,\beta}\,,
\]
where $\mathfrak{C}_{\alpha,\beta}$ is the set of measures concentrated on absolutely continuous curves in $\Om$, i.e., pairs $(\rho_\gamma,m_\gamma)$ satisfying
\begin{align}\label{eq:extben}
\rho_\gamma = a_\gamma \, \de t \otimes \delta_{\gamma(t)}, \quad m_\gamma = \dot{\gamma}(t) \rho_\gamma \,, \quad a_\gamma = \left( \frac{\beta}{2} \int_0^1|\dot{\gamma}(t)|^2\, \de t + \alpha \right)^{-1}\,,
\end{align}
for some $\gamma \colon [0,1] \rightarrow \Omega$ an absolutely continuous curve with weak derivative in $L^2$. 
Therefore, it is possible to apply the FC-GCG method of Algorithm \ref{alg:gcg} for computing solutions to  \eqref{eq:bb}. In this setting the iterates are of the form $u_k=(\rho_k,m_k)$ with 
\[
\rho_k= \sum_{i=1}^{N_k} \lambda_i^k \rho_{\gamma_i^k}\,, \quad  m_k = \sum_{i=1}^{N_k} \lambda_i^k m_{\gamma_i^k} \,,
\]
with $\lambda_i^k \geq 0$ and $\gamma_i^k \colon [0,1] \to \Omega$ absolutely continuous. 
Moreover it can be shown, see \cite{DGCG}, that Step~$2$ in Algorithm \ref{alg:gcg} is equivalent to solving
\begin{align}\label{eq:insertionBB}
    \hat{\gamma} \in \argmax_{\gamma \colon [0,1] \to \Omega} \ a_\gamma\int_0^1 p_k(t, \gamma(t))\, \de t\,,
\end{align}
where $p_k =   - K_* \nabla F(K \rho_k) \in C(X)$ is the dual variable at the $k$-th iteration. The new point inserted is then $\widehat{v}_k=(\rho_{\hat\gamma},m_{\hat\gamma})$. 
Thanks to the validity of Assumption \ref{ass:func1}-\ref{ass:func3}, Theorem \ref{thm:convofcondu} guarantees the sublinear convergence of Algorithm \ref{alg:gcg}. On the other hand, the verification of hypotheses \ref{ass:fast:F1}-\ref{ass:fast:F5}, necessary for ensuring fast convergence of Algorithm~\ref{alg:gcg}, is non-trivial and is left to future work. 
We remark that an implementable version of Algorithm~\ref{alg:gcg} for solving \eqref{eq:bb} under specific choices of the fidelity term $F$ and the operator $K$ has been recently proposed in 
\cite{DGCG} (see also \cite{duval2021dynamical}). We refer to these papers for more details about the practical implementation and the modifications needed to deal with time dependent measurement operators. 
Similarly to \cite{DGCG}, solving \eqref{eq:descr_insertion} amounts to computing an absolutely continuous curve $\gamma$ by solving \eqref{eq:insertionBB} at every iteration of the algorithm. This is a challenging non-concave variational problem in the space of curves. In \cite{DGCG}, the authors proposed to solve \eqref{eq:insertionBB} by a multistart gradient descent approach in the space of curves. The initialization curves are chosen to be a linear interpolation of randomly generated points $\{x^h\}$ in $\Omega$. Moreover several heuristic rules are employed to optimally select $\{x^h\}$ and to reduce the computational time of the multistart gradient descent routine. Together with further acceleration strategies, this method is shown to be computationally feasible and very accurate for the task of tracking several dynamic sources in presence of high noise and severe spatial undersampling.  
In \cite{duval2021dynamical}, the authors proposed to speed up the algorithm in \cite{DGCG} by considering inexact subproblems \eqref{eq:descr_insertion} and solving them using known algorithms for computing shortest paths on directed acyclic graphs.

\section{A lifting to a space of measures} \label{sec:lifting}

Our approach to proving Theorem~\ref{thm:convofcondu} and Theorem~\ref{thm:fastconvpreview} relies on the observation of~\textit{Choquet's Theorem}. This classical result allows us to prove that for every~$u \in \dom(\G)$ there exists a positive measure~$\mu$ concentrated on~$\Ext B$ such that~$\G(u)=\|\mu\|_{M(\mathcal{B})}$ and
\begin{align} \label{eq:represalgo}
\langle p, u \rangle= \int_{\mathcal{B}} \langle p, v \rangle~\mathrm{d} \mu(v) \quad \text{ for all }~~~ p \in \Cc\,,
\end{align}
see Proposition \ref{prop:I}.
Motivated by this observation, we study the auxiliary problem \eqref{def:sparseprob}, which we remind the reader is of the form
\[
\inf_{\mu \in M^+(\mathcal{B})} \left\lbrack F(\mathcal{K}\mu) +\|\mu\|_{M(\mathcal{B})} \right \rbrack \,,
\]
where the infimum is taken over~$M^+(\Bb)$, the cone of positive measures on~$\Bb$. Note that the forward operator is replaced by a ``lifted'' mapping~$\K \colon M(\Bb) \to Y$ which satisfies~$\K\mu=Ku$ whenever~\eqref{eq:represalgo} holds for~$\mu \in M^+(\Bb)$ and~$u\in \M$, see Proposition \ref{prop:def_K}. Moreover, the role of the non-smooth regularizer $\G$ is now played by the total variation norm. It turns out (see Section~\ref{subsec:eqproblem}) that problems~\eqref{def:minprob} and~\eqref{def:sparseprob} are equivalent in the sense that every minimizer to~\eqref{def:minprob} can be converted to a solution of~\eqref{def:sparseprob} (and vice versa). Subsequently, in Section~\ref{subsec:PDAP}, we propose an extension of the Primal-Dual-Active-Point method from~\cite{pieper} to compute a solution of~\eqref{def:sparseprob}. This is described in Algorithm \ref{alg:pdap}. Finally, using the results of Section~\ref{subsec:eqproblem}, we show that Algorithm \ref{alg:pdap} and Algorithm~\ref{alg:gcg} are equivalent, setting the ground to prove Theorem~\ref{thm:convofcondu} and Theorem~\ref{thm:fastconvpreview} by means of convergence results for Algorithm \ref{alg:pdap}.

\subsection{Measure theory notation} \label{subsec:measure_notation}
Let $C(\mathcal{B})$ denote the vector space of real valued bounded continuous functions over $\mathcal{B}$, which we equip with the supremum norm
\[
\|P\|_{C(\mathcal{B})}:=\max_{u \in \mathcal{B}}|P(u)|\,,
\]
making it a Banach space. %
 Following the definitions of \cite{afp}, we denote by $\Pi$ the $\sigma$-algebra of Borel sets on~$\mathcal{B}$ formed with respect to the topology induced by~$d_{\mathcal{B}}$. A finite Radon measure on $\Bb$ is a $\sigma$-additive map $\mu \colon \Pi \to \R$. We call~$\mu$ positive if~$\mu(\Pi)\subset [0,\infty)$. Given a finite Radon measure $\mu$ its total variation measure $|\mu|$ is defined by 
\[
|\mu|(E):=\sup \left.\left\{ \sum_{i=1}^\infty |\mu(E_i)| \, \right\vert \, \bigcup_{i=1}^\infty E_i = E, \,\, E_i \in \Pi \, \text{ pairwise disjoint} \right\}  \,, 
\]
for all $E \in \Pi$. 
Note that $|\mu|$ is always a positive finite Radon measure.  The set of finite Radon measures over $\Bb$ is a vector space denoted by $M(\Bb)$, which becomes a Banach space when endowed with the total variation norm  
\[
\|\mu\|_{M(\mathcal{B})}:=|\mu|(\mathcal{B})\,.
\]
We recall that $M(\Bb)$ is the dual space of $C(\Bb)$ with respect to $\norm{\cdot}_{C(\Bb)}$. The duality pairing will be denoted by $\sm{\cdot}{\cdot}$. 
The set~$M^+(\mathcal{B})$ of all positive finite Radon measures on~$\mathcal{B}$ forms a weak* closed cone. For $\mu \in M(\mathcal{B})$ and $E \in \Pi$, we say that $\mu$ is concentrated on $E$ if $\mu (\mathcal{B} \setminus E)=0$. For a set $E' \in \Pi$ we define the restriction of $\mu$ to $E'$ as the measure $\mu \zak E'$ such that $\mu \zak E'(E)=\mu(E' \cap E)$ for all $E \in \Pi$. Finally, the support of a measure $\mu \in M(\mathcal{B})$, denoted by $\supp \mu$, is the closure of the set of all points 
$v \in \mathcal{B}$ such that $|\mu|(U)>0 $ for all neighbourhoods $U$ of $v$.

\subsection{An equivalent problem} \label{subsec:eqproblem}

We require the following definition.

\begin{definition} \label{def:represenation}
We say that a measure~$\mu \in M^+(\mathcal{B})$ \textit{represents} $u \in \M$ if 
\begin{equation} \label{eq:choquet}
\langle p ,u \rangle=\int_{\mathcal{B}}  \langle p ,v \rangle \, \mathrm{d}\mu(v)\,, \, \,\,\, \text{ for all } \,\,\, p \in \Cc \,.
\end{equation}
An element $u \in \M$ such that \eqref{eq:choquet} holds is also called the \textit{(weak) barycenter} of $\mu$ in $\Bb$.
\end{definition}

It turns out that each~$u \in \dom(\G)$ is the barycenter of at least one measure~$\mu \in M^+(\Bb)$. %
 
\begin{proposition} \label{prop:I}
There exists a surjective linear map $\I \colon M^+(\mathcal{B}) \to  \operatorname{dom}(\G)$ such that
\begin{equation} \label{def:I}
\langle p, \I(\mu) \rangle = \int_{\mathcal{B}} \langle p, v \rangle~\de\mu (v) \,, \, \,\,\, \text{ for all } \,\,\, p \in \Cc \,.
\end{equation}
In particular, $\mu$ represents $\I(\mu)$ in the sense of \eqref{eq:choquet} and it holds:
\begin{itemize}
\item[i)]    For any $\mu \in  M^+(\mathcal{B})$ we have $ \G(\I(\mu)) \leq \nor{\mu}_{M(\mathcal{B})}$.
\item[ii)] For any $u \in \operatorname{dom}(\G)$ there exists $\mu \in M^+(\mathcal{B})$ concentrated on $\Ext{B}$ such that
\[
u=\I(\mu) \quad \text{ and } \quad \G(u) = \nor{\mu}_{M(\mathcal{B})}.
\]
\item[iii)] $\mathcal{I}$ is weak*-to-weak* continuous. 
\end{itemize}
\end{proposition}
For a proof of the above statement we refer the reader to Appendix \ref{app:proofs_sec1}. Here we only mention that existence of the map $\I$ is easily obtained by the theory of weak* integration, while surjectivity is a consequence of the classic Choquet's Theorem \ref{thm:real_choquet}.

Thus, instead of solving~\eqref{def:minprob} directly we can equivalently determine a measure~$\bar{\mu} \in M^+(\Bb)$ which represents any of its minimizers. Again, this can be done by solving a suitable minimization problem. To make this idea more rigorous we argue the existence
of a linear continuous operator $\mathcal{K} \colon M(\mathcal{B}) \to Y$ which agrees with $K$ on measures representing points of $\mathcal{M}$. Again, we postpone the proof of such statement to Appendix \ref{app:proofs_sec1}. 
\begin{proposition} \label{prop:def_K}
There exists a linear continuous operator $\mathcal{K} \colon M(\mathcal{B}) \to Y$ such that
\begin{equation} \label{eq:def_K}
(\mathcal{K}\mu, y )_Y = \int_{\mathcal{B}}  (Kv,y)_Y \, \de\mu(v) \,, \, \,\,\, \text{ for all } \,\,\, y \in Y\,,\,\, \mu \in M(\mathcal{B})\,.
\end{equation}
Moreover $\mathcal{K}$ satisfies the following properties:
\begin{itemize}
\item[i)] The norm of $\mathcal{K}$ is such that
\[
\nor{\mathcal{K}}_{\mathcal{L}(M(\mathcal{B}),Y)} \leq C \|K\|_{\mathcal{L}(\M,Y)}\,,  \quad C:=\sup_{v \in \mathcal{B}} \,\norm{v}_{\M} \,.
\]
\item[ii)] For every $\mu\in M^+(\mathcal{B})$ representing $u \in \mathcal{M}$, there holds $\mathcal{K} \mu = Ku$.
\item[iii)] $\mathcal{K}$ is weak*-to-strong continuous.
\item[iv)] $\mathcal{K}$ is the adjoint operator of $\mathcal{K}_* \colon Y \to C(\mathcal{B})$, where
\[
\lbrack \mathcal{K}_* y \rbrack (v) := \langle K_*y,v\rangle \quad \text{ for all } \quad  v \in \mathcal{B}\,.
\]
\end{itemize}
If in addition $Y$ is separable, then \eqref{eq:def_K} holds in the strong sense, that is, 
\begin{equation} \label{eq:def_K_strong}
\mathcal{K}\mu = \int_{\mathcal{B}}  Kv \, \de\mu(v) \,, 
\end{equation}
for all $\mu \in M(\mathcal{B})$, where the right hand side integral is in the Bochner sense. 
\end{proposition}

We are now in position to investigate the announced equivalence of \eqref{def:minprob} and the sparse minimization problem~\eqref{def:sparseprob}. To this end, define
\begin{align*}
 j(\mu):=  F(\mathcal{K}\mu) +\|\mu\|_{M(\mathcal{B})}\,,
\end{align*}
where $\mathcal{K} \colon M(\mathcal{B}) \to Y$ is the operator from Proposition \ref{prop:def_K}. The following theorem establishes existence of minimizers for \eqref{def:sparseprob} and clarifies the connection between~\eqref{def:minprob} and~\eqref{def:sparseprob}, which is given in terms of the map $\I$ introduced in Proposition \ref{prop:I}.

\begin{theorem} \label{prop:equivalence}
The functional $j$ is weak* lower semicontinuous, has weak* compact sublevels and \eqref{def:sparseprob} admits at least a solution. In addition, minimizers of \eqref{def:minprob} and \eqref{def:sparseprob} enjoy the following relationship: 
\begin{itemize}
\item[i)] If~$\bar{u}\in \M$ is a minimizer to~\eqref{def:minprob}, there exists a minimizer~$\bar{\mu} \in M^+(\mathcal{B}) $ to~\eqref{def:sparseprob} such that $\bar{u}=\I(\bar\mu)$. 
\item[ii)]    Vice versa, if~$\bar{\mu} \in M^+(\Bb)$ is an optimal solution to~\eqref{def:sparseprob} then~$\bar{u}:=\I(\bar{\mu})$ minimizes in~\eqref{def:minprob}. 
\end{itemize}
In particular we have that 
\begin{equation} \label{prop:equvalence:equal}
\min_{u \in \M} J(u) = \min_{\mu \in M^+(\Bb)} j(\mu) \,.
\end{equation}
\end{theorem}

\begin{proof}
First recall that $\mathcal{K}$ is weak*-to-strong continuous thanks to Proposition \ref{prop:def_K}. Given that $F$ is continuous, we then infer weak* lower semicontinuity of $j$. As $F$ is bounded from below, we immediately have that $j$ is bounded from below and its sublevels are weak* compact, concluding the existence of minimizers to \eqref{def:sparseprob} by the direct method.
 
We pass to the proof of i). Assume that $\bar{u}$ is a minimizer to \eqref{def:minprob}, so that $\bar{u} \in \operatorname{dom}(\G)$. According to ii) in Proposition~\ref{prop:I} there exists~$\bar{\mu} \in M^+(\mathcal{B})$ such that $\I(\bar \mu)=\bar{u}$ and $\nor{\bar\mu}_{M(\mathcal{B})}=\G(\bar u)$. Let~$\mu \in M^+(\mathcal{B})$ be arbitrary and set~$u:=\I(\mu)$. The same proposition yields that $u \in \operatorname{dom} (\G)$, $\mu$ represents $u$ and $\G(u) \leq \nor{\mu}_{M(\mathcal{B})}$. Finally, point ii) in Proposition \ref{prop:def_K} guarantees $\mathcal{K}\mu = K u$ and $\mathcal{K}\bar \mu = K \bar u$.  Thus
\begin{equation} \label{eq:equivalence:1}
j(\bar{\mu})=J(\bar{u}) \leq J(u) \leq j(\mu)\,,
\end{equation}
where the first inequality follows from the optimality of~$\bar{u}$. This proves i).
Second, let~$\bar{\mu} \in M^+(\mathcal{B})$ be a solution to~\eqref{def:sparseprob} and~$\bar{u}:=\I(\bar{\mu})$. Thus, by Proposition \ref{prop:I}, we have that $\bar u \in \operatorname{dom}(\G)$, $\bar{\mu}$ represents $\bar u$ and $\G(\bar u) \leq \nor{\bar \mu}_{M(\mathcal{B})}$. Moreover, let~$u \in \operatorname{dom}(\G)$. By point ii) in Proposition \ref{prop:I}, we get $\mu \in M^+(\mathcal{B})$ representing~$u$ and such that~$\|\mu\|_{M(\mathcal{B})}=\G(u)$. Again, $\mathcal{K}\mu = K u$, $\mathcal{K}\bar \mu = K \bar u$ and we conclude the proof of ii) noting that
\begin{equation} \label{eq:equivalence:2}
J(\bar{u}) \leq j(\bar{\mu}) \leq j(\mu)=J(u)\,.
\end{equation}
The final part of the statement follows from \eqref{eq:equivalence:1}-\eqref{eq:equivalence:2}.
\end{proof}

\subsection{Optimality conditions} \label{subsec:optimality}
 
In this section we establish the relationship between the dual variables of the problems \eqref{def:minprob} and \eqref{def:sparseprob}. 
Moreover we characterize optimality conditions for \eqref{def:sparseprob}.

\begin{proposition} \label{prop:existofextrem}
Let~$\mu \in M^+(\mathcal{B})$ be given and set $u:=\I( \mu)$. Define the corresponding dual variables 
$P:=-\mathcal{K}_*\nabla F(\mathcal{K}\mu)\in C(\mathcal{B})$ and $p:=-K_* \nabla F(K  u) \in \Cc$.
Then
\begin{equation}\label{eq:dual_var_rel}
P(v) = \langle p, v \rangle   \,\,\, \text{ for all } \,\,\, v \in \mathcal{B} \,.
\end{equation}
Moreover, there exists~$\bar{v} \in \operatorname{Ext}(B)$ such that 
\begin{equation} \label{eq:existofextrem}
P(\bar{v})=\max_{v \in \mathcal{B}} P(v) = \max_{v \in \mathcal{B}} \, \langle p, v \rangle\,.
\end{equation}
\end{proposition}

\begin{proof}
Let $v \in \mathcal{B}$ be arbitrary. It is immediate to check that $\delta_{v} \in M^+(\mathcal{B})$ represents $v$ in the sense of \eqref{eq:choquet}. Moreover $\mu$ represents $u$ by Proposition \ref{prop:I}. Applying point ii) in Proposition \ref{prop:def_K} yields $\mathcal{K}\delta_{v}= K  v$ and $\mathcal{K} \mu= K  u$. Thus, by iv) in Proposition \ref{prop:def_K},
\[
\begin{aligned}
\langle p,v \rangle & = -\langle K_*\nabla F(Ku),v \rangle  =-( \nabla F(Ku),Kv)_Y
\\ & = -( \nabla F(\mathcal{K}\mu),\mathcal{K}\delta_{v})_Y 
 = -\sm{\mathcal{K}_* \nabla F(\mathcal{K}\mu)}{\delta_{v}}= P(v)\,,
\end{aligned}
\]
yielding \eqref{eq:dual_var_rel}.  The statement of \eqref{eq:existofextrem} now follows from Lemma~\ref{lem:existoflinearized}.
\end{proof}

\begin{theorem} \label{thm:optcondition}
A measure~$\bar{\mu} \in M^+(\mathcal{B})$ is a solution of~\eqref{def:sparseprob} if and only if the dual variable $\bar{P}:=- \mathcal{K}_*\nabla F(\mathcal{K} \bar{\mu}) \in C(\mathcal{B})$ satisfies
\begin{align}\label{eq:optconditionstatement}
\bar{P}(v) \leq 1  \,\,\, \text{ for all } \,\, v \in \mathcal{B}\,, \quad  \sm{\bar{P}}{ \bar{\mu}}= \|\bar{\mu}\|_{M(\mathcal{B})}\,.
\end{align}
\end{theorem}

\begin{proof}
Define the mapping~$\Phi \colon M(\Bb) \to [0,+\infty]$ by
$\Phi(\mu)= \|\mu\|_{M(\Bb)}+I_{M^+(\Bb)}(\mu)$ where $I_{M^+(\Bb)}$ denotes the convex indicator function of~$M^+(\Bb)$. Note that the map~$f(\mu):=F(\mathcal{K}\mu)$ is G\^{a}teaux-differentiable with
$f'(\mu)(\nu)= \langle \K_* \nabla F(\K\mu), \nu \rangle$ for all $\mu,\nu \in \mathcal{M}(\mathcal{B})$. Then,
by standard subdifferential calculus \cite{brediesbook}, we have that~$\bar{\mu}$ is a solution to~\eqref{def:sparseprob} if and only if the dual variable~$\bar{P}$ satisfies $\bar{P} \in \partial \Phi(\bar{\mu})$  where $\partial \Phi(\bar{\mu})$ is the convex subdifferential of~$\Phi$ at~$\bar{\mu}$. Since~$\Phi$ is positively~$1$-homogeneous, it can be easily checked that this inclusion is equivalent to~\eqref{eq:optconditionstatement}.
\end{proof}

\subsection{Uniqueness of solutions}

As a consequence of Theorem \ref{prop:equivalence} and Proposition \ref{prop:uniqueu} we have the following uniqueness result for solutions to \eqref{def:sparseprob}, under the stronger 
Assumptions~\ref{ass:fast:F2}-\ref{ass:fast:F3}. 

\begin{proposition} \label{prop:convimpliconmu}
Let Assumptions~\ref{ass:func1}-\ref{ass:func3} and \ref{ass:fast:F2}-\ref{ass:fast:F3} hold. Then~\eqref{def:minprob} and~\eqref{def:sparseprob} admit unique solutions given by, respectively, 
\[
\optu = \sum^{N}_{i=1} \bar{\lambda}_i \bar{u}_i\,, \qquad \bar{\mu}= \sum^{N}_{i=1} \bar{\lambda}_i \delta_{\bar{u}_i}\,,
\]
where the points $\{\bar{u}_i\}_{i=1}^N \subset \operatorname{Ext}(B)$ are as in \ref{ass:fast:F2} and $\bar{\lambda}_i \geq 0$.  
\end{proposition}

\begin{proof}
By Proposition~\ref{prop:uniqueu} we have that \eqref{def:minprob} admits a unique solution $\bar u$, which is of the form $\optu = \sum^{N}_{i=1} \bar{\lambda}_i \bar{u}_i$ for some $\bar{\lambda}_i \geq 0$. Let $\bar \mu \in M^+(\Bb)$ be a solution to \eqref{def:sparseprob}, which exists by Theorem~\ref{prop:equivalence}.
From the optimality conditions \eqref{eq:optconditionstatement} in  Theorem~\ref{thm:optcondition}, one can easily verify that 
\[
\supp \bar \mu \subset \{ v \in \Bb \; | \; \bar P(v)=1\}\,.
\]
Hence, by Assumption~\ref{ass:fast:F2} and Proposition \ref{prop:existofextrem}, there exist $\bar{\sigma}_i \geq 0$ such that  $\bar \mu =\sum^N_{i=1}\bar{ \sigma}_i \delta_{\bar{u}_i}$. By \eqref{def:I} we have  $\mathcal{I}( \delta_{u} )=u$ for all $u \in \Bb$. As $\mathcal{I}$ is linear, we then conclude $\mathcal{I}(\bar \mu)=\sum_{i=1}^N \bar{\sigma}_i  \bar{u}_i$. On the other hand, by Theorem~\ref{prop:equivalence}~ii), we know that 
$\mathcal{I}(\bar{\mu})$ minimizes in \eqref{def:minprob}. 
Thus $\mathcal{I}(\bar{\mu})=\bar{u}$, given that $\bar{u}$ is the unique solution of \eqref{def:minprob}. 
We have then shown
$\sum_{i=1}^N (\bar{\lambda}_i - \bar{\sigma}_i) \bar{u}_i = 0$.
By applying the linear operator $K$ to such identity, and invoking \ref{ass:fast:F3}, we infer $\bar{\lambda}_i = \bar{\sigma}_i$. Therefore $\bar{\mu}=\sum_{i=1}^N \bar{\lambda}_i \delta_{\bar{u}_i}$. As $\bar{\mu}$ is an arbitrary minimizer of \eqref{def:sparseprob}, the thesis is achieved.    
\end{proof}

\subsection{\texorpdfstring{A Primal-Dual-Active-Point method for $(\mathcal{P}_{M^+})$}{A Primal-Dual-Active-Point method for P\_M+}} \label{subsec:PDAP}
In the following we describe a variant of the Primal-Dual-Active-Point strategy (PDAP) from~\cite{pieper} for the solution of~\eqref{def:sparseprob}. 
The latter is a fully-corrective version of a generalized conditional gradient method (also known as Frank-Wolfe algorithm) for solving convex minimization problems over spaces of measures supported on subsets of the euclidean space. In this section we generalize such procedure to~\eqref{def:sparseprob} and discuss its connection to Algorithm~\ref{alg:gcg}.
Similarly to~\cite{pieper}, our proposed PDAP method alternates between the update of an active set~$\mathcal{A}^\mu_k=\{u^k_i\}^{N_k}_{i=1}$ contained in $\Ext B$, and of a sparse iterate~$\mu_k \in M^+(\mathcal{B})$ supported on~$\mathcal{A}^\mu_k$, i.e., 
\begin{equation} \label{eq:crnt_itrt}
\mu_k = \sum_{i=1}^{N_k} \lambda_{i}^k \, \delta_{u^k_i}\,,
\end{equation}
for some $\lambda_i^k \geq 0$. 
We now provide a short description of the individual steps of this method and we summarize them in Algorithm~\ref{alg:pdap} below.
Given the current iterate~$\mu_k$ of the form \eqref{eq:crnt_itrt}, we first compute the corresponding dual variable~$P_k=-\K_* \nabla F(\K\mu_k) \in C(\Bb)$ and enrich the active set~$\mathcal{A}^\mu_k$ by adding to it a global maximizer~$\{\widehat{v}^{\mu}_k\}$ of~$P_k$ over~$\Ext B$, i.e., we set
\begin{align*}
\mathcal{A}^{\mu,+}_k =\{u^k_i\}^{N_k}_{i=1} \cup \{\widehat{v}^{\mu}_k\}\,, \quad ~\widehat{v}^\mu_k \in \argmax_{v \in \Ext B } P_k (v)\,.
\end{align*}
Using Proposition~\ref{prop:existofextrem} we note that this update step is equivalent to maximizing a linear functional over~$\Ext B$. This is the content of the next lemma, whose proof is an immediate consequence of Proposition~\ref{prop:existofextrem}, and is hence omitted. 
\begin{lemma} \label{lem:linearizedmeasure}
Let~$\mu_k$ be as in \eqref{eq:crnt_itrt} and set~$u_k = \I(\mu_k)$. Define~$P_k:=-\K_* \nabla F(\K \mu_k) \in C(\Bb)$ and~$p_k:= -K_* \nabla F(Ku_k) \in \Cc$. Then, there holds
\begin{align*}
\argmax_{v \in \Ext B} P_k(v)= \argmax_{v \in \Ext B}\, \langle p_k, v \rangle\,.
\end{align*}
\end{lemma}
Subsequently, setting~$N^+_k \coloneqq N_k+1$ and~$u^k_{N^+_k} \coloneqq \widehat{v}^\mu_k$, we find the next iterate~$\mu_{k+1}$ by solving 
\begin{align} \label{def:pdapsubprob}
\min_{\mu \in M^+(\mathcal{A}^{\mu,+}_k)} \ \left \lbrack F(\mathcal{K}\mu) +\|\mu\|_{M(\mathcal{B})} \right \rbrack,
\end{align}
where the whole cone~$M^+(\Bb)$ is replaced by the restricted subset
\begin{align*}
M^+(\mathcal{A}^{\mu,+}_k):= \left.\left\{\,\sum^{N^+_k}_{i=1} \lambda_i \delta_{u^k_i}\;\right\vert\;\lambda \in \R_+^{N^+_k}\,\right\} \subset M^+(\Bb)\,,
\end{align*}
see Step $5$ of Algorithm \ref{alg:pdap}. The following lemma compares the update obtained by solving \eqref{def:pdapsubprob} to the finite dimensional minimization problem \eqref{eq:subprobcoeffs} in Step $5$ of Algorithm~\ref{alg:gcg}.
\begin{lemma}\label{lem:finitedimonmeasures} 
A measure~$\widehat{\mu} \in M^+(\mathcal{A}^{\mu,+}_k)$ is a solution to~\eqref{def:pdapsubprob}
if and only if
\[
\widehat{\mu}=\sum^{N^+_k}_{i=1} \widehat{\lambda}_i  \delta_{u^k_i} \,,
\]
where~$\widehat{\lambda}\in \R_+^{N^+_k}$ is a minimizer of the finite dimensional minimization problem \eqref{eq:subprobcoeffs}.
\end{lemma}
\begin{proof}
Let~$\mu\in M^+(\mathcal{A}^{\mu,+}_k)$, so that 
there exists at least one~$\lambda^\mu \in \R_+^{N^+_k}$ such that~$\mu=\sum^{N^+_k}_{i=1} \lambda^\mu_i  \delta_{u^k_i}$. Noting that~$u^k_i =\I(\delta_{u^k_i})$, from Proposition~\ref{prop:def_K} we get that 
$\K \delta_{u^k_i}=K u^k_i$. Therefore
\begin{align*}
j(\mu)= F(\K \mu)+\|\mu\|_{M(\Bb)}=F\left(\sum^{N^+_k}_{i=1} \lambda^\mu_iKu^k_i\right)+\sum^{N^+_k}_{i=1} \lambda^\mu_i \,,
\end{align*}
from which the characterization of minimizers~$\widehat{\mu}$ to~\eqref{def:pdapsubprob} readily follows.
\end{proof}
Finally, see Step $6$ of Algorithm \ref{alg:pdap}, the active set is truncated by choosing~$\mathcal{A}^{\mu}_{k+1}$ as the~\textit{support} of~$\mu_{k+1}$, that is,
\begin{align*}
 \mathcal{A}^{\mu}_{k+1}:= \supp \mu_{k+1} 
 = \mathcal{A}^{\mu,+}_k \setminus \left\{ \,u^k_i \in \mathcal{A}^{\mu,+}_k \;|\;\lambda^{k+1}_i =0\,\right\}.
\end{align*}
The method is summarized in Algorithm~\ref{alg:pdap}.
\begin{algorithm}[tbh]\caption{PDAP for~\eqref{def:sparseprob}}
\begin{algorithmic} \label{alg:pdap}
\STATE 1. Let~$\mu_0= \sum^{N_0}_{i=1} \lambda^0_i \delta_{u^0_i}$,~$\lambda^0_i >0$,~$\mathcal{A}^\mu_0= \{u^0_i\}^{N_0}_{i=1}\subset \Ext B$.
\FOR{$k=0,1,2,\dots$}
\STATE 2. Given~$\mathcal{A}^\mu_k=\{u^k_i\}^{N_k}_{i=1} \subset \operatorname{Ext}(B) $ and~$\mu_k \in M^+(\mathcal{A}^\mu_k)$, calculate $P_k$ and $\widehat{v}^\mu_k$ with
\[
P_k=-\K_*\nabla F(\K \mu_k)\,, \quad \widehat{v}^\mu_k \in \argmax_{v\in \operatorname{Ext}(B)} P_k(v) \,.
\]
\IF{$P_k(\widehat{v}^\mu_k)\leq 1~\text{and}~k \geq 1$ %
}
\STATE 3. Terminate with~$\bar{\mu}= \mu_k$ a minimizer to \eqref{def:sparseprob}.
\ENDIF
\STATE 4. Update~$N^+_k=N_k +1$,~$u^k_{N^+_k}=\widehat{v}^\mu_k$ and~$\mathcal{A}^{\mu,+}_k= \mathcal{A}^\mu_k \cup \{\widehat{v}^\mu_k\}$.
\STATE 5. Determine~$\mu_{k+1}$ with
\[
\mu_{k+1} \in \argmin_{\mu \in M^+(\mathcal{A}^{\mu,+}_k)} \left \lbrack F(\mathcal{K}\mu) +\|\mu\|_{M(\mathcal{B})} \right \rbrack.
\]
\STATE 6. Update
\begin{align*}
\mathcal{A}^\mu_{k+1}= \supp \mu_{k+1} 
\end{align*}
and set~$N_{k+1}=\# \mathcal{A}^{\mu}_{k+1}$.
\ENDFOR
\end{algorithmic}
\end{algorithm}

As for Algorithm~\ref{alg:gcg}, we define the residuals associated with the iterates~$\mu_k$ of Algorithm~\ref{alg:pdap} by
\begin{align} \label{eq:residualmu}
r_j(\mu_k) \coloneqq j(\mu_k)-\min_{\mu \in M^+(\Bb)} j(\mu)\,.
\end{align}
Note that due to Theorem \ref{prop:equivalence}, such residuals can be written as 
\[
r_j(\mu_k) = j(\mu_k)-\min_{u \in \M} J(u)\,.
\] 
Summarizing the previous observations, we conclude the equivalence between Algorithm~\ref{alg:gcg} and Algorithm~\ref{alg:pdap} as stated in the next theorem.
\begin{theorem} \label{thm:eqofpdapandgcg}
Let~$\mathcal{A}^\mu_k=\{u^k_i\}^{N_k}_{i=1}$ and~$\mu_k \in M^+(\mathcal{A}^\mu_k)$ be given. Set~$\mathcal{A}^u_k \coloneqq \mathcal{A}^\mu_k $ and~$u_k \coloneqq \I(\mu_k)$. Then, the update Steps from $2$ to $6$ in Algorithm~\ref{alg:gcg} and~Algorithm~\ref{alg:pdap} can be realized such that
\[
\widehat{v}^u_k=\widehat{v}^\mu_k \,, \,\,\,u_{k+1}= \I(\mu_{k+1})\,, \,\,\, \mathcal{A}^u_{k+1} = \mathcal{A}^\mu_{k+1}. 
\]
In particular, if~$\{u_k\}_k$ and~$\{\mathcal{A}^u_k\}_k$ are sequences of iterates and active sets generated by Algorithm~\ref{alg:gcg} and we have
\begin{align*}
u_0= \I(\mu_0 )\,, \  ~\mathcal{A}^u_0=\mathcal{A}^{\mu}_0 \,,
\end{align*}
then there exist sequences~$\{\mu_k\}_k$ and~$\{\mathcal{A}^\mu_k\}_k$ generated by Algorithm~\ref{alg:pdap} such that
\begin{align}\label{eq:induction}
u_k= \I(\mu_k )\,, \  ~\mathcal{A}^u_k=\mathcal{A}^{\mu}_k\,, \ ~\widehat{v}^u_k=\widehat{v}^{\mu}_k
\end{align}
and it holds
\begin{align}\label{eq:resinequ}
0 \leq r_J(u_k) \leq r_j(\mu_k) \quad \text{for all}~ k \in \N\,,
\end{align}
where the residuals~$r_J(u_k)$ and~$r_j(\mu_k)$ are defined in~\eqref{def:residualu} and~\eqref{eq:residualmu}, respectively.
\end{theorem}
\begin{proof}
By Lemma~\ref{lem:linearizedmeasure} we can choose~$\widehat{v}^u_k= \widehat{v}^\mu_k$. Thus~$\mathcal{A}^{u,+}_k=\mathcal{A}^{\mu,+}_k=\{u^k_i\}^{N^+_k}_{i=1}$. Next, in Step $5$ of Algorithm~\ref{alg:gcg} we compute a solution~$\lambda^{k+1} \in \R_+^{N^+_k}$ to
\begin{align*}
\min_{\lambda\in \R_+^{N^+_k}} \left \lbrack F\left(\sum^{N_k^+}_{i=1} \lambda_i Ku^k_i\right)+\sum^{N_k^+}_{i=1} \lambda_i \right \rbrack
\end{align*}
and we update $u_k$ as~$u_{k+1}=\sum^{N^+_k}_{i=1} \lambda^{k+1}_i u^k_i$. Therefore, defining $\mu_{k+1}=\sum^{N^+_k}_{i=1} \lambda^{k+1}_i \delta_{u^k_i}$ it holds that $u_{k+1}=\I(\mu_{k+1})$ and~$\mu_{k+1}$ is a solution to~\eqref{def:pdapsubprob} by Lemma~\ref{lem:finitedimonmeasures}. Moreover, we note that
\begin{align*}
\mathcal{A}^\mu_{k+1}= \supp \mu_{k+1}= \mathcal{A}^{u,+}_k \setminus \{\,u^k_i\;|\;\lambda^{k+1}_i=0\,\}=\mathcal{A}^u_{k+1}\,.
\end{align*}
Concerning the claim in \eqref{eq:induction}, given 
$u_k=\sum_{i=1}^{N^+_k} \lambda_i^k u_i^k$ and $\mathcal{A}^u_k$, we set $\mu_k=\sum_{i=1}^{N^+_k} \lambda_i^k \delta_{u_i^k}$ and $\mathcal{A}^\mu_k=\mathcal{A}^u_k$. Therefore  
\eqref{eq:induction} follows by the first part of the statement and an induction argument. Finally, 
as $\mathcal{I}(\mu_k)=u_k$, by Proposition~\ref{prop:I}~$i)$ and Proposition~\ref{prop:def_K}~$ii)$ we have $\G(u_k) \leq \norm{\mu_k}_{M(\mathcal{B})}$ and $\K\mu_k=Ku_k$. 
Therefore $J(u_k) \leq j(\mu_k)$, and 
\eqref{eq:resinequ} follows by \eqref{prop:equvalence:equal}.
\end{proof}

\section{Convergence analysis}\label{sec:convergence}

We are now prepared to prove Theorem~\ref{thm:convofcondu} and Theorem~\ref{thm:fastconvpreview}. For this purpose we rely on Theorem~\ref{thm:eqofpdapandgcg}, which states that the FC-GCG method from Algorithm~\ref{alg:gcg} converges at least as fast as the PDAP method in Algorithm~\ref{alg:pdap}, thanks to the estimate \eqref{eq:resinequ}. In particular, by Step~5 in Algorithm~\ref{alg:pdap}, we have that the iterates of FC-GCG and PDAP satisfy
\begin{align} \label{eq:ineqchain}
    r_J(u_k) \leq r_j(\mu_k)  \leq r_j(\mu) \quad \text{for all }~ \mu \in M^+(\mathcal{A}^{\mu}_k)\,, \, k \in \N \,.
\end{align}
In the following we use \eqref{eq:ineqchain}, as well as specific choices of the measure~$\mu$ in the upper bound, to prove Theorems~\ref{thm:convofcondu} and Theorem~\ref{thm:fastconvpreview}.
Specifically, the proof of Theorem~\ref{thm:convofcondu} is carried out in Section~\ref{subsec:worstpdap}. The proof of 
Theorem~\ref{thm:fastconvpreview}, which is more technical, is conducted 
in Section~\ref{subsec:fastconvproof}, after establishing some preliminary results in Section~\ref{subsec.auxiliary}.

For the remainder of the paper we silently assume that Algorithm~\ref{alg:pdap} does not stop after a finite number of iterations, and generates a sequence~$\{\mu_k\}_k$ in $M^+(\mathcal{B})$. Dropping superscripts, we denote the $k$-th active set, iterate, dual variable, candidate point computed in Step 2, and enlarged active set by, respectively, 
\begin{equation} \label{eq:sec6_notations}
\begin{gathered}
 \mathcal{A}_k  = \{ u_i^k\}_{i=1}^{N_k},   \quad \mu_k  = \sum_{i=1}^{N_k} \lambda_i^k \delta_{u_i^k} \,, \\   
 P_k  = -\K_* \nabla F(\K \mu_k) \in C(\mathcal{B})\,,  \quad \widehat{v}_k \in \argmax_{v \in \mathcal{B}}  \,
 P_k(v) \,, \quad \mathcal{A}_k^+:=\mathcal{A}_k \cup \{\widehat{v}_k\} \,,
 \end{gathered}
\end{equation}
where we recall that $\lambda_i^k > 0$ and $u_i^k, \widehat{v}_k  \in \operatorname{Ext}(B)$.

\subsection{Worst-case convergence rate} \label{subsec:worstpdap}
We first argue that Algorithm~\ref{alg:pdap} converges at least sublinearly.
To start, define the sublevel set
\begin{equation} \label{def:set_initial}
E_{\mu_0}:=\left\{\,\mu \in M^+(\mathcal{B})\; | \;j(\mu) \leq j(\mu_0)\,\right\}\,.
\end{equation}
By Theorem \ref{prop:equivalence}, we have that~$E_{\mu_0}$ is weak* compact. Let~$M_0>0$ be an arbitrary but fixed upper bound on the norm of elements in~$E_{\mu_0}$ and consider the norm constrained problem
\begin{align}\label{def:constprob}
\min_{\mu \in M^+(\mathcal{B})} j(\mu) \quad \text{ s.t. } \quad \|\mu\|_{M(\mathcal{B})}\leq M_0\,. \tag{$\widehat{P}_{M^+}$}
\end{align}
Clearly, by definition of $M_0$, the additional norm constraint does not change the set of global minimizers. The following proposition relates $\widehat v_k$  %
to a particular conditional gradient descent direction~$\eta_k$ for \eqref{def:constprob}.

\begin{proposition} \label{prop:linprob}
Define~$\eta_k \in M^+(\Bb)$ as
\begin{align*}
\eta_k:=
\begin{cases}
0 & \text{ if } \, P_k(\widehat{v}_k) <1\,, \\
M_0\, \delta_{\widehat{v}_k} & \text{ otherwise.} 
\end{cases}
\end{align*}
Then,~${\eta}_k$ is a minimizer of the partially linearized problem
\begin{align} \label{def:linprob}
\min_{\substack{\eta \in M^+(\mathcal{B}), \\ \|\eta\|_{M(\mathcal{B})}\leq M_0 }} \lbrack -\sm{P_k}{\eta} +\|\eta\|_{M(\mathcal{B})} \rbrack \,.
\end{align}
Moreover, we have
\begin{align*}
r_j(\mu_{k+1}) \leq r_j((1-s)\mu_k+s \eta_k) \quad \text{for all}~ s \in [0,1]\,.
\end{align*}
\end{proposition}

\begin{proof}
Since we are testing against positive measures, we can estimate
\[
\begin{aligned}
\min_{\substack{\eta \in M^+(\mathcal{B}), \\ \|\eta\|_{M(\mathcal{B})}\leq M_0 }} \lbrack -\sm{ P_k}{ \eta} +\|\eta\|_{M(\mathcal{B})} \rbrack & \geq
\min_{\substack{\eta \in M^+(\mathcal{B}), \\ \|\eta\|_{M(\mathcal{B})}\leq M_0 }} \lbrack (1-\max_{v \in \mathcal{B}} P(v))\|\eta\|_{M(\mathcal{B})} \rbrack \\
&=\begin{cases}
0 & \text{ if } \, P(\widehat{v}_k) <1\,, \\
(1-\max_{v \in \mathcal{B}} P(v)) M_0 & \text{ otherwise,} 
\end{cases}
\\&=  -\sm{ P_k}{ \eta_k} +\|\eta_k\|_{M(\mathcal{B})} \,.
\end{aligned}
\]
The proof is finished noting that~$(1-s)\mu_k+s \eta_k$ belongs to $M^+(\mathcal{A}_k^+)$, and that $\mu_{k+1}$ solves \eqref{def:pdapsubprob}.
\end{proof}

In particular, Proposition~\ref{prop:linprob} shows that, in each iteration, Algorithm~\ref{alg:pdap} achieves at least as much descent as a conditional gradient update. Such observation allows us to prove sublinear convergence for Algorithm~\ref{alg:pdap} using known convergence results for conditional gradient methods in general Banach space (see Theorem \ref{thm:convofcond} below).
Finally, the combination of Theorem~\ref{thm:eqofpdapandgcg}  with Theorem \ref{thm:convofcond} yields the convergence of Algorithm~\ref{alg:gcg}.

\begin{theorem} \label{thm:convofcond}
Let \ref{ass:func1}-\ref{ass:func3} in Assumption \ref{ass:functions} hold. 
Then, the sequence $\{j(\mu_k)\}_k$ is monotone decreasing,~$\mu_k \in E_{\mu_0}$, and there exists a constant $c>0$ such that
\begin{equation} \label{thm:convofcond:2}
r_j(\mu_k) \leq c \, \frac{1}{k+1} \quad \text{ for all } \,\, k \in \N\,,
\end{equation}
where $r_j$ is defined at \eqref{eq:residualmu}. 
The sequence~$\{\mu_k\}_{k}$ admits at least one weak* accumulation point and each such point is a solution to~\eqref{def:sparseprob}. If the solution~$\bar{\mu}$ to~\eqref{def:sparseprob} is unique, we have~$\mu_k \weakstar \bar{\mu}$ for the whole sequence. 
\end{theorem}
\begin{proof}
Since~$\mu_{k+1}$ is a solution to~\eqref{def:pdapsubprob} we clearly have $j(\mu_{k+1})\leq j(\mu_{k}) \leq j(\mu_{0})$.
Thus~$\{j(\mu_k)\}_k$ is monotone decreasing and~$\mu_k \in E_{\mu_0}$.
Now, we show that ~$\nabla(F \circ \K)$ is Lipschitz continuous on~$E_{\mu_0}$. Indeed, since~$E_{\mu_0}$ is weak* compact and~$\K$ is weak*-to-strong continuous, see Proposition~\ref{prop:def_K}~iii), the image set
$\K E_{\mu_0}=\left\{\,\K \mu\; |\;\mu \in E_{\mu_0}\,\right\}$
is compact in $Y$.
Using Assumption~\ref{ass:func1}, we have that $\nabla F$ is Lipschitz continuous on $\K E_{\mu_0}$ for some constant~$L_{\mu_0}>0$. Hence
\begin{align*}
\|\K_*(\nabla F(\K \mu_1)-\nabla F(\K \mu_2))\|_{C(\Bb)} &\leq C\|K\|_{\mathcal{L}(\M,Y)} \ynorm{\nabla F(\K \mu_1)-\nabla F(\K \mu_2)} \\ &\leq   L_{\mu_0} C \|K\|_{\mathcal{L}(\M,Y)} \ynorm{\K(\mu_1-\mu_2)} \\
& \leq L_{\mu_0}C^2 \|K\|^2_{\mathcal{L}(\M,Y)} \|\mu_1-\mu_2 \|_{M(\Bb)}
\end{align*}
for all~$\mu_1,\mu_2 \in E_{\mu_0}$, where $C$ is the constant from Proposition~\ref{prop:def_K}~i). %
The claimed convergence statement now follows from~\cite[Theorem 6.14]{daniel_thesis}. 
\end{proof}
\begin{proof}
[Proof of Theorem~\ref{thm:convofcondu}]
Assume that Algorithm~\ref{alg:gcg} does not converge after finitely many steps and generates a sequence~$\{u_k\}_k$. According to Theorem~\ref{thm:eqofpdapandgcg} there exists a sequence~$\{\mu_k\}_k$ generated by Algorithm~\ref{alg:pdap} with~$u_k= \I(\mu_k)$. Invoking Theorem~\ref{thm:eqofpdapandgcg} as well as Theorem~\ref{thm:convofcond} yields
\begin{align*}
 r_J(u_k) \leq r_j(\mu_k) \leq c \, \frac{1}{k+1} \quad \text{ for all } \,\, k \in \N\,.
\end{align*}
In particular Theorem~\ref{thm:convofcond} implies that, up to subsequences, $\mu_k \weakstar \bar \mu$ with $\bar \mu$ solution of \eqref{def:sparseprob}. Since
$\mathcal{I}$ is weak*-to-weak* continuous by Proposition~\ref{prop:I}~iii), and since $u_k = \mathcal{I}(\mu_k)$, we infer $u_k \weakstar \bar{u}$ with 
$\bar u := \mathcal{I}(\bar\mu)$. As $\bar\mu$ minimizes in \eqref{def:sparseprob}, by Theorem~\ref{prop:equivalence}~ii), we infer that $\bar u$ minimizes in \eqref{def:minprob}.
The rest of the statement follows by a similar argument. 
\end{proof}

\begin{remark}
In the next section, devoted to the proof of the linear convergence of Algorithm~\ref{alg:gcg}, the global sublinear convergence provided by Theorem~\ref{thm:convofcondu} plays an important role. Indeed, since the claimed linear convergence in Theorem \ref{thm:fastconvpreview} is only asymptotic, i.e., it relies on the iterates~$u_k$ being sufficiently close, in the weak* sense, to a solution $\bar u$ of \eqref{def:minprob}, the application of Theorem~\ref{thm:convofcondu} is a necessary starting point. Our proof of Theorem \ref{thm:fastconvpreview} relies on the lifting strategy of Section~\ref{sec:lifting}, the interpretation of PDAP as a monotone, fully-corrective GCG method for~\eqref{def:sparseprob}, as well as known convergence results for this algorithm. A similar identification for Algorithm~\ref{alg:gcg} is~\textit{not} directly possible. In fact, since in general we only have~$\mathcal{G}(u_k) \leq\kappa_{\mathcal{A}_k}(u_k)$, the residuals~$\{r_J(u_k)\}_k$ in Algorithm~\ref{alg:gcg} are not necessarily monotone.  
While we believe that a proof of Theorem~\ref{thm:convofcondu} which does not rely on~\eqref{def:sparseprob} is possible, this is, to the best of our knowledge, non-standard and would require further work. In particular, we point out that known results of~\cite{yu} for GCG methods with gauge-like regularizers seem not to be applicable since~$\mathcal{B}$ is, generally, not strongly compact. As such, the lifting strategy provides an elegant way to circumvent these additional arguments.   
\end{remark}

\subsection{Fast convergence and proof of Theorem \ref{thm:fastconvpreview}} \label{subsec.auxiliary}

In this section we further investigate the convergence behaviour of the iterates~$\{\mu_k\}_k$ of Algorithm~\ref{alg:pdap}, but now under the premise of Assumptions~\ref{ass:fast:F1}-\ref{ass:fast:F5}. 
The goal is to show an improved linear local convergence rate for Algorithm~\ref{alg:pdap}, see Theorem \ref{thm:fast_pdap} below. Thanks to this result, and to Theorem \ref{thm:eqofpdapandgcg}, we will then be able to prove linear convergence for the FC-GCG method of Algorithm~\ref{alg:gcg}, as stated in Theorem~\ref{thm:fastconvpreview}. As the proofs are quite technical, after establishing some notations we give a detailed summary of the results.

\subsubsection{Notation}
We employ the notation at \eqref{eq:sec6_notations} for $\mu_k$, $\mathcal{A}_k$, $\mathcal{A}_k^+$, $\widehat{v}_k$, $P_k$. Further, denote by
\[
\bar{u}=\sum_{i=1}^N \bar{\lambda}_i \bar{u}_i \,, \quad
\bar{\mu} = \sum_{i=1}^N \bar{\lambda}_i \delta_{\bar{u}_i} \,, \quad \bar{\lambda}_i>0\,, \quad \bar{u}_i \in \operatorname{Ext}(B)\,,
\]
the unique solutions to \eqref{def:minprob} and \eqref{def:sparseprob}, respectively. The existence and uniqueness are guaranteed by 
Proposition~\ref{prop:convimpliconmu}, which holds since we 
are assuming \ref{ass:fast:F2}-\ref{ass:fast:F3}, while $\bar{\lambda}_i >0$ by the non-degeneracy Assumption~\ref{ass:fast:F4}. 
Since \eqref{def:sparseprob} has a unique solution and 
the prerequisites of Theorem~\ref{thm:convofcond} are fulfilled, we have the convergences 
\begin{equation} \label{eq:whole_sequence}
\mu_k \weakstar \bar{\mu} \quad \text{weakly* in} \quad M(\mathcal{B})\,, \quad r_j(\mu_k) \to 0 \,,
\end{equation}
as $k \to \infty$, along the whole sequence. 
Also, we denote by
\[
\bar{\mathcal{A}}:= \{ \bar{u}_i\}_{i=1}^N\,, \quad 
  \bar{p}:= -K_* \nabla F(K \bar{u}) \in \Cc \,, \quad
\bar{P}:= -\K_* \nabla F(\K \bar{\mu}) \in C(\mathcal{B})\,,
\]
the set of optimal extremal points and the optimal dual variables associated to $\bar{u}$ and $\bar{\mu}$, respectively. 
Moreover, let $\{\bar{U}_i\}_{i=1}^N$ be the pairwise disjoint $d_\Bb$-closed neighborhoods of $\{\bar u_i\}_{i=1}^N$, 
as in Assumption~\ref{ass:fast:F5}, and set $U_i:=\bar{U}_i \cap \ext (B)$.
Finally, the observations are denoted by
\[
y_k:= \K \mu_k \,, \quad \bar{y}:=K\bar{u} = \K \bar{\mu} \,,
\]
where the equality $K\bar{u} = \K \bar{\mu}$ follows from Proposition~\ref{prop:def_K}~ii).

\subsubsection{Summary of results}
Our aim is to show the existence of $\zeta \in [3/4,1)$ and $c>0$ such that 
\begin{equation} \label{eq:summary_rate}
r_j(\mu_k) \leq c \, \zeta^k\,,
\end{equation}
for all $k$ sufficiently large, see Theorem~\ref{thm:fast_pdap} below. To obtain \eqref{eq:summary_rate} it will be sufficient, for fixed $k$, to construct a measure $\mu^+ \in M(\mathcal{A}^+_k)$ such that
\begin{align*}
r_j(\mu_{k+1}) \leq r_j(\mu) \leq \zeta r_j(\mu_k)\,.   %
\end{align*}
 We will choose $\mu:=\mu_k^s$ as the convex combination
\begin{equation} \label{eq:summary_combination}
\mu_k^s:=\mu_k + s (\widehat{\mu}_k - \mu_k), 
\end{equation}
for some $s \in [0,1]$, 
where the surrogate sequence $\widehat{\mu}_k$ is obtained by suitably modifying $\mu_k$ in a neighbourhood of the inserted point $\widehat{v}_k$. More in detail, we proceed as follows:

\begin{itemize}

    \item[i)] We start with some preparatory results in Sections~\ref{subsec:iterates},  \ref{subsec:dual_variables},~\ref{subsec:conv_dual}. After, in Section~\ref{subsec:asymptotic}, we show that the active sets~$\mathcal{A}_k$ cluster around~$\bar{\mathcal{A}}$, in the sense that $\mathcal{A}_k \subset \cup_{i=1}^N U_i$ for $k$ sufficiently large. Moreover, every point in~$\bar{\mathcal{A}}$ is approximated by at least one in~$\mathcal{A}_k$. This is a consequence of the uniform convergence of the dual variables $P_k$ to the optimal dual variable $\bar P$, see Section~\ref{subsec:conv_dual},  as well as the isolation of the maximizers of~$\bar{P}$, see Assumption~\ref{ass:fast:F2}. 
 
    \item[ii)] In Section~\ref{subsec:distance}, we quantify the distance between $\mathcal{A}_k^+$ and $\bar{\mathcal{A}}$ in terms of the function $g$ in Assumption~\ref{ass:fast:F5}, see Proposition~\ref{lem:estonallpoints}. This is a crucial part of our analysis, and also the first point where the growth estimates in~\ref{ass:fast:F5} come into play. Precisely, we show that $\mathcal{A}_k \cap U_i$ approaches $\bar{u}_i$ in the sense that 
    \begin{equation} \label{eq:summary_1}
        \int_{U_i} g(v,\bar{u}_i)~\de \mu_k (v)  \lesssim r_j(\mu_k)^{1/2} \,, \quad \text{ for all } \,\, i=1,\ldots,N\,.
    \end{equation}
    In addition, we prove that the inserted point $\widehat{v}_k$ at Step~2 is close to $\bar{\mathcal{A}}$, in the sense that there exists an index $\hat{\imath}_k \in \{1,\ldots,N\}$ such that $\widehat{v}_k \in U_{\hat{\imath}_k}$ and
    \begin{equation} \label{eq:summary_2}
    g(\widehat{v}_k,\bar{u}_{\hat{\imath}_k}) \lesssim r_j(\mu_k)^{1/2}\,. 
    \end{equation}
    
    \item[iii)] In Section~\ref{subsec:surrogate} we introduce the surrogate sequence 
    \[
    \widehat{\mu}_k := 
\mu_k(\bar{U}_{\hat{\imath}_k}) \delta_{\widehat{v}_k} + \sum^N_{i=1,~i\neq \hat{\imath}_k} \mu_k \zak \bar{U}_{i}\,,
    \]
    obtained by modifying $\mu_k$ in the neighbourhood $\bar{U}_{\hat{\imath}_k}$ of the inserted point $\widehat{v}_k$. 
    In Lemma~\ref{lem:estforwide} we prove that $\widehat{\mu}_k \weakstar \bar{\mu}$, as well as the following key estimate
    \begin{equation} \label{eq:summary_3}
    \norm{\K( \widehat{\mu}_k - \mu_k  )}_Y \lesssim r_j(\mu_k)^{1/2} \,.
    \end{equation}
    This estimate relies on \eqref{eq:summary_1}-\eqref{eq:summary_2}, and thus inherently requires Assumption~\ref{ass:fast:F5}.

     \item[iv)] In Section~\ref{subsec:fastconvpdap} we  employ~\eqref{eq:summary_3} to prove the estimate 
     \begin{equation} \label{eq:summary_growth}
     r_j(\mu_{k+1}) \leq r_j(\mu_k^s) \leq \zeta r_j(\mu_k) \,,
     \end{equation}
     for some $\zeta \in [3/4,1)$ and a suitable step size $s \in (0,1)$, 
     where $\mu_k^s$ is defined in \eqref{eq:summary_combination}. %
     From~\eqref{eq:summary_growth} we then obtain the local linear convergence rate for PDAP, see Theorem~\ref{thm:fast_pdap}.

     \item[v)] In Section~\ref{subsec:fastconvproof} we finally prove Theorem~\ref{thm:fastconvpreview}. Thanks to the above analysis, the proof is a simple consequence of the linear convergence rate for PDAP established in Theorem~\ref{thm:fast_pdap}, and of the link between PDAP and FC-GCG granted by the lifting strategy, see Theorem~\ref{thm:eqofpdapandgcg}.
\end{itemize}

\begin{remark}
The proof strategy described above is inspired by~\cite{pieper}, %
where the authors propose a method in the spirit of Algorithm~\ref{alg:pdap} for solving a minimization problem in the space of measures supported on a compact set in~$\R^n$.
However, the proofs in our setting often require novel techniques compared to the ones in~\cite{pieper}. Loosely speaking, this can be attributed to the fact that the metric space~$(\Bb,d_{\mathcal{B}})$ does not posses an obvious geometric structure. For example, in the euclidean setting of~\cite{pieper}, the estimates for the distance between~$\bar{\mathcal{A}}$ and~$\mathcal{A}_k$, or~$\widehat{v}_k$, respectively, rely on the convexity of euclidean balls, as well as on the higher order differentiability of the dual variable. Such strategy does not extend to our setting, where~$d_\mathcal{B}$-neighborhoods are generally non-convex and no apparent differentiable structure is available. The difference between both approaches shows most prominently in the key estimate for~$\mathcal{K}(\widehat{\mu}_k-\mu_k)$ anticipated in~\eqref{eq:summary_3}. %
In contrast, the fast convergence result in \cite{pieper} relies %
on
\begin{equation} \label{eq:euclidean_estimate}
\|\mathcal{K}(\widehat{\mu}_k-\mu_k)\|_Y \lesssim (P_k(\widehat{v}_k)-1)^{1/2}\,,
\end{equation}
see~\cite[Lemma~5.15]{pieper}. With the notation of the current paper, estimate \eqref{eq:euclidean_estimate} is obtained from a \textit{perturbed quadratic growth condition} of the form
\begin{align*}
    \sm{P_k}{ \widehat{v}_k-u} \geq \frac{\kappa}{2} \, g(u, \widehat{v}_k)^2
\end{align*}
in the vicinity of~$\widehat{v}_k$, the derivation of which relies on higher order differentiability of the dual variable. In general though, the quadratic growth condition in Assumption~\ref{ass:fast:F5} is \textit{not} stable with respect to perturbations of the dual variable. Thus, such arguments cannot be applied in our setting.  
\end{remark}

\subsubsection{Properties of iterates} \label{subsec:iterates}
Since~$\bar{\mu} \neq 0$ and $\norm{\cdot}_{M(\mathcal{B})}$ is weak* lower semicontinuous, from \eqref{eq:whole_sequence} we conclude the existence of $M \in \N$ such that 
\begin{equation} \label{eq:non_empty}
\mu_k \neq 0\,, \quad \mathcal{A}_k \neq \emptyset\,, \quad \text{ for all } \, k \geq M \,.
\end{equation}
Also, as a consequence of Theorem~\ref{thm:convofcond}
\begin{equation} \label{eq:bound_norm_mu_k}
\{\mu_k\}_k \subset  E_{\mu_0}\, , \qquad \sup_{k \in \N} \ \norm{\mu_k}_{M(\mathcal{B})} \leq M_0\,,
\end{equation}
where $E_{\mu_0}$ is the weak* compact set in \eqref{def:set_initial}, and $M_0>0$ is an upper bound on the norm of elements in $E_{\mu_0}$. In particular
$\bar\mu \in  E_{\mu_0}$ and $\norm{\bar\mu}_{M(\mathcal{B})} \leq M_0$.

\subsubsection{Properties of dual variables} \label{subsec:dual_variables}
Recall that $\bar{P}$ and $\bar{p}$ are related by
$\bar{P}(v) = \langle  \bar{p}, v\rangle$  for all $v \in \mathcal{B}$, thanks to Proposition \ref{prop:existofextrem}. In particular, Assumption~\ref{ass:fast:F2} reads
\begin{equation} \label{eq:P_equals_1}
\argmax_{v \in \mathcal{B}} \bar{P}(v)
=  \{ v \in \mathcal{B} \; | \; \bar{P}(v) = \langle  \bar{p}, v\rangle =1 \} = \bar{\mathcal{A}}\,.
\end{equation}
Similarly, we can reformulate Assumption~\ref{ass:fast:F5} and Remark \ref{rem:wlog} in terms of $\bar{P}$ to obtain 
\begin{align} \label{eq:isolbarP}
 \bar{P}(v) \leq 1- \sigma \quad \text{ for all } \,\, v \in \mathcal{B} \setminus \bigcup^N_{i=1} \bar{U}_i \,, \\
 \label{eq:quadgrowthbarP}
 \bar{P}(v) \leq 1 - \kappa g(v,\bar{u}_i)^2 \quad \text{ for all } \,\, v \in U_i \,,
\end{align}
where we remind the reader that 
$g\colon \ext(B)\times \ext(B)\to [0,\infty)$, and
$\kappa,\sigma >0$ are constants. 
The dual variables~$P_k$ satisfy the following optimality conditions. For a proof, see Appendix~\ref{sec:app_optcondpdapsub}.

\begin{proposition} \label{prop:optcondpdapsub}
Let $M$ be as in \eqref{eq:non_empty}. Then, for all~$k \geq M$ we have
\begin{align*}
P_k=1 \, \text{ on } \, \mathcal{A}_k \,,  \quad ~\sm{P_k}{\mu_k}=\|\mu_k\|_{M(\Bb)}\,, 
\quad \max_{v \in \Bb} P_k(v) \geq 1 \,.
\end{align*}
 \end{proposition}

\subsubsection{Convergence of dual variables and observations}
\label{subsec:conv_dual}
Due to the strong convexity of~$F$ around~$\bar{y}$, see Assumption~\ref{ass:fast:F1}, the worst-case convergence guarantee of Theorem~\ref{thm:convofcond} also carries over to the observations $y_k$ and the dual variables $P_k$, as stated in the following proposition. For a proof, we refer the reader to Appendix~\ref{sec:app_eststates}.

\begin{proposition} \label{prop:eststates}
There exist $M \in \N$ and $c>0$ such that, for all $k \geq M$, there holds
\begin{equation} \label{prop:eststates:1}
\ynorm{y_k-\bar{y}}+ \ynorm{\nabla F(y_k)-\nabla F(\bar{y})}+ \|P_k-\bar{P}\|_{C(\mathcal{B})} \leq c\sqrt{r_j(\mu_k)}\,.
\end{equation}
In particular, we have~$y_k \rightarrow \bar{y} $ in~$Y$ and~$P_k \rightarrow \bar{P}$ in~$C(\Bb)$.
\end{proposition}

\subsubsection{Asymptotic behavior of $\mathcal{A}_k$} \label{subsec:asymptotic}
In the following we show that active sets~$\mathcal{A}_k$ cluster around~$\bar{\mathcal{A}}$. Specifically, for $k$ sufficiently large, we have $\mathcal{A}_k \subset \cup_{i=1}^N U_i$, with $\mathcal{A}_k \cap U_i \neq \emptyset$ for all $i=1,\ldots,N$. 
 
\begin{proposition} \label{prop:nonempty}
There exists $M \in \N$ such that for all~$k \geq M$ we have
\begin{gather}
 P_k(v) \leq 1-\sigma/2 \quad \text{ for all } \,\, v \in \mathcal{B} \setminus \bigcup^N_{i=1} \bar{U}_i\,,  \label{eq:localization1}\\
 \mathcal{A}_k \subset \bigcup^N_{i=1} U_i\label{eq:localization2}\,,
\end{gather}
where $\sigma>0$ is the constant in \eqref{eq:isolbarP}. 
Moreover, for all $k \geq M$ and $i=1,\ldots,N$,  
\begin{equation} \label{eq:non_empty_inline}
\lim_{k \to \infty} \mu_k (\bar{U}_i) = \bar{\lambda}_i\,, \qquad \mathcal{A}_k \cap U_i \neq \emptyset \,.
\end{equation}
\end{proposition}
 
\begin{proof}
As $P_k \to \bar{P}$ uniformly by
Proposition~\ref{prop:eststates}, we deduce the existence of~$M \in \N$ such that, for all $k \geq M$, it holds~$ \|P_k-\bar{P}\|_{C(\mathcal{B})} \leq \sigma/2$. As a consequence, for all $v \in \mathcal{B} \setminus \bigcup^N_{i=1} \bar{U}_i$ and $k$ sufficiently large we have
\begin{align*}
P_k(v) & =\bar{P}(v)+P_k(v)-\bar{P}(v) \\
& \leq 1- \sigma + \|P_k-\bar{P}\|_{C(\mathcal{B})}  \leq 1-\sigma/2\,,
\end{align*}
where we used~\eqref{eq:isolbarP}.
This proves \eqref{eq:localization1}. Now, recall that $P_k=1$ on  $\mathcal{A}_{k}$ by Proposition~\ref{prop:optcondpdapsub}. Therefore $\mathcal{A}_k \subset \cup_{i=1}^N \bar{U}_i$, otherwise \eqref{eq:localization1} would yield a contradiction. Since by construction $\mathcal{A}_k \subset \ext(B)$ and $U_i:=\bar{U}_i \cap \ext(B)$, we conclude \eqref{eq:localization2}. 
Consider now an arbitrary but fixed index~$l$ in $\{1,\dots,N\}$. Recall that
the sets~$\bar{U}_i$ are pairwise disjoint and~$d_{\mathcal{B}}$-closed. Therefore we can apply Urysohn's lemma to obtain a $d_{\mathcal{B}}$-continuous function~$\varphi_l \colon \mathcal{B} \to [0,1]$ such that~$\varphi_l=1$ in $\bar{U}_l$, and~$\varphi_l=0$ in $\bar{U}_i$,~$i\neq l$. Recall that $\mu_k \weakstar \mu$ along the whole sequence by \eqref{eq:whole_sequence}.  
Therefore, since $\bar \mu = \sum_{i=1}^N \bar \lambda_i \delta_{\bar u_i}$, we get
\begin{align*}
\bar{\lambda}_l= \sm{\varphi_l}{ \bar{\mu}}= \lim_{k\rightarrow \infty } \sm{ \varphi_l}{ \mu_k}= \lim_{k\rightarrow \infty} \mu_k(\bar{U}_l)\,,
\end{align*}
where in the last equality we used \eqref{eq:localization2}.
Finally, since~$\bar{\lambda}_l>0$, from the above convergence we get~$\mu_k(\bar{U}_l)>0$ for $k$ sufficiently large. Recalling that $\mathcal{A}_k \subset \ext(B)$ and that $U_l:=\bar{U}_l \cap \ext(B)$, we then infer $\mu_k(U_l)>0$ for $k$ sufficiently large, concluding~$\mathcal{A}_k \cap U_l \neq 0$. 
\end{proof}

\subsubsection{Distance between $\mathcal{A}_k^+$ and $\bar{\mathcal{A}}$} \label{subsec:distance}

In the following proposition we use the previous results to quantify the distance between~the active set~$\mathcal{A}_k$ and the set of optimal extremal points $\bar{\mathcal{A}}$, see \eqref{lem:estonallpoints:1} below. 
We also provide an estimate for the distance of~$\widehat{v}_k$ to the closest element in $\bar{\mathcal{A}}$, see \eqref{eq:estoonnewpoint}. We remark that this is the first point 
in which the estimates of Assumption~\ref{ass:fast:F5} are employed.

\begin{proposition} \label{lem:estonallpoints}
There exist $M \in \N$ and a constant $c>0$ such that for all $k \geq M$ and $i=1,\dots,N$ it holds
\begin{equation} \label{lem:estonallpoints:1}
\sum_{v  \in \mathcal{A}_k \cap  U_i } \mu_k(\{v\}) g(v,\bar{u}_i) \leq c \sqrt{r_j(\mu_{k})}\,.
\end{equation}
Moreover, for each $k \geq M$ there is an index~$\hat{\imath}_k \in\{1,\dots,N\}$ such that $\widehat{v}_k \in U_{\hat{\imath}_k}$ and
\begin{align} \label{eq:estoonnewpoint}
g(\widehat{v}_k,\bar{u}_{\hat{\imath}_k}) \leq c \sqrt{r_j(\mu_k)}\,.
\end{align}
\end{proposition}

\begin{proof}
By minimality of $\bar{\mu}$ in \eqref{def:sparseprob} and convexity of $F$ we obtain
\begin{equation}\label{eq:lem:estonallpoints:1}
\begin{aligned} 
r_j(\mu_k) & = F(\K \mu_k) - F(\K \bar \mu) + \nor{\mu_k}_{M(\Bb)} - \nor{\bar\mu}_{M(\Bb)} \\ 
&  \geq ( \nabla F (\K \bar \mu), \K \mu_k-\K \bar \mu )_Y + \nor{\mu_k}_{M(\Bb)} - \nor{\bar\mu}_{M(\Bb)} \\
&  = \sm{ \bar{P}}{\bar{\mu}-\mu_k } + \nor{\mu_k}_{M(\Bb)} - \nor{\bar\mu}_{M(\Bb)} =\|\mu_k\|_{M(\mathcal{B})} -\sm{ \bar{P}}{\mu_k}\,,
\end{aligned}
\end{equation}
where we used the optimality condition~\eqref{eq:optconditionstatement} in the last line. Fix $i_0\in\{1,\dots,N\}$ and an arbitrary~$k\in\N$ large enough such that all previous results in this section hold. We will show \eqref{lem:estonallpoints:1} for the index $i_0$. 
By \eqref{eq:localization2} and definition of $\mu_k$ we obtain
\begin{gather} \label{eq:lem:estonallpoints:2}
\nor{\mu_k}_{M(\Bb)} = \sum^N_{i=1} \sum_{v  \in \mathcal{A}_k \cap  U_i } \mu_k(\{v\}) \,, \\
\sm{ \bar{P}}{\mu_k} 
 = \sum^N_{i=1} \sum_{v  \in \mathcal{A}_k \cap  U_i } \mu_k(\{v\}) \, \bar{P}(v)\,. \label{eq:lem:estonallpoints:3}
\end{gather}
Putting together \eqref{eq:lem:estonallpoints:1}-\eqref{eq:lem:estonallpoints:3}, and using~\eqref{eq:quadgrowthbarP}, that is, \ref{ass:fast:F5}, we estimate 
\begin{align*}
r_j(\mu_k) & %
 \geq  \sum^N_{i=1} \sum_{v  \in \mathcal{A}_k \cap  U_i } \mu_k(\{v\})(1 - \bar{P}(v)) \\
 & \geq \kappa \sum^N_{i=1} \sum_{v  \in \mathcal{A}_k \cap  U_i } \mu_k(\{v\})g (v,\bar{u}_i)^2 
 \geq \kappa \sum_{v  \in \mathcal{A}_k \cap  U_{i_0} } \mu_k(\{v\})g (v,\bar{u}_{i_0})^2 \,.
\end{align*}
Since $\mathcal{A}_k \cap U_{i_0} \neq \emptyset$ by Proposition \ref{prop:nonempty}, we have $\mu_k(\bar{U}_{i_0})= \sum_{v \in \mathcal{A}_k \cap U_{i_0}} \mu_k(\{v\})>0$. Due to the convexity of~$(\cdot)^2$, we conclude
\begin{equation} \label{eq:lem:estonallpoints:4}
\begin{aligned} 
r_j(\mu_k) & \geq \kappa  \sum_{v  \in \mathcal{A}_k \cap  U_{i_0} } \mu_k(\{v\})g (v,\bar{u}_{i_0})^2 \\
& \geq \frac{\kappa}{\mu_k(\bar{U}_{i_0})} \left(  \sum_{v  \in \mathcal{A}_k \cap  U_{i_0} } \mu_k(\{v\})g (v,\bar{u}_{i_0}) \right)^2 \,.
\end{aligned}
\end{equation}
Recall~$\mu_k(\bar{U}_{i_0}) \leq M_0$ by \eqref{eq:bound_norm_mu_k}. Estimate \eqref{lem:estonallpoints:1} now follows by \eqref{eq:lem:estonallpoints:4} with $c:=\sqrt{M_0/\kappa}$.

We now show~\eqref{eq:estoonnewpoint}. Note that by Proposition~\ref{prop:optcondpdapsub} we have  $P_k(\widehat{v}_k)=\max_{v \in \Bb} P_k(v) \geq 1$ for all~$k\in\N$ large enough. Therefore $\widehat{v}_k \in \cup_{i=1}^N \bar{U}_i$ by \eqref{eq:localization1}. Recalling that $\widehat{v}_k \in \operatorname{Ext}(B)$ and that $U_i:=\bar{U}_i \cap \operatorname{Ext}(B)$
are pairwise disjoint, we deduce that $\widehat{v}_k \in U_{\hat{\imath}_k}$ for some unique index $\hat{\imath}_k$ in $\{1,\dots,N\}$.
Utilizing~\eqref{eq:quadgrowthbarP}, i.e.~\ref{ass:fast:F5}, and the fact that $\bar{P}(\bar{u}_{\hat{\imath}_k})=1$ by \eqref{eq:P_equals_1}, we estimate
\begin{equation} \label{eq:estnewpoint1}
\begin{aligned}
\kappa\, g(\widehat{v}_k,\bar{u}_{\hat{\imath}_k})^2 & \leq 
1-\bar{P}(\widehat{v}_k) =
\bar{P}(\bar{u}_{\hat{\imath}_k})-\bar{P}(\widehat{v}_k) \\
& \leq \bar{P}(\bar{u}_{\hat{\imath}_k})-\bar{P}(\widehat{v}_k) +P_k(\widehat{v}_k)-P_k(\bar{u}_{\hat{\imath}_k})\,,
\end{aligned}
\end{equation}
where in the last line we used that $P_k(\widehat{v}_k) = \max_{v \in  \Bb} P_k(v)$. Using Assumption~\ref{ass:fast:F5} as well as Proposition \ref{prop:existofextrem}, the right-hand side in \eqref{eq:estnewpoint1} is further bounded by 
\begin{equation} \label{eq:estnewpoint2}
\begin{aligned}
\bar{P}(\bar{u}_{\hat{\imath}_k})-\bar{P}(\widehat{v}_k) & +P_k(v_k)-P_k(\bar{u}_{\hat{\imath}_k}) \\ & = \langle \bar{p}-p_k,  \bar{u}_{\hat{\imath}_k} - \widehat{v}_k \rangle
\\ &= (\nabla F(\bar{y})-\nabla F(y_k),K(\widehat{v}_k - \bar{u}_{\hat{\imath}_k}))_Y
\\ &\leq \ynorm{\nabla F(\bar{y})-\nabla F(y_k)}\, \ynorm{K(\widehat{v}_k - \bar{u}_{\hat{\imath}_k})} \\& \leq \tau \ynorm{\nabla F(\bar{y})-\nabla F(y_k)}\, g(\widehat{v}_k,\bar{u}_{\hat{\imath}_k}) \,,
\end{aligned}
\end{equation}
where $\tau>0$ is the constant from \eqref{eq:quadgrowth}, which does not depend on $k$.
Finally, using \eqref{eq:estnewpoint1} and \eqref{eq:estnewpoint2} together with Proposition~\ref{prop:eststates}, we conclude \eqref{eq:estoonnewpoint}. 
\end{proof}

\subsubsection{Surrogate sequence}
\label{subsec:surrogate}

Let $M \in\N$ be sufficiently large so that all the above results hold. For $k \geq M$ we denote by $\hat{\imath}_k  \in\{1,\dots,N\}$ the index from Proposition~\ref{lem:estonallpoints}. Starting from the sequence $\{\mu_k\}_k$ generated by Algorithm \ref{alg:pdap}, we define the  \emph{surrogate sequence} $\{\widehat \mu_k\}_k$ in $M^+(\mathcal{B})$ by setting
\begin{equation}  \label{def:surrogate}
\widehat{\mu}_k := 
\mu_k(\bar{U}_{\hat{\imath}_k}) \delta_{\widehat{v}_k} + \sum^N_{i=1,~i\neq \hat{\imath}_k} \mu_k \zak \bar{U}_{i}\,,
\end{equation}
where $\mu_k \zak \bar{U}_{i}$ is defined as in Section~\ref{subsec:measure_notation}. 
Notice that $\widehat{\mu}_k$ is just a local modification of $\mu_k$ around $\widehat{v}_k$. 
In the following lemma we investigate the properties of $\widehat{\mu}_k$. Most importantly, 
we establish the weak* convergence of~$\widehat{\mu}_k$ towards~$\bar{\mu}$, and the crucial estimate \eqref{eq:summary_3}.  %

\begin{lemma} \label{lem:estforwide}
For all $k \geq M$ there holds
\begin{gather} \label{eq:samelinear}
\|\mu_k\|_{M(\mathcal{B})}=\|\widehat{\mu}_k\|_{M(\mathcal{B})} \,, \\ 
\label{eq:samelinear:1}
\sm{ {P}_k}{\widehat{\mu}_k-\mu_k}= \mu_k(\bar{U}_{\hat{\imath}_k})(P_k(\widehat{v}_k)-1)\,, \\
 \label{lem:estforwide:1}
\ynorm{\mathcal{K}(\widehat{\mu}_k-\mu_k)} \leq c \sqrt{r_j(\mu_k)}\,,
\end{gather}
where $c>0$ does not depend on $k$. Moreover, as $k \to \infty$,  we have
\[
 \mathcal{K}\widehat{\mu}_k \to \mathcal{K}\bar{\mu} \,\, \text{ strongly in } \, Y \,, \,\, \quad j(\widehat{\mu}_k) \to j(\bar \mu)  \, , \,\,  \quad \widehat{\mu}_k \weakstar \bar \mu   \,.
\]
In particular,~$\widehat{\mu}_k \in E_{\mu_0}$ for all~$k\in\N$ large enough, where $E_{\mu_0}$ is the set in \eqref{def:set_initial}.
\end{lemma}

\begin{proof}
Recall that $\widehat{v}_k \in U_{\hat{\imath}_k}$ by Proposition \ref{lem:estonallpoints}. Using \eqref{eq:localization2}, the definition of $\widehat \mu_k$, and the fact that the sets $U_i$ are pairwise disjoint, it is immediate to check that \eqref{eq:samelinear} holds.
Noting that 
\begin{equation} \label{eq:subtraction}
\widehat{\mu}_k-\mu_k = \mu_k(\bar{U}_{\hat{\imath}_k}) \delta_{\widehat{v}_k} - \mu_k \zak \bar{U}_{\hat{\imath}_k}\,,
\end{equation}
we also obtain \eqref{eq:samelinear:1}, since
\begin{align*}
\sm{ P_k}{\widehat{\mu}_k-\mu_k} 
 & = \mu_k(\bar{U}_{\hat{\imath}_k}) P_k(\widehat{v}_k) - 
\int_{\bar{U}_{\hat{\imath}_k}} P_k(v)\,\de \mu_k(v) \\
 & =\mu_k(\bar{U}_{\hat{\imath}_k})(P_k (\widehat{v}_k)-1)\,,
\end{align*}
where in the last equality we used that $P_k=1$ on $\mathcal{A}_k = \supp \mu_k$ by Proposition~\ref{prop:optcondpdapsub}. We now show \eqref{lem:estforwide:1}. By \eqref{eq:localization2} and \eqref{eq:subtraction} we compute
\begin{align*}
\widehat \mu_k - \mu_k  = \sum_{ v \in \mathcal{A}_k \cap {U}_{\hat{\imath}_k}  } \mu_k(\{v\}) (\delta_{\widehat{v}_k} - \delta_v)\,.
\end{align*}
Note that $\mathcal{I}(\delta_{\widehat{v}_k} - \delta_v)=\widehat{v}_k -v$. Thus linearity of $\K$, Proposition~\ref{prop:def_K}~ii), and triangle inequality imply
\begin{align*}
\ynorm{\mathcal{K}(\widehat{\mu}_k-\mu_k)} \leq \sum_{ v \in \mathcal{A}_k \cap {U}_{\hat{\imath}_k}  } \mu_k(\{v\}) \ynorm{K(\widehat v_k-v)} \,.
\end{align*}
Recalling that $\widehat v_k \in {U}_{\hat{\imath}_k}$, by \ref{ass:fast:F5} and \eqref{lem:estonallpoints:1}, \eqref{eq:estoonnewpoint} we estimate
\begin{align*}
\sum_{ v \in \mathcal{A}_k \cap {U}_{\hat{\imath}_k}  }  \mu_k(\{v\}) &  \ynorm{K(\widehat{v}_k-v)}    \\ &
 \leq \sum_{ v \in \mathcal{A}_k \cap {U}_{\hat{\imath}_k}  } \mu_k(\{v\})(\ynorm{K(\widehat{v}_k-\bar{u}_{\hat{\imath}_k})} + \ynorm{K(v-\bar{u}_{\hat{\imath}_k})})  \\
& \leq 
\tau \sum_{ v \in \mathcal{A}_k  \cap {U}_{\hat{\imath}_k}  } \mu_k(\{v\}) (g(\widehat{v}_k,\bar{u}_{\hat{\imath}_k})+g(v,\bar{u}_{\hat{\imath}_k}))  \\
& \leq (\mu_k( {U}_{\hat{\imath}_k} ) + 1) \tau c \, \sqrt{r_j(\mu_k)} \leq ( M_0 + 1) \tau c \, \sqrt{r_j(\mu_k)}\,,
\end{align*}
where in the last inequality we used~\eqref{eq:bound_norm_mu_k}. This establishes \eqref{lem:estforwide:1}. 
 
 As for the remaining part of the statement, first recall that
  $r_j(\mu_k) \to 0$ and $\mu_k \weakstar \bar{\mu}$ along the whole sequence by~\eqref{eq:whole_sequence}. Since $\K$ is weak*-to-strong continuous, see Proposition~\ref{prop:def_K}~iii), we conclude that $\K\mu_k \to \K \bar\mu$ strongly in $Y$. Thus, $\K \widehat{\mu}_k \to \K \bar \mu$ from \eqref{lem:estforwide:1}. We now show that $j(\widehat{\mu}_k) \to j (\bar \mu)$. Recalling that $\|\widehat{\mu}_k\|_{M(\mathcal{B})}=\|\mu_k\|_{M(\mathcal{B})}$ by \eqref{eq:samelinear}, we obtain
\begin{equation}
\begin{aligned}
0 \leq r_j(\widehat{\mu}_k) & = F(\K \widehat{\mu}_k) + \nor{\mu_k}_{M(\Bb)} - j(\bar\mu) \\ & \leq r_j(\mu_k) +| F(\K \widehat{\mu}_k) - F(\K \mu_k)| \,. 
\end{aligned} \label{eq:estimate_j_proof}
\end{equation}
We have $r_j(\mu_k) \to 0$ by Theorem~\ref{thm:convofcond}, while $| F(\K \widehat{\mu}_k) - F(\K \mu_k)| \to 0$ since   $\K \widehat{\mu}_k, \K \mu_k \to \K\bar \mu$ and $F$ is continuous. 
Therefore the right hand side of \eqref{eq:estimate_j_proof} converges to zero, which implies $r_j(\widehat{\mu}_k) \to 0$, that is, $j(\widehat{\mu}_k) \to j (\bar \mu)$. We are left to show that $\widehat{\mu}_k \weakstar \bar{\mu}$. 
Indeed, we have shown that   
$\{\widehat{\mu}_k\}_k$ is a minimizing sequence for \eqref{def:sparseprob}. Since $j$ is weak* lower semicontinuous and has weak* compact sublevels (Theorem \ref{prop:equivalence}), we infer the existence of a subsequence $n_k$ and of $\hat{\mu} \in M^+(\mathcal{B})$ such that  $\mu_{n_k} \weakstar \hat\mu$, with $\hat\mu$ minimizer of $j$. By uniqueness of the minimizer, see Proposition \ref{prop:convimpliconmu}, we conclude that $\mu_{n_k} \weakstar \bar\mu$. Moreover, since the weak* limit does not depend on the subsequence $n_k$, we also infer $\mu_{k} \weakstar \bar\mu$. Finally, since $\mu_0$ is not a minimizer of $j$, we have $\widehat{\mu}_k \in E_{\mu_0}$ for $k$ sufficiently large.
\end{proof}

\subsubsection{Linear convergence of PDAP} \label{subsec:fastconvpdap}

Finally, we are in position to prove the linear convergence of Algorithm~\ref{alg:pdap}.

\begin{theorem} \label{thm:fast_pdap}
Assume \ref{ass:func1}-\ref{ass:func3} and \ref{ass:fast:F1}-\ref{ass:fast:F5}. Let~$\{\mu_k\}_k$ be generated by Algorithm~\ref{alg:pdap}. Then, we have $\mu_k \weakstar \bar{\mu}$ in~$M(\mathcal{B})$. Moreover, there is~$\bar{k} \in \N$ and~$\zeta \in [3/4,1)$ such that
\begin{equation} \label{eq:thm_fast_pdap}
r_j(\mu_{k+1}) \leq \zeta \, r_j(\mu_k) 
\end{equation}
for all~$k \geq \bar{k}$. In particular, there is~$c>0$ with
\begin{equation} \label{eq:thm_fast_pdap_1}
r_j(\mu_k) \leq c \, \zeta^k 
\end{equation}
for all~$k\in\N$ sufficiently large.
\end{theorem}

\begin{proof}
The fact that $\mu_k \weakstar \bar{\mu}$ in~$M(\mathcal{B})$ along the whole sequence is already established in \eqref{eq:whole_sequence}. We have to show the improved convergence rate at \eqref{eq:thm_fast_pdap}. To this end, 
for a fixed $s\in[0,1]$ define
\[
\mu^s_k :=\mu_k+s(\widehat{\mu}_k-\mu_k) \,,
\] 
with $\widehat{\mu}_k$ as in Definition \eqref{def:surrogate}.
We will obtain \eqref{eq:thm_fast_pdap} by estimating the residual of $\mu^s_k$ and choosing an optimal value of $s$. We start by noting that $\mathcal{A}_k \subset \cup_{i=1}^N U_i$ by \eqref{eq:localization2}, with the sets $U_i$ pairwise disjoint. Therefore $\mu_k=\sum_{i=1}^N \mu_k \zak \bar{U}_i$ and
\[
\mu^s_k = \left[ \sum_{i=1, i \neq \hat{\imath}_k}^N \mu_k \zak \bar{U}_i \right] +
s \mu_k(\bar{U}_{\hat{\imath}_k}) \delta_{\widehat{v}_k} + (1-s)  \mu_k \zak \bar{U}_{\hat{\imath}_k}  \,.
\]
Since $\widehat{v}_k \in \bar{U}_{\hat{\imath}_k}$, 
from the above we deduce
\begin{align} \label{eq:mu_s_decomposition_0}
\nor{\mu_k^s}_{M(\Bb)}= \nor{\mu_k}_{M(\Bb)}\,.
\end{align}
Recall that $\mu_{k+1}$ is optimal in  \eqref{def:pdapsubprob} by definition of Algorithm \ref{alg:pdap}. As $\supp \mu^s_k \subset \mathcal{A}_k^+$, we infer
\begin{equation} \label{eq:prooffast:0}
j(\mu_{k+1})\leq j(\mu^s_{k}) \quad \text{ for all } \,\, s \in[0,1]\,.
\end{equation}
Next, we estimate the residual of $\mu^s_{k}$. By \eqref{eq:mu_s_decomposition_0} and regularity of $F$ we infer
\begin{equation} \label{eq:prooffast:1}
\begin{aligned}
j(\mu_k^s)-j(\mu_k) & = F(\K \mu_k + s ( \K \hat{\mu}_k -  \K \mu_k   )) -   F(\K \mu_k)\\
& \leq s (\nabla F(\K \mu_k), \K \widehat{\mu}_k - \K\mu_k)_Y  + R(\mu_k) \\
& =-s \sm{P_k}{\widehat{\mu}_k -\mu_k}  + R_s(\mu_k)\,,
\end{aligned}
\end{equation}
where $R_s(\mu_k)$ is a remainder defined by
\[
R_s(\mu_k):=\int_0^s (\nabla F(\K\mu_k^t) - \nabla F(\K\mu_k),\\K \widehat{\mu}_k - K\mu_k  )_Y~\de t \,.
\]
In order to estimate $R_s(\mu_k)$, first recall that $\{\mu_k\}_k \subset E_{\mu_0}$ by \eqref{eq:bound_norm_mu_k}. Moreover, by Lemma \ref{lem:estforwide}, we have $\widehat{\mu}_k \in E_{\mu_0}$ for $k$ sufficiently large. Therefore, as $E_{\mu_0}$ is convex, we also have $\mu^s_k \in E_{\mu_0}$ for all $s \in [0,1]$ and $k$ sufficiently large. 
Note that the set $\K E_{\mu_0}:=\{\,\K \mu \; | \; \mu \in E_{\mu_0}\,\}$ is compact, given that $\K$ is weak*-to-strong continuous, see~Proposition~\ref{prop:def_K}~iii). Denote by $L_{\mu_0}$ the Lipschitz constant of $\nabla F$ on the set $\K E_{\mu_0}$, which exists by Assumption~\ref{ass:func1}. 
By Cauchy-Schwarz we then estimate
\begin{equation} \label{eq:prooffast:2}
\begin{aligned}
|R_s(\mu_k)| & \leq \nor{ \K \widehat{\mu}_k - \K\mu_k}_Y  \int_0^s  \|\nabla F( \K\mu_k^t )- \nabla F(\K\mu_k) \|_Y~\de t \\
& \leq L_{\mu_0}  \nor{ \K \widehat{\mu}_k - \K\mu_k}_Y \, \int_0^s  \|\K \mu_k^t - \K\mu_k \|_Y~\de t  \\ & = s^2 \, \frac{ L_{\mu_0}}{2} \nor{ \K (\widehat{\mu}_k - \mu_k)}_Y^2 \,.
\end{aligned}
\end{equation}
By minimality of $\bar \mu$ and \eqref{eq:prooffast:1}-\eqref{eq:prooffast:2}, we infer
\begin{equation} \label{eq:prooffast:3}
\begin{aligned}
r_j(\mu^s_k) & =  j(\mu^s_k) - j(\bar \mu) \\
 & = j(\mu_k) - j(\bar \mu) -s \sm{P_k}{\widehat{\mu}_k -\mu_k}  + R_s(\mu_k) \\
& \leq r_j(\mu_k)-s  \sm{ P_k}{\widehat{\mu}_k-\mu_k } +s^2 \,\frac{L_{\mu_0}}{2} \ynorm{\mathcal{K}(\widehat{\mu}_k-\mu_k)}^2.
\end{aligned}
\end{equation}
Recall that by construction $P_k(\widehat{v}_k)= \max_{v \in \mathcal{B}} P_k(v)$. Moreover $\max_{v \in \mathcal{B}} P_k(v) \geq 1$ by Proposition \ref{prop:optcondpdapsub}. Thus
$P_k(\widehat{v}_k) \geq 1$ and we can apply Proposition \ref{prop:linprob} to infer that $M_0 \delta_{\widehat{v}_k}$ minimizes in \eqref{def:linprob}, that is,
\begin{equation} \label{eq:linprob_proof}
\sm{P_k}{\eta} - \norm{\eta}_{M(\mathcal{B})} \leq  
M_0 (P_k(\widehat{v}_k)-1)
\end{equation}
for all $\eta \in M^+(\mathcal{B})$ with $\norm{\eta}_{M(\mathcal{B})} \leq M_0$. 
We can now use convexity of $F$ to estimate
\[
\begin{aligned}
r_j(\mu_k) & = F(\K \mu_k)   - F(\K \bar\mu) + \norm{\mu_k}_{M(\mathcal{B})} - \norm{\bar\mu}_{M(\mathcal{B})} \\
& \leq (\nabla F(\K \mu_k), \K (\mu_k - \bar\mu))_Y +
\norm{\mu_k}_{M(\mathcal{B})} - \norm{\bar\mu}_{M(\mathcal{B})}  \\
& = \sm{P_k}{\bar\mu - \mu_k} + \norm{\mu_k}_{M(\mathcal{B})} - \norm{\bar\mu}_{M(\mathcal{B})} \\ 
 & = \sm{P_k}{\bar\mu }  - \norm{\bar\mu}_{M(\mathcal{B})}\,,
\end{aligned}
\]
where in the last equality we used Proposition~\ref{prop:optcondpdapsub}. Since $\|\bar \mu\|_{M(\mathcal{B})} \leq M_0$, see Section \ref{subsec:iterates}, we can apply \eqref{eq:linprob_proof} with $\eta=\bar\mu$ to obtain
\begin{align*}
r_j(\mu_k)  \leq M_0(P_k(\widehat{v}_k)-1) \,.
\end{align*}
By \eqref{eq:samelinear:1} we then infer
\begin{equation} \label{eq:prooffast:4}
\sm{ P_k}{\widehat{\mu}_k-\mu_k } = \mu_k(\bar{U}_{\hat{\imath}_k }) (P_k(\widehat{v}_k)-1) \geq \frac{\mu_k(\bar{U}_{\hat{\imath}_k })}{M_0}  \, r_j(\mu_k)\,.
\end{equation}
Using \eqref{eq:prooffast:4} and \eqref{lem:estforwide:1} in \eqref{eq:prooffast:3} yields
\begin{equation} \label{eq:prooffast:5}
r_j(\mu_k^s) \leq r_j(\mu_k)-s \, \frac{\mu_k(\bar{U}_{\hat{\imath}_k})}{M_0}\, r_j(\mu_k)+s^2 \,\frac{  L_{\mu_0} c_2}{2} \,r_j(\mu_k)\,,
\end{equation}
where $c_2>0$ is the square of the constant in \eqref{lem:estforwide:1}. Define the constant
\begin{align*}
c_1:=\min_{i=1,\dots,N} \frac{\bar{\mu}(\bar{U}_i)}{2M_0}=\min_{i=1,\dots,N} \frac{\bar{\lambda}_i}{2M_0}\,.
\end{align*}
Notice that $c_1>0$, since \ref{ass:fast:F4} holds. Moreover,~$c_1 \leq 1/2$ given that~$\norm{\bar \mu}_{M(\mathcal{B})} \leq M_0$. Invoking \eqref{eq:non_empty_inline} we conclude that
\[
\frac{\mu_k(\bar{U}_{\hat{\imath}_k})}{M_0} \geq \min_{i=1,\ldots,N}
\frac{\mu_k(\bar{U}_i)}{M_0} \to \min_{i=1,\ldots,N}
\frac{\bar \lambda_i}{M_0} = 2c_1\,, \,\, \, \text{ as } \,\, k \to \infty\,,
\] 
which, together with \eqref{eq:prooffast:5} yields
\begin{align*}
r_j(\mu^s_k)\leq \left(1-sc_1+s^2 \,\frac{L_{\mu_0}c_2}{2}\right) r_j(\mu_k)\,, 
\end{align*}
for all $k$ sufficiently large and all $s \in [0,1]$. Subtracting $j(\bar \mu)$ from both sides of \eqref{eq:prooffast:0} yields
\begin{equation} \label{eq:prooffast:6}
r_j(\mu_{k+1}) \leq \f(s) r_j(\mu_k)  \,, \qquad \f(s):=1-sc_1+s^2 \,\frac{L_{\mu_0}c_2}{2}\,,
\end{equation}
for all $k$ sufficiently large and all $s \in [0,1]$. 
It is immediate to check that
\[
\min_{s \in [0,1]} \f(s) = 
\begin{cases}
1- \displaystyle c_1 + \frac{L_{\mu_0}c_2}{2} & \text{ if } \, c_1 >L_{\mu_0} c_2 \,,\\
1-\displaystyle \frac{c_1^2}{2L_{\mu_0} c_2}  & \text{ if } \, c_1 \leq L_{\mu_0} c_2 \,. \\
\end{cases}
\]
Notice that $\min_{s \in [0,1]} \f(s) \leq \zeta$, where
\begin{align*}
\zeta:= 1-\frac{c_1}{2} \min \left\{1,\frac{c_1}{L_{\mu_0} c_2}\right\}\,, \,\,\,\,\,\, \frac{3}{4} \leq \zeta <1 \,. 
\end{align*}
Thus, from \eqref{eq:prooffast:6} we obtain an integer $\bar k \in \N$ such that $r_j(\mu_{k+1}) \leq \zeta \, r_j(\mu_k)$ for all $k \geq \bar k$, 
establishing \eqref{eq:thm_fast_pdap}. In particular, 
$r_j(\mu_k) \leq r_{j}(\mu_{\bar k}) \zeta^{k-\bar k}$ for all $k \geq \bar k$, and the final thesis \eqref{eq:thm_fast_pdap_1} follows by setting $c:=r_{j}(\mu_{\bar k}) \zeta^{-\bar{k}}>0$.  
\end{proof}

\subsubsection{Proof of Theorem \ref{thm:fastconvpreview}} \label{subsec:fastconvproof} 
 
 Assume that Algorithm~\ref{alg:gcg} does not converge after finitely many steps and generates a sequence~$\{u_k\}_k$ in $\M$. By Proposition \ref{prop:convimpliconmu} both problems \eqref{def:minprob} and \eqref{def:sparseprob} admit a unique solution, given by $\bar{u}=\sum_{i=1}^N \bar{\lambda}_i \bar{u}_i$ and $\bar{\mu}=\sum_{i=1}^N \bar{\lambda}_i \delta_{\bar{u}_i}$, respectively. 
 Thanks to Theorem~\ref{thm:eqofpdapandgcg}, there exists a sequence~$\{\mu_k\}_k$ generated by Algorithm~\ref{alg:pdap} with~$u_k= \I(\mu_k)$. Invoking Theorem~\ref{thm:eqofpdapandgcg} as well as Theorem~\ref{thm:fast_pdap} yields $r_J(u_k) \leq r_j(\mu_k) \leq c \, \zeta^k$, for all $k$ sufficiently large, where $c>0$ and $\zeta \in [3/4,1)$. This shows \eqref{thm:convofcond:1rate}. 
In addition, Theorem~\ref{thm:fast_pdap} ensures that $\mu_k \weakstar \bar \mu$ along the whole sequence. 
Recalling that $\mathcal{I}$ is weak*-to-weak* continuous by Proposition~\ref{prop:I}~iii), we infer $u_k \weakstar \mathcal{I}(\bar{\mu})$. Since $\bar\mu$ minimizes in \eqref{def:sparseprob}, by Theorem~\ref{prop:equivalence}~ii), we infer that $\mathcal{I}(\bar \mu)$ minimizes in \eqref{def:minprob}. As the minimizer is unique, we conclude $\bar u = \mathcal{I}(\bar \mu)$.

\section{Conclusions}
We have introduced a fully-corrective generalized conditional gradient method (FC-GCG) to solve a class of non-smooth minimization problems in Banach spaces. For such algorithm we provided a global sublinear rate of convergence under the mild Assumptions~\ref{ass:func1}-\ref{ass:func3}, as well as a local linear convergence rate under the Assumptions \ref{ass:fast:F1}-\ref{ass:fast:F5}. Several examples of applications were considered, showing that it is possible to formulate a problem-dependent natural set of assumptions which is easy to verify and implies Assumptions \ref{ass:fast:F1}-\ref{ass:fast:F5}, thus ensuring linear convergence of our method. We demonstrate numerically the fast convergence for our first example, showing that, compared to standard conditional gradient methods (GCG), our algorithm exhibits a vastly improved rate of convergence. %
Additionally, we have discussed in details the computational burden of each presented example, and we have described viable strategies to solve the (partially) linearized problem \eqref{eq:updateAkintro}.

As in the euclidean case of~\cite{pieper}, we believe that all results in the present paper are transferable to non-smooth regularizers of the form~$\Phi(\mathcal{G}(\cdot))$ where~$\mathcal{G}$ is as in Assumptions~\ref{ass:func1}-\ref{ass:func3}, and~$\Phi$ is a suitable convex, monotone increasing function,  e.g., $\Phi(\mathcal{G}(u))=(1/2)\mathcal{G}(u)^2$. This generalization was omitted in the present paper for the sake of readability.

 \section*{Acknowledgements}
 
KB and SF gratefully acknowledge support by the Christian Doppler Research Association (CDG) and Austrian Science Fund (FWF) through the Partnership in Research
project PIR-27 ``Mathematical methods for motion-aware medical imaging'' and project P 29192 ``Regularization graphs for variational imaging''.
MC is supported by the Royal Society (Newton International Fellowship NIF\textbackslash R1\textbackslash 192048 Minimal partitions as a robustness boost for neural network classifiers). The Institute of Mathematics and Scientific Computing, to which KB and SF are affiliated, is a member of NAWI Graz (\texttt{https://www.nawigraz.at/en/}). The authors  KB and SF are further members of/associated with BioTechMed Graz (\texttt{https://biotechmedgraz.at/en/}).

\bibliography{bibliography}
 
\bibliographystyle{plain}

\appendix

\section{Complements to Section \ref{sec:condgrad}} \label{append:A1}

In this section we state and prove two results. The first concerns well-posedness for the minimization problem  at~\eqref{eq:descr_insertion}, while the second shows 
equivalence, in a suitable sense, of~\eqref{eq:subprobalgo} and~\eqref{eq:subprobcoeffs}.

\begin{lemma} \label{lem:existoflinearized}
Let~$p \in \Cc$ be given. Then, there exists~$ \bar v \in \operatorname{Ext}(B)$ with
\begin{align*}
\langle p, \bar v \rangle= \max_{v \in \mathcal{B}}\, \langle p, v \rangle = \max_{v \in B} \, \langle p, v \rangle\,.
\end{align*}
\end{lemma}

\begin{proof}
In order to show the claimed result we first prove that the maximization problem
\begin{equation} \label{eq:existofextrem_aux}
\max_{v \in B} \, \langle p, v \rangle
\end{equation}
admits a solution $\bar v \in \Ext{B}$. This statement is classical, but for the reader's convenience we produce a proof. Let $v^*$ be a maximizer of \eqref{eq:existofextrem_aux}, which exists since $v \mapsto  \langle p, v \rangle$ is weak* continuous and $B$ is weak* compact. Consider the set
    $H:=\{\,v \in B \; | \; \langle p, v \rangle=\langle p, v^* \rangle\,\}$. Then $\Ext{H} \neq \emptyset$ by Krein-Milman's Theorem, since $H\subset B$ is convex, non-empty, weak* closed and thus weak* compact. Thus, let $\bar v \in \ext(H)$ be given.  We claim that $\bar v \in \Ext{B}$. Indeed, assume that $\bar v = s v_1 + (1-s)v_2$ for $v_1,v_2 \in B$, $s \in (0,1)$. Notice that $v_1,v_2 \in H$, otherwise we would have $ \langle p, \bar v \rangle <  \langle p, v^* \rangle$, contradicting the maximality of~$\bar{v}$. As $\bar v \in \Ext{H}$ we then have $\bar v=v_1=v_2$, showing that $\bar{v} \in \Ext{B}$. The claimed statement now follows from $\Ext{B} \subset \mathcal{B} \subset B$ ending the proof. 
\end{proof}

\begin{proposition}\label{prop:eqgauge}
Let $\mathcal{A}^{u,+}_k=\{u^k_i\}^{N^+_k}_{i=1} \subset \ext(B)$ be given.
If~$\widehat{\lambda} \in  \R_+^{N^+_k}$ solves~\eqref{eq:subprobcoeffs}, then
$\widehat{u}=\sum^{N^+_k}_{i=1} \widehat{\lambda}_i u^k_i$
solves \eqref{eq:subprobalgo}. Conversely, if~$\widehat{u}$ solves~\eqref{eq:subprobalgo}, then there exists a solution~$\widehat{\lambda} \in \R_+^{N^+_k}$ to \eqref{eq:subprobcoeffs} such that
$\widehat{u}= \sum^{N^+_k}_{i=1} \widehat{\lambda}_i u^k_i$.
\end{proposition}

\begin{proof}
Let~$\widehat{\lambda} \in \R_+^{N^+_k}$ be a solution to~\eqref{eq:subprobcoeffs}. Set~$\widehat{u}=\sum^{N^+_k}_{i=1} \widehat{\lambda}_i u^k_i$ and fix an arbitrary~$u \in \operatorname{cone}(\mathcal{A}^{u,+}_k)$. Then there exists~$\lambda^u \in \R_+^{N^+_k}$ such that
$u= \sum^{N^+_k}_{i=1} \lambda^u_i u^k_i$ with $\sum^{N^+_k}_{i=1} \lambda^u_i = \kappa_{\mathcal{A}^+_k}(u)$.
We estimate
\begin{align*}
F(K\widehat{u})+ \kappa_{\mathcal{A}^+_k}(\widehat{u}) & \leq F\left(\sum^{N^+_k}_{i=1} \widehat{\lambda}_i K u^k_i\right)+ \sum^{N^+_k}_{i=1} \widehat{\lambda}_i  \\
& \leq F\left(\sum^{N^+_k}_{i=1} \lambda^u_i K u^k_i\right)+ \sum^{N^+_k}_{i=1} \lambda^u_i = F(Ku)+\kappa_{\mathcal{A}^+_k}(u)\,,
\end{align*}
showing that~$\widehat{u}$ is a solution to~\eqref{eq:subprobalgo}. 
Conversely, let~$\widehat{u}$ be a minimizer to~\eqref{eq:subprobalgo}. Since~$\widehat{u} \in \operatorname{cone}(\mathcal{A}^{u,+}_k)$, there exists~$\widehat{\lambda} \in \R_+^{N^+_k}$ with
$\widehat{u}= \sum^{N^+_k}_{i=1} \widehat{\lambda}_i u^k_i$, $\sum^{N^+_k}_{i=1} \widehat{\lambda}_i = \kappa_{\mathcal{A}^{u,+}_k}(\widehat{u})$. 
Now let~$\lambda \in \R_+^{N^+_k}$ be given and set~$u^\lambda= \sum^{N^+_k}_{i=1} \lambda_i u^k_i$. From the optimality of~$\widehat{u}$ and the definition of the gauge function we get
\begin{align*}
F\left(\sum^{N^+_k}_{i=1} \widehat{\lambda}_i K u^k_i\right)+ \sum^{N^+_k}_{i=1} \widehat{\lambda}_i  & \leq F(Ku^\lambda)+ \kappa_{\mathcal{A}^{u,+}_k}(u^\lambda) \\
 & \leq F\left(\sum^{N^+_k}_{i=1} \lambda_i K u^k_i\right)+ \sum^{N^+_k}_{i=1} \lambda_i\,.
\end{align*}
Thus,~$\widehat{\lambda}$ is a minimizer of~\eqref{eq:subprobcoeffs}. 
\end{proof}

\section{Complements to Section \ref{sec:examples}}\label{app:guidingexample}

Here we collect the technical statements and the proofs of Section \ref{sec:examples}. For sake of clarity we place the proofs of each example in different subsections.

\subsection{Sparse source identification (Section \ref{subsec:guideex})}

The next lemma justifies the well-posedness of the definition of $K_*$.
\begin{lemma}
\label{lem:regadjointheat}
Let~$\varphi \in L^2(\Omega)$ be given and let~$z \in L^2(0,T;H^1_0(\Omega)) \cap H^1(0,T;H^{-1}(\Omega))$ the associated unique solution of \eqref{eq:heatadjoint}. Moreover, let~$\Omega_0$ be a subdomain with~$\bar{\Omega}_0 \subset \Omega$. Then, we have~$z(0) \in C_0(\Omega)\cap C^2(\Omega_0)$ and
\begin{align*}
\|z(0)\|_{\mathcal{C}}+ \|z(0)\|_{C^2(\Omega_0)} \leq c \|\varphi\|_{L^2(\Omega)}
\end{align*}
for some~$c>0$ independent of~$\varphi \in L^2(\Omega)$.
\end{lemma}
\begin{proof}
This follows from~\cite[Lemma 3.1]{leyk} and the Sobolev embedding~$H^2(\Omega) \cap H^1_0(\Omega) \hookrightarrow C_0(\Omega)$.
\end{proof}

\subsubsection{Proof of Lemma \ref{lem:quadgrowthmeasure}}\label{app:lem:quadgrowthmeasure}

We start by noting that~$\nabla \bar{z}(0)(\bar{x}_i)=0$ thanks to \ref{ass:guidingG1} and Proposition~\ref{prop:optheat}. Now fix~$x \in \Omega$. By Taylor's expansion we infer that
\begin{align*}
\beta-|\bar{z}(0)(x)| & = |\bar{z}(0)(\bar{x}_i)|-|\bar{z}(0)(x)|\\
& =- (1/2)\operatorname{sign}(\bar{z}(0)(\bar{x}_i))( (x-\bar{x}_i),\nabla^2 \bar{z}(0)(x_s) (x-\bar{x}_i))_{\R^2}\,,
\end{align*}
for some~$x_s=(1-s)\bar{x}_i+ sx$,~$s\in(0,1)$. Due to the continuity of~$\nabla^2 \bar{z}(0)$, see Lemma~\ref{lem:regadjointheat}, as well as Assumption $ \ref{ass:guidingG2}$, there exists~$R>0$ such that~$B_R(\bar{x}_i) \subset \Omega$ and
\begin{align*}
 \operatorname{sign}(\bar{z}(0)(\bar{x}_i))( (x-\bar{x}_i),\nabla^2 \bar{z}(0)(x_s) (x-\bar{x}_i))_{\R^2} \leq -(\gamma/2) |x-\bar{x}_i|^2 \,,
\end{align*}
for all $x \in B_R(\bar{x}_i)$. 
Combining both observations finally yields
\begin{align*}
\beta-|\bar{z}(0)(x)| \geq (\gamma/4)|x-\bar{x}_i|^2 \,,
\end{align*}
for all $x \in B_R(\bar{x}_i)$, showing \eqref{eq:heatest1}.
Moreover for every~$x \in B_R(\bar{x}_i) $ we readily verify
\begin{align*}
\ynorm{K(\delta_{x}-\delta_{\bar{x}_i})} & = \sup_{\ynorm{y}\leq 1} (y, K(\delta_{x}-\delta_{\bar{x}_i}))_Y \\ 
& = \sup_{\ynorm{y}\leq 1} \lbrack K_*y \rbrack(x)-\lbrack K_*y \rbrack(\bar{x}_i) \\ &\leq \sup_{\ynorm{y}\leq 1} \|{K_*y}\|_{\operatorname{Lip}(B_R(\bar{x}_i))} |x-\bar{x}_i| \\ & \leq \|K_*\|_{Y, \operatorname{Lip}(B_R(\bar{x}_i))} |x-\bar{x}_i|\,.
\end{align*}
To finish, $R>0$ can be chosen small enough to ensure
\[
\operatorname{sign}(\bar{z}(0)(x))=\operatorname{sign}(\bar{z}(0)(\bar{x}_i))\,,
\] 
for all~$x \in B_R(\bar{x}_i)$, due to~$\bar{z}(0)(\bar{x}_i)\neq 0$ and~$\bar{z}(0)\in C_0(\Omega)$. \qed

\subsubsection{Proof of Lemma 
\ref{lem:equivalenceofweaktog}}\label{app:lem:equivalenceofweaktog}

First assume that~$g(u_k,\bar{u}_i) \to 0$. Then, we have~$x_k \rightarrow \bar{x}_i$ in~$\Omega$ as well as~$\sigma_k \rightarrow \operatorname{sign}(\bar{z}(0)(\bar{x}_i))$. Consequently, for all $p \in C_0(\Omega)$ we have
\begin{align*}
\lim_{k \to \infty} \ \langle p,u_k  \rangle & = \lim_{k \to \infty} \ \sigma_k \beta^{-1} p(x_k) \\ 
& = \operatorname{sign}(\bar{z}(0)(\bar{x}_i)) \beta^{-1} p(\bar{x}_i)= \langle p, \bar{u}_i \rangle \,. 
\end{align*}
This implies~$u_k \weakstar \bar{u}_i$. For the other direction assume that $u_k \weakstar \bar{u}_i$. 
Note that from every subsequence of  $(\sigma_k, x_k)$ we can extract a further convergent subsequence still denoted by~$(\sigma_k, x_k)$, relabeling the indices. Let~$(\bar{\sigma}, \bar{x}) \in \{-1,+1\} \times \bar \Omega$ denote its limit. As~$\bar{x}\in \partial \Omega$ would imply~$u_k \weakstar 0 \neq \bar{u}_i$, we conclude~$\bar{x}\in \Omega$. Then, using again the weak* convergence of~$u_k$, we get
\begin{align*}
\bar{\sigma} \beta^{-1} p(\bar{x})= \operatorname{sign}(\bar{z}(0)(\bar{x}_i)) \beta^{-1} p(\bar{x}_i)  \,,
\end{align*}
for all $p \in C_0(\Omega)$, from which we immediately conclude~$\bar{\sigma}= \operatorname{sign}(\bar{z}(0)(\bar{x}_i))$ and~$\bar{x}=\bar{x}_i$. Since the initial subsequence was chosen arbitrary we get~$(\sigma_k, x_k)\rightarrow (\operatorname{sign}(\bar{z}(0)(\bar{x}_i)), \bar{x}_i) $ for the whole sequence and thus,~$g(u_k,\bar{u}_i) \to 0$. Finally, since~$\sigma_k \in \{-1,+1\}$ for all~$k \in\N$ and~$\sigma_k \rightarrow \operatorname{sign}(\bar{z}(0)(\bar{x}_i))$, we necessarily have~$\sigma_k=\operatorname{sign}(\bar{z}(0)(\bar{x}_i))$ for all~$k \in \N$ sufficiently large. \qed

\subsection{Rank-one matrix reconstruction by trace regularization (Section \ref{sec:rank-one})}

\subsubsection{Proof of Lemma \ref{lem:exttrace} }\label{app:lem:exttrace}

Recall that every extremal point~$\mathcal{U}$ of~$B$ necessarily satisfies~$\operatorname{Tr}(\mathcal{U})=\beta^{-1}$ or~$\operatorname{Tr}(\mathcal{U})=0$. We first focus on those extremal points with~$\operatorname{Tr}(\mathcal{U})=\beta^{-1}$.
According to Lidskii's theorem, for every~$\mathcal{U} \in B $ there holds~$\operatorname{Tr}(\mathcal{U})=\sum_{i \in \mathbb{N}} \sigma^\mathcal{U}_i$.
From these observations, we already see that every extremal point of~$B$ has at most rank one. Indeed, if~$\mathcal{U} \in B$,~$\operatorname{Tr}(\mathcal{U})=\beta^{-1}$, is at least of rank two, i.e.~$\sigma^\mathcal{U}_1 \geq \sigma^\mathcal{U}_2 >0$, then we can decompose~$\mathcal{U}$ as
\begin{align*}
    \mathcal{U} &= (\beta \sigma^\mathcal{U}_1) \beta^{-1} h^\mathcal{U}_1 \otimes h^\mathcal{U}_1+  \left(\beta \sum_{j \in \N, j \neq 1} \sigma^\mathcal{U}_j \right) \sum_{i \in \N, i \neq 1} \left(\beta \sum_{j \in \N, j \neq 1} \sigma^\mathcal{U}_j \right)^{-1} \sigma^\mathcal{U}_i  h^\mathcal{U}_i \otimes h^\mathcal{U}_i \\
    &= (\beta \sigma^\mathcal{U}_1) \beta^{-1} h^\mathcal{U}_1 \otimes h^\mathcal{U}_1+  (1-\beta \sigma^\mathcal{U}_1 ) \sum_{i \in \N, i \neq 1} \left(\beta \sum_{j \in \N, j \neq 1} \sigma^\mathcal{U}_j \right)^{-1} \sigma^\mathcal{U}_i  h^\mathcal{U}_i \otimes h^\mathcal{U}_i \\
    &:= (\beta \sigma^\mathcal{U}_1)\mathcal{U}_1+ (1-\beta \sigma^\mathcal{U}_1) \mathcal{U}_2
\end{align*}
where~$(\beta \sigma^\mathcal{U}_1) \in (0,1)$ and the second equality follows from
\begin{align*}
    \beta^{-1}= \operatorname{Tr}(\mathcal{U})= \sum_{i \in \N} \sigma^\mathcal{U}_i. 
\end{align*}
Noting that~$\mathcal{U}_1,\mathcal{U}_2 \in B$,
as well as~$\mathcal{U}_1 \neq \mathcal{U}_2$, we conclude that~$\mathcal{U}$ is not extremal. Hence, if~$\mathcal{U} \in B$,~$\mathcal{G}(\mathcal{U})=1$, is extremal then we have~$\mathcal{U}= \beta^{-1} h \otimes h $,~$\|h\|_H=1$. It remains to show that every operator of this form is actually an extremal point. For this purpose assume that there is $\mathcal{U}= \beta^{-1} h \otimes h $,~$\|h\|_H=1$, $s\in(0,1)$ as well as~$\mathcal{U}_1,\mathcal{U}_2 \in B$, with~$\mathcal{U}=(1-s)\mathcal{U}_1+s \mathcal{U}_2$. Then we also have
\begin{align} \label{eq:convex1}
    1= (1-s) \sum_{i\in \N} \beta \sigma^{\mathcal{U}_1}_i  (h^{\mathcal{U}_1}_i,h)^2_H+ s \sum_{i\in \N} \beta \sigma^{\mathcal{U}_2}_i  (h^{\mathcal{U}_2}_i,h)^2_H \,,
\end{align}
as well as
\begin{align}\label{eq:convex2}
    \max \left\{\sum_{i\in \N} \beta \sigma^{\mathcal{U}_1}_i, \sum_{i\in \N} \beta \sigma^{\mathcal{U}_2}_i\right\} \leq 1.
\end{align}
Moreover, we have~$(h^{\mathcal{U}_j}_i,h)^2_H \leq 1$ for all~$i \in \N$,~$j=1,2$. Together with~\eqref{eq:convex1} and~\eqref{eq:convex2}, this yields
\begin{align} \label{eq:convex3}
    1=\sum_{i\in \N} \beta \sigma^{\mathcal{U}_1}_i  (h^{\mathcal{U}_1}_i,h)^2_H= \sum_{i\in \N} \beta \sigma^{\mathcal{U}_2}_i  (h^{\mathcal{U}_2}_i,h)^2_H.
\end{align}
We will now prove that these observations imply~$\mathcal{U}_1=\beta^{-1}h \otimes h$, the proof for~$\mathcal{U}_2$ follows by the same steps.
Note that~$(h^{\mathcal{U}_1}_i,h)^2_H = 1$ holds if and only if we have~$h^{\mathcal{U}_1}_i=h$ or~$h^{\mathcal{U}_1}_i=-h$. Recalling that~$\{h^{\mathcal{U}_1}_i\}_{i \in \N }$ is an orthonormal set,~\eqref{eq:convex3} implies the existence of a unique index~$\bar{\imath}$ with 
\begin{align*}
    (h^{\mathcal{U}_1}_{\bar{\imath}},h)^2_H = 1 \quad \text{as well as}~(h^{\mathcal{U}_1}_{i},h)^2_H=0,~i \neq \bar{\imath}.
\end{align*}
Together with \eqref{eq:convex2} and~\eqref{eq:convex3}, we finally get
\begin{align*}
    1=\sum_{i\in \N} \beta \sigma^{\mathcal{U}_1}_i  (h^{\mathcal{U}_1}_i,h)^2_H= \beta \sigma^{\mathcal{U}_1}_{\bar{\imath}}
\end{align*}
as well as~$\sigma^{\mathcal{U}_1}_i=0$,~$i \neq \bar{\imath}$. Consequently, there holds
\begin{align*}
    \mathcal{U}_1= \sigma^{\mathcal{U}_1}_{\bar{\imath}} h^{\mathcal{U}_1}_{\bar{\imath}} \otimes h^{\mathcal{U}_1}_{\bar{\imath}}=\beta^{-1}h \otimes h.
\end{align*}
Repeating these arguments for~$\mathcal{U}_2$ implies that $\mathcal{U}$ is an extremal point.

Finally,~$\operatorname{Tr}(\mathcal{U})=0$ holds if and only if~$\mathcal{U}=0$. Assume that~$\mathcal{U} = 0$ is not extremal, i.e., there are operators~$\mathcal{U}_1, \mathcal{U}_2 \in B$,~$\mathcal{U}_1\neq \mathcal{U}_2$, as well as~$s \in (0,1)$ with~$0=(1-s)\mathcal{U}_1+ s \mathcal{U}_2$. Then, without loss of generality, we have~$\sigma^{\mathcal{U}_1}_1 >0$ and thus
\begin{align*}
    0= (h^{\mathcal{U}_1}_1, ((1-s)\mathcal{U}_1+ s \mathcal{U}_2) h^{\mathcal{U}_1}_1)_H \geq \sigma^{\mathcal{U}_1}_1>0
\end{align*}
yielding a contradiction. Hence~$0$ is an extremal point and the proof is finished.

It remains to characterize~$\mathcal{B}$. Since~$H$ is infinite dimensional, for every~$h\in H$,~$\|h\|_H \leq 1$, there exists a weakly convergent sequence~$\{h_k\}_k$ with limit~$h$ such that $\|h_k\|_H =1$ for all~$k$. Setting~$\mathcal{U}_k= \beta^{-1} h_k \otimes h_k$ and~$\mathcal{U}= \beta^{-1} h \otimes h$, we then also conclude
\begin{align*}
    \langle P, \mathcal{U}_k \rangle =\beta^{-1}(h_k,P h_k) \rightarrow \beta^{-1}(h,P h)=\langle P, \mathcal{U} \rangle 
\end{align*}
for all $P \in K(H)$, 
using compactness of $P$. Hence,~$\mathcal{U}_k \weakstar U $ and thus~$ \mathcal{U} \in \mathcal{B}$. For the converse inclusion, let~$\mathcal{U}_k= \beta^{-1} h_k \otimes h_k \in B$ with $\|h_k\|_H= 1$ denote a weak* convergent sequence with limit~$\mathcal{U}$. Then there exists a subsequence of~$\{h_k\}_k$, denoted by the same index, which converges weakly to some~$h \in H $,~$\|h\|_H \leq 1$. Repeating previous arguments and noting the uniqueness of weak* limits, we then arrive at~$\mathcal{\mathcal{U}}= \beta^{-1} h \otimes h$. \qed

\subsubsection{Proof of Proposition \ref{prop:opttrace}}\label{app:prop:opttrace}

The optimality condition in~\eqref{eq:optimalitytrace} follows immediately from Proposition~\ref{prop:existu}, the characterization of the extremal points as well as~$\sigma^{\bar{P}}_1= \sup_{\|h\|_H=1}(h, \bar{P}h)_H$. Now, let~$\bar{\mathcal{U}}$ be a minimizer of~\eqref{def:rankproblem} and assume that~$\sigma^{\bar{P}}_1 = \beta$. Then we have
\begin{align*}
    \operatorname{Tr}(\bar{P} \bar{\mathcal{U}})= \sum_{i \in \N}  \sigma^{\bar{\mathcal{U}}}_i (h^{\bar{\mathcal{U}}}_i,\bar{P} h^{\bar{\mathcal{U}}}_i)_H =\beta \operatorname{Tr}(\bar{\mathcal{U}})=\beta \sum_{i \in \N}  \sigma^{\bar{\mathcal{U}}}_i,
\end{align*}
as well as~$(h^{\bar{\mathcal{U}}}_i,\bar{P} h^{\bar{\mathcal{U}}}_i)_H \leq \beta$,~$i \in\N$. Consequently, we conclude that~$\sigma^{\bar{\mathcal{U}}}_i>0$ implies~$\bar{P} h^{\bar{\mathcal{U}}}_i=\beta h^{\bar{\mathcal{U}}}_i$,~i.e.,~$h^{\bar{\mathcal{U}}}_i$ is an eigenfunction for the leading eigenvalue of~$\bar{P}$.
Thus, since~$\{h^{\bar{\mathcal{U}}}_i\}_{i}$ form an ONB, and possibly after a change of basis,~$\bar{\mathcal{U}}$ is of the form~\eqref{eq:representationtrace}. \qed

\subsubsection{Proof of Theorem \ref{thm:rankprobquad}}\label{app:thm:rankprobquad}

The statement in~\eqref{eq:linearminoverrank}, the structure of~$\bar{\mathcal{U}}$ and its uniqueness as well as the Lipschitz result on~$K$ follow immediately from Assumption~\ref{ass:rankone}, the assumptions on~$K$ and the strict convexity of~$F$, see Assumption~\ref{ass:functions}. Let~$\mathcal{U} \in B$ with~$\operatorname{Tr}(\mathcal{U})=\beta^{-1}$ and $\mathcal{U} \geq 0$ be arbitrary but fixed. For the proof of the quadratic growth behavior we follow similar steps as in the finite dimensional setting, see~\cite[Lemma 4]{garber}. Setting~$\delta= \sigma^{\bar{P}}_1- \sigma^{\bar{P}}_2 >0$, we estimate
\begin{align*}
    1-\langle \bar{P}, \mathcal{U} \rangle &=1- \beta(h^{\bar{P}}_1 ,\mathcal{U} h^{\bar{P}}_1)_H-\sum_{i\in \N, i >1} \sigma^{\bar{P}}_i (h^{\bar{P}}_i ,\mathcal{U} h^{\bar{P}}_i)_H \\
    &  \geq 1-\beta\sum_{i\in \N}  (h^{\bar{P}}_i ,\mathcal{U} h^{\bar{P}}_i)_H +\delta  \sum_{i\in \N, i >1}  (h^{\bar{P}}_i ,\mathcal{U} h^{\bar{P}}_i)_H \\
    &= \delta \left( \beta^{-1}- (h^{\bar{P}}_1 ,\mathcal{U} h^{\bar{P}}_1)_H \right) =\delta \beta \left( \beta^{-2}-  (\mathcal{U}, \bar{\mathcal{U}}_1)_{\text{HS}} \right)
    \\&=\delta \beta \left( \frac{1}{2}\operatorname{Tr}(\mathcal{U})^2+\frac{1}{2} \operatorname{Tr}(\bar{\mathcal{U}}_1)^2-  (\mathcal{U}, \bar{\mathcal{U}}_1)_{\text{HS}} \right) \\
    & \geq \delta \beta \left(\frac{1}{2} \|\mathcal{U}\|_{\text{HS}}^2 +\frac{1}{2}\|\bar{\mathcal{U}}_1\|_{\text{HS}}^2-  (\mathcal{U}, \bar{\mathcal{U}}_1)_{\text{HS}} \right)\\& = (\delta \beta/2) \|\mathcal{U}-\bar{\mathcal{U}}_1\|^2_{\text{HS}},
\end{align*}
where we use~$\operatorname{Tr}(\mathcal{U}) \geq \|\mathcal{U}\|_{\text{HS}}$ for all~$\mathcal{U}\in B$ in the last inequality. This finishes the proof. \qed

\subsection{Minimum effort problems (Section \ref{subsec:mineffort})}

\subsubsection{Proof of Lemma \ref{prop:extremalsoflinf}}\label{app:prop:extremalsoflinf}

Without loss of generality assume $\alpha=1$. The characterization  \eqref{lem:extremalsoflinf:1} is classical \cite[page 144]{conway}. 
 Moreover the inclusion $\overline{\ext(B)}^* \subset B$ follows from weak* lower semi-continuity of the norm. Let us now prove that $B \subset \overline{\ext(B)}^*$. To this end, set $R:=(0,1)^d$ and $R_\lambda :=(0,\sqrt[d]{\lambda})^d$ for any $\lambda \in [0,1]$. Define the map $h^\lambda$ as $h^\lambda:=1$ in $R\setminus R_\lambda$, $h^\lambda:=-1$ in $R_\lambda$, and extend it $R$-periodically to the whole $\R^d$. For $k \in \N$, $x \in \R^d$, set $h^\lambda_k(x):=h^\lambda(kx)$. Then 
 \begin{equation} \label{ball-murat}
 h^\lambda_k \weakstar \int_R h^\lambda \, dx = 1-2\lambda \,\,\, \text{ weakly*  in } \,\, L^\infty(A)
 \end{equation}
 as $k \to \infty$, for all open, bounded and non-empty sets $A \subset \R^d$, see \cite[Lemma A.1]{ball_murat}. Let $\{Q_{n,j}\}_{j \in \N}$ be a family of pairwise disjoint and open $d$-cubes of side length $1/n$ covering $(\R \setminus (\mathbb{Z}/n))^d$. Define $\tilde{Q}_{n,j} := Q_{n,j} \cap \Om$ and let $J_n$ be the collection of indices $j \in \N$ such that $\tilde{Q}_{n,j} \neq \emptyset$. Let $\mathcal{P}_n$ be the set of maps $u \in B$ such that $u=c_{n,j}$ in $\tilde{Q}_{n,j}$ for some $c_{n,j} \in [-1,1]$ and for all $j \in J_n$. Fix $u \in \mathcal{P}_n$ and $n \in \N$. For every $j \in J_n$, $k \in \N$ define $u_k(x):=h_k^{\lambda_{n,j}}(x)$ for all $x \in \tilde{Q}_{n,j}$, where $\lambda_{n,j}:=(1-c_{n,j})/2$. Therefore $\{u_k\}_k \subset \ext(B)$. By \eqref{ball-murat} we have that $u_k \weakstar c_{n,j}$ weakly* in $L^\infty(\tilde{Q}_{n,j})$ as $k \to \infty$, for all $j \in J_n$. Thus  $u_k \weakstar u$ in $L^\infty(\Om)$. In particular we have shown that $\mathcal{P}_n \subset \overline{\ext(B)}^*$. Therefore $\mathcal{P} \subset \overline{\ext(B)}^*$, where $\mathcal{P}:=\cup_{n \in \N} \, \mathcal{P}_n$. As clearly $\overline{\mathcal{P} }^*=B$, the proof is concluded.   \qed

\subsubsection{Proof of Lemma \ref{prop:Optimalitymineffort}}\label{app:prop:Optimalitymineffort}

The optimality condition in~\eqref{def:propmineff} follows immediately from Proposition~\ref{prop:existu} by pointing out
\begin{align}
\int_{\Omega} |\bar{p}(x)|~\de x & = \alpha \,   \max_{v \in B }  \int_\Om \bar{p}(x) v(x)~\mathrm{d}x \\ 
& = \alpha \, \max_{v \in \Ext B } \ \langle \bar{p}, v \rangle \leq \alpha.
\end{align}
This also implies
\begin{align*}
    \alpha \|\bar{u}\|_\infty=\int_{\Omega} \bar{p}(x) \bar{u}(x)~\de x \leq \|\bar{u}\|_\infty\int_{\Omega} |\bar{p}(x)|~\de x \leq \alpha \|\bar{u}\|_\infty,
\end{align*}
hence, equality holds everywhere. If~$\bar{u} \neq 0$, this can only be the case if
\begin{align*}
    \int_{\Omega} |\bar{p}(x)|~\de x=\alpha \quad \text{and} \quad \frac{\bar{u}(x)}{\|\bar{u}\|_\infty} \in \begin{cases}
    \{1\} & \text{if } \, \bar{p}(x)>0  \\
    \{-1\} & \text{if } \,\bar{p}(x)<0 \\
    [-1, 1] & \text{else},
    \end{cases},
\end{align*}
for a.e.~$x \in\Omega$, yielding~\eqref{eq:representationlinf}.\qed

\section{Complements to Section \ref{sec:lifting} and Section \ref{sec:convergence}}

\subsection{Proofs of Section \ref{sec:lifting}} \label{app:proofs_sec1}
 
In this section we exhibit proofs for Propositions \ref{prop:I},  \ref{prop:def_K}. In order to define the map $\I$ at Proposition \ref{prop:I} we employ the classical Choquet's Theorem \cite[Page 14]{phelps}, which we recall for reader's convenience.  
 
\begin{theorem}[Choquet] \label{thm:real_choquet}
Let $X$ be a locally convex space and $K \subset X$ be a metrizable compact convex subset. For each $v_0 \in K$ there exists a probability measure $\mu$ over $X$ concentrated on $\Ext{K}$ which represents $v_0$, that is,
\[
T(v_0)= \int_{X} T(v) \, \mathrm{d}\mu(v) \,,
\]
for all $T$ in the topological dual of $X$. 
\end{theorem}
 
We are now ready to prove Proposition \ref{prop:I}.
 
 \subsubsection{Proof of Proposition \ref{prop:I}}
Let $\mu \in M^+(\mathcal{B})$ and consider the embedding $f \colon \mathcal{B} \to \M$, $f(v):=v$.     We have that $f$ is weak* $\mu$-measurable, since $v \mapsto \langle p,v \rangle$ is weak* continuous and hence, $\mu$-measurable for all $p \in \Cc$. Moreover, $v \mapsto \langle p,v \rangle$ is $\mu$-integrable for each $p \in \Cc$, since 
\[
\begin{aligned}
\int_{\mathcal{B}} | \langle p,v \rangle|~\de\mu(v) & \leq  \nor{p}_{\Cc} \int_{\mathcal{B}} \nor{v}_{\M}~\de\mu(v) \\
&  \leq  \nor{p}_{\Cc}
\nor{\mu}_{M(\mathcal{B})} \max_{v \in \mathcal{B}} \nor{v}_{\mathcal{M}} \,,
\end{aligned}
\]
and the last term is finite, since $\mathcal{B}$ is norm bounded. Therefore, $f$ is Gelfand integrable over $\mathcal{B}$ (see \cite[Thm 11.52]{aliprantis}), that is, there exists a unique $u \in \mathcal{M}$ such that \eqref{eq:choquet} holds. We thus set $\I(\mu):=u$. Recalling that the weak* topology separates points, we immediately conclude that $\I$ is well defined and
\begin{equation*}\label{eq:linearityIapp}
\I(\lambda_1 \mu_1 + \lambda_2 \mu_2) = \lambda_1 \I(\mu_1) + \lambda_2 \I(\mu_2) \,,
\end{equation*}
for all $\lambda_1,\lambda_2 \in \R_{+}$, $\mu_1,\mu_2 \in M^+(\mathcal{B})$.  
We now show i). Notice that $\I(0)=0$ and the estimate holds trivially. If $\mu \neq 0$, we can apply Theorem 11.54 in \cite{aliprantis} to obtain 
\[
\frac{1}{\norm{\mu}_{M(\mathcal{B})}} \, \I(\mu) \in \overline{\operatorname{conv}(\mathcal{B})}^* \,.
\]
By recalling that $B$ is weak* compact and convex, we conclude that $\overline{\operatorname{conv}(\mathcal{B})}^* \subseteq B$. Thus, by 1-homogeneity of $\G$ we conclude $\G(\I(\mu)) \leq \norm{\mu}_{M(\mathcal{B})}$
 and, in particular, $\I(\mu) \in \operatorname{dom} \G$.
 
For the proof of ii), let $u \in \operatorname{dom} \G$ be fixed. If $\G(u)=0$ then $u=0$, and the measure $\mu =0$ satisfies the statement. Assume now $\G(u) > 0$. Apply Theorem \ref{thm:real_choquet} with $X=(\mathcal{M},weak^*)$, $K=B$, $v_0=u/\G(u)$ to obtain a probability measure $\tilde{\mu}$ over $X$ concentrated on $\Ext{B}$ and such that
\begin{equation} \label{eq:prop:I:proof}
\langle p , u \rangle  = \int_{\M} \langle p , v \rangle \G(u)~\de\tilde{\mu}(v)  \,,
\end{equation}
for all $p \in \Cc$, 
since $(\mathcal{M},weak^*)^*=\Cc$. Set $\mu:=\G(u) \tilde{\mu} \zak \mathcal{B}$. Since $\mathcal{B}$ is weak* closed, we have that $\mu \in M^+(\mathcal{B})$. Moreover, $\nor{\mu}_{M(\mathcal{B})}=\G(u)$, $\mu$ is concentrated on $\Ext{B}$ and it represents $u$. As the weak* topology separates points, from \eqref{eq:prop:I:proof} we conclude that $\I(\mu)=u$. This shows in particular that $\I$ is surjective, achieving the proof of ii).

Finally, $\mathcal{I}$ is weak*-to-weak* continuous: indeed if $\mu_k \weakstar \mu$ weakly* in $M(\mathcal{B})$, then by definition of weak* convergence we have $\sm{P}{\mu_k} \to \sm{P}{\mu}$ for all $P \in C(\mathcal{B})$. Note that the map $P(v)=\langle p, v \rangle$, $v \in \mathcal{B}$, $p \in \Cc$ belongs to $C(\mathcal{B})$. Thus, for all $p \in \Cc$,
\[
\int_{\mathcal{B}} \langle p, v \rangle  \, \mathrm{d}\mu_k(v)  \to \int_{\mathcal{B}} \langle p, v \rangle  \, \mathrm{d}\mu(v)  \,.
\]
Thanks to \eqref{def:I} the above reads $\mathcal{I}(\mu_k) \weakstar \mathcal{I}(\mu)$ weakly* in $\M$, concluding the proof. \qed \\

Finally, we prove Proposition \ref{prop:def_K}.
 
\subsubsection{Proof of Proposition \ref{prop:def_K}}
For $\mu \in M(\mathcal{B})$ define the functional $T_\mu \colon Y \to \R$ by setting
\[
T_\mu (y):= \int_{\mathcal{B}}  (Kx,y)_Y \, \mathrm{d}\mu(x)\,,
\]    
for all $y \in Y$. Notice that $T_\mu$ is well defined, since the map $v \mapsto (Kv,y)_Y= \langle K_*y, v \rangle$ is weak* continuous over $\M$, and hence $\mu$-measurable. It is clear that $T_\mu$ is linear. Moreover, $T_\mu$ is continuous. Indeed, it holds
\begin{align*}
\nor{T_\mu (y)} & \leq \int_{\mathcal{B}}  |(Kv,y)_Y | \, \mathrm{d}|\mu|(v) \\ & \leq 
\nor{K} \nor{y}_Y \,\sup_{v \in \Bb}\,\nor{v}_{\M} \int_{\mathcal{B}}  1 \, \mathrm{d}|\mu|(v) \\ & = C \nor{K} \nor{\mu}_{M(\mathcal{B})} \nor{y}_Y ,
\end{align*}

where we recalled that the constant $C$ is defined at i) and is finite, since $\mathcal{B}$ is norm bounded.  
Therefore, by Riesz's Theorem, there exists a unique element in $Y$ representing $T_\mu$, thus defining a map $\mathcal{K} \colon M(\mathcal{B}) \to Y$. In particular, $\mathcal{K}$ is linear and satisfies \eqref{eq:def_K}. Moreover, by \eqref{eq:def_K} we have
\[
\nor{\mathcal{K}\mu}_Y = \sup_{\nor{y}_Y \leq 1}  \nor{T_\mu (y)} \leq C \nor{K} \nor{\mu}_{M(\mathcal{B})},
\]  
by the above, so~$\mathcal{K}$ is bounded and~i) follows.
 
As for ii), assume that $\mu \in M^+(\mathcal{B})$ represents $u \in \mathcal{M}$. By \eqref{eq:choquet} and \eqref{eq:def_K} we have 
\begin{align*}
(Ku,y)_Y & = \langle K_*y, u \rangle = \int_{\mathcal{B}} \langle K_*y,v \rangle \, \mathrm{d} \mu(v) \\ & = 
\int_{\mathcal{B}}  (Kv,y)_Y \, \mathrm{d} \mu(v) = (\mathcal{K}\mu,y )_Y\,, 	
\end{align*}
which holds for all $y \in Y$, thus showing that $\mathcal{K}\mu=Ku$. 
 
In order to show iii), assume that $\mu_k \weakstar \mu$ in $M(\mathcal{B})$. First notice that by \eqref{eq:def_K},
\begin{align*}
\lim_{k\rightarrow \infty}(\mathcal{K} \mu_k,y )_Y & = \lim_{k\rightarrow \infty} \int_{\mathcal{B}} ( Kv,y)_Y \, \de \mu_k(v) \\ & =\int_{\mathcal{B}} ( Kv,y) \, \de \mu(v) =(\mathcal{K} \mu,y )_Y\,,
\end{align*}
for every~$y\in Y$, where the second equality follows because the map $v \mapsto (Kv,y)_Y$ is weak* continuous, as $K$ is weak*-to-weak continuous in $\mathcal{M}$, and hence, belongs to $C(\mathcal{B})$. 
Thus,~$\mathcal{K}\mu_k \rightharpoonup \mathcal{K} \mu $ in~$Y$. Moreover, by applying \eqref{eq:def_K} twice we see that
\begin{align}\label{eq:fubinistrong}
\|\mathcal{K}\mu_k\|^2_Y & = \int_{\mathcal{B}} (Kv, \mathcal{K}\mu_k)_Y~\de \mu_k(v) \\
& =\int_{\mathcal{B}} \int_{\mathcal{B}} (Kv, Kw)_Y~ \de \mu_k(w)\de \mu_k(v)\,.
\end{align}
Note that the function $(v,w) \mapsto (Kv, Kw)_Y$ is an element of~$C(\mathcal{B} \times \mathcal{B})$. Indeed, given a sequence $(v_k,w_k) \in  \mathcal{B} \times \mathcal{B}$ such that $(v_k,w_k) \weakstar (v,w)$ for $(v,w) \in \mathcal{B} \times \mathcal{B}$ we estimate
\begin{align*}
|(Kv_k, Kw_k)_Y & - (Kv, Kw)_Y|  \\ & \leq  C\|K\|_{\mathcal{L}(\M,Y)}(\|Kw_k - K w\|_Y +  \|Kv_k - K v\|_Y)\,,
\end{align*}
where $C$ is defined at i).
As, $v_k,w_k,v,w \in \mathcal{B}$, using the weak*-to-strong continuity of $K$ on ${\rm dom}(\mathcal{G})$, cf. Assumptions \ref{ass:func1}-\ref{ass:func3}, we conclude that $(Kv_k, Kw_k)_Y \rightarrow(Kv, Kw)_Y$ proving that $(v,w) \mapsto (Kv, Kw)_Y$ is an element of~$C(\mathcal{B} \times \mathcal{B})$. Moreover, we have $\mu_k \otimes \mu_k \weakstar \mu \otimes \mu$ in $M(\mathcal{B}\times \mathcal{B})$, which, combined with \eqref{eq:fubinistrong} and the fact that $(v,w) \mapsto (Kv, Kw)_Y \in C (\mathcal{B} \times \mathcal{B})$ implies the convergence $\|\mathcal{K}\mu_k\|_Y \to \|\mathcal{K}\mu\|_Y$. Combined with the weak convergence $\mathcal{K}\mu_k \weak \K \mu$ we finally conclude~$\mathcal{K}\mu_k \rightarrow  \mathcal{K}\mu$ in~$Y$, so that $\mathcal{K}$ is weak*-to-strong continuous. 
 
Last, note that the operator~$\mathcal{K}_*$ is well-defined, linear and continuous. For any~$y\in Y$ and~$\mu \in M(\mathcal{B})$ we obtain
\begin{align*}
\sm{\mathcal{K}_* y}{\mu} & = \int_{\mathcal{B}} \langle K_* y ,v \rangle~\de \mu(v) \\ 
& = \int_{\mathcal{B}} ( K v,y )_Y~\de \mu(v)= (\mathcal{K}\mu,y)_Y \,,
\end{align*}
where again, we used \eqref{eq:def_K}. This concludes the proof of iv).
 
Assume now that $Y$ is separable and fix $\mu \in M(\mathcal{B})$. Notice that $f \colon \mathcal{B} \to Y$ defined by $f(v):=Kv$ is weakly $\mu$-measurable, since the map
$v \mapsto (Kv,y)_Y = \langle K_*y , v \rangle$ is weak* continuous and hence, $\mu$-measurable, for each $y \in Y$ fixed. As $Y$ is separable, we also have that $f$ is essentially separably valued. Therefore, Pettis' Theorem (\cite[Sec II.1, Thm 2]{diestel}) implies that $f$ is strongly measurable with respect to $\mu$. Moreover 
\[
\int_{\mathcal{B}} \nor{Kv}_Y~\de\mu(v) \leq C \nor{K}_{\mathcal{L} (\M,Y)} \nor{\mu}_{M(\mathcal{B})}<\infty \,,  
\]
where $C$ is defined at i), showing that $f$ is Bochner integrable with respect to $\mu$. Hence, by \eqref{eq:def_K} and Hille's Theorem (\cite[Sec II.1, Thm 6]{diestel}), we infer
\[
(\mathcal{K}\mu,y)_Y = \int_{\mathcal{B}} (Kv,y)_Y~\de\mu(v) = \left(\int_{\mathcal{B}} Kv~\de\mu(v),y   \right)_Y  
\]
for all $y \in Y$, concluding \eqref{eq:def_K_strong} and the proof.  \qed

\subsection{Proofs of Section \ref{sec:convergence}}
\subsubsection{Proof of Proposition \ref{prop:optcondpdapsub}} \label{sec:app_optcondpdapsub}

Let $k \geq 1$. By construction,  $\lambda^k \in \R^{N_k}$ solves
\begin{align*}
\min_{\lambda \in \R_+^{N_k}} \left \lbrack F \left(\sum^{N_k}_{i=1} \lambda_i Ku^k_i \right)+ \sum^{N_k}_{i=1} \lambda_i \right \rbrack .
\end{align*}
Deriving the first order necessary optimality conditions for this problem, we obtain that
\begin{align*}
P_k(u^k_i)  \leq 1 \quad \text{ for all } \,\,   u^k_i \in \mathcal{A}_k\,, \quad ~\sum^{N_k}_{i=1} \lambda^k_i P_k(u^k_i)= \sum^{N_k}_{i=1} \lambda^k_i\,.
\end{align*}
Since~$\lambda^k_i >0$ by construction, we deduce that~$P_k(u^k_i)=1$ for~$i=1,\dots,N_k$. In particular,
\begin{align*}
\sm{P_k}{\mu_k}= \sum^{N_k}_{i=1} \lambda^k_i P_k(u^k_i)= \sum^{N_k}_{i=1} \lambda^k_i= \|\mu_k\|_{M(\Bb)}\,.
\end{align*}
Finally, $\max_{v \in \Bb} P_k(v) \geq \max_{v \in \mathcal{A}_k} P_k(v) = 1$, concluding.

\subsubsection{Proof of Proposition \ref{prop:eststates}} 
\label{sec:app_eststates}

Note that~$\{\mu_k\}_k \subset E_{\mu_0}$ by  \eqref{eq:bound_norm_mu_k}. Arguing as in the proof of Theorem~\ref{thm:convofcond}, we get 
\begin{align*}
\ynorm{\nabla F(y_k)-\nabla F(\bar{y})} & + \|P_k-\bar{P}\|_{C(\mathcal{B})} 
 \\ & \leq L_{\mu_0} (1+C\|K\|_{\mathcal{L}(\M,Y)}) \ynorm{y_k- \bar{y}}\,,
\end{align*}
where we recall that $L_{\mu_0}>0$ is the Lipschitz constant of $\nabla F$ in $\K E_{\mu_0}$, and $C>0$ is the constant in Proposition~\ref{prop:def_K}~i).
Hence, it suffices to prove \eqref{prop:eststates:1} for~$\|y_k- \bar{y}\|_Y$.
Let~$\mathcal{N}(\bar{y})\subset Y$ and~$\theta >0$ denote the neighbourhood and constant from Assumption~\ref{ass:fast:F1}, respectively.
Recall that $\mu_k \weakstar\bar{\mu}$ along the whole sequence by \eqref{eq:whole_sequence}. By weak*-to-strong continuity of $\K$, see Proposition~\ref{prop:def_K}~iii), we then conclude $y_k \rightarrow \bar y$ in $Y$. Thus, there exist $M \in \N$ such that $y_k \in \mathcal{N}(\bar{y})$ for all $k\geq M$. Using the strong convexity of $F$ in $\mathcal{N}(\bar y)$, we estimate
\begin{align*}
j(\mu_k) & = F(y_k) + \|\mu_k\|_{M(\mathcal{B})}\\
&  \geq F(\bar{y})  +  \theta \|y_k-\bar{y}\|_Y^2/2 + (\nabla F(\bar{y}), \mathcal{K}\mu_k - \mathcal{K}\bar{\mu} )_Y + \|\mu_k\|_{M(\mathcal{B})} \\
& = j(\bar{\mu}) +  \theta \|y_k-\bar{\mu}\|_Y^2/2 + \sm{\bar{P}}{ \bar{\mu} - \mu_k} + \|\mu_k\|_{M(\mathcal{B})} - \|\bar{\mu}\| _{M(\mathcal{B})}\\
& \geq  j(\bar{\mu})+ \theta \|y_k-\bar{y}\|_Y^2/2\,,
\end{align*}
where we used~\eqref{eq:optconditionstatement} in the final inequality. Thus \eqref{prop:eststates:1} follows by rearranging the terms in the above estimate, recalling that $r_j(\mu_k)=j(\mu_k) - j(\bar{\mu})$. The final part of the statement holds since $r_j(\mu_k) \to 0$ by Theorem~\ref{thm:convofcond}.

\end{document}